\documentclass[a4paper,10pt]{amsart}
\usepackage[utf8]{inputenc}
\usepackage{amsthm}
\usepackage{amsmath}
\usepackage{bm}
\usepackage{amsfonts}
\usepackage{amssymb}
\usepackage{mathtools}
\usepackage{mathrsfs}
\usepackage{setspace}
\usepackage{xcolor}
\usepackage{textgreek}
\usepackage{todonotes}
\usepackage[a4paper, left=3cm, right=3cm, top=3cm, bottom=3cm]{geometry}
\usepackage{thmtools} 

\usepackage[pdfdisplaydoctitle,colorlinks,breaklinks,urlcolor=blue,linkcolor=blue,citecolor=blue]{hyperref} 

\newtheorem{thm}{Theorem}[section]

\newtheorem{definition}[thm]{Definition}
\newtheorem{cor}[thm]{Corollary}

\newtheorem{lem}[thm]{Lemma}

\newtheorem{prop}[thm]{Proposition}

\newcommand{\N}{\mathbb{N}}
\newcommand{\Z}{\mathbb{Z}}
\newcommand{\R}{\mathbb{R}}
\newcommand{\C}{\mathbb{C}}

\newcommand{\T}{\mathbb{T}}

\newcommand{\dvg}{\mathord{{\rm div}}\,}
\newcommand{\dvgphi}{\mathord{{\rm div}}^\phi}
\newcommand{\dvgphin}{\mathord{{\rm div}}^{\phi_{n+1}}}
\newcommand{\nablaphin}{\nabla^{\phi_{n+1}}}

\theoremstyle{remark}
\newtheorem{rmk}[thm]{Remark}

\numberwithin{equation}{section}

\title[Global existence and non-uniqueness 3D Euler transport noise]{Global existence and non-uniqueness of 3D Euler equations perturbed by transport noise}

\author[M. Hofmanov\'a]{Martina Hofmanov\'a}
  \address{Fakult\"at f\"ur Mathematik, Universit\"at Bielefeld, D-33501 Bielefeld, Germany}
  \email{hofmanova(at)math.uni-bielefeld.de}
\author[T. Lange]{Theresa Lange}
  \address{Fakult\"at f\"ur Mathematik, Universit\"at Bielefeld, D-33501 Bielefeld, Germany}
  \email{tlange(at)math.uni-bielefeld.de}
\author[U. Pappalettera]{Umberto Pappalettera}
  \address{Scuola Normale Superiore, Piazza dei Cavalieri, 7, 56126 Pisa, Italia}
  \email{umberto.pappalettera(at)sns.it}
  
\keywords{Fluid dynamics, Euler equations, transport noise, convex integration}
\date\today

\begin{document}

\begin{abstract}
We construct H\"older continuous, global-in-time probabilistically strong solutions to 3D Euler equations perturbed by Stratonovich transport noise.
Kinetic energy of the solutions can be prescribed a priori up to a stopping time, that can be chosen arbitrarily large with high probability. 
We also prove that there exist infinitely many H\"older continuous initial conditions leading to non-uniqueness of solutions to the Cauchy problem associated with the system.
Our construction relies on a flow transformation reducing the SPDE under investigation to a random PDE, and convex integration techniques introduced in the deterministic setting by De Lellis and Székelyhidi, here adapted to consider the stochastic case. 
In particular, our novel approach allows to construct probabilistically strong solutions on $[0,\infty)$ directly.
\end{abstract}

\maketitle

\section{Introduction}

In the present work we consider the Euler equations on the three-dimensional torus $\T^3$ perturbed by Brownian noise of transport type:
\begin{align} \label{eq:euler}
\begin{cases}
d u
+ (u \cdot \nabla) u \, dt 
+ \sum_{k \in I} (\sigma_k \cdot \nabla) u \bullet dB^k
+ \nabla p \, dt
= 0,
\\
\dvg u = 0,
\end{cases}
\end{align}
where $I$ is a finite set, $\{\sigma_k\}_{k \in I}$ is a collection of smooth divergence-free vector fields on the torus and $\{B^k\}_{k \in I}$ is a family of i.i.d. standard Brownian motions on a given filtered probability space $(\Omega,\mathcal{F},\{\mathcal{F}_t\}_{t \geq 0}, \mathbb{P})$, with $\{\mathcal{F}_t\}_{t \geq 0}$ assumed to be complete and right-continuous. 
The notation $\bullet \,dB^k$ denotes the Stratonovich interpretation of the stochastic integral.
The notion of solution to \eqref{eq:euler} we use throughout this paper is that of \emph{probabilistically strong, analytically weak} solutions, here recalled.

\begin{definition} \label{def:sol_euler}
Let $(\Omega,\mathcal{F},\{\mathcal{F}_t\}_{t \geq 0}, \mathbb{P})$ and $\{B^k\}_{k \in I}$ be given as above.
A progressively measurable stochastic process $(u,p):\Omega \to C_{loc}([0,\infty), L^2(\mathbb{T}^3 ,\R^3 \times \R))$ almost surely is a probabilistically strong, analytically weak solution to \eqref{eq:euler} if for every $H \in C^\infty(\T^3)$ it holds almost surely
\begin{align*}
\int_{\mathbb{T}^3} u(x,t) \cdot \nabla H(x) dx &= 0,
\quad
\forall t \in [0,\infty),
\end{align*}
and for every progressively measurable processes $H_0, \{H_k\}_{k \in I} : \Omega \to C_{loc}([0,\infty),C^\infty(\mathbb{T}^3,\R^3))$ and semimartingale $h:\Omega \to C_{loc}([0,\infty),C^\infty(\mathbb{T}^3,\R^3))$ satisfying 
\begin{align*}
dh=H_0\,dt + \sum_{k \in I} H_k \bullet dB^k,
\end{align*}
the process $t \mapsto \int_{\mathbb{T}^3} u(x,t) \cdot h(x,t) \,dx$ is a semimartingale and it holds almost surely
\begin{align*}
d\int_{\mathbb{T}^3} u \cdot h 
&=
\int_{\mathbb{T}^3} u \cdot \left( H_0 + (u \cdot \nabla) h \right) dt
+
\int_{\mathbb{T}^3} p \,\dvg h \,dt
+
\sum_{k \in I} \int_{\mathbb{T}^3} u \cdot \left( H_k + (\sigma_k \cdot \nabla) h \right) \bullet dB^k.
\end{align*}
\end{definition}
Notice that $L^2$ integrability in space of the velocity field $u$ is required in order to make sense of the nonlinear term $(u \cdot \nabla) u$ as a distribution.
(Local) continuity of trajectories allows to uniquely identify the initial data $u_0 = u(\cdot,0)$ and $p_0 = p(\cdot,0)$ as random variables taking values in $L^2(\T^3,\R^3)$ and $L^2(\T^3,\R)$ respectively, and consider the Cauchy problem associated to \eqref{eq:euler}.
For constant-in-time $h \in C^{\infty}(\T^3,\R^3)$, corresponding to $H_0=H_k=0$, this definition boils down to a more standard notion of distributional solutions to Euler equations, as for instance that used in \cite{BrFlMa16}; and in fact the two definitions are equivalent, which can be shown using mollification and It\=o formula as done in Appendix A of \cite{DeHoVo16}. 
We decided to adopt this definition because it copes better with the flow transformation we are going to introduce in the next section, see \autoref{def:sol_random_euler} and \autoref{prop:equivalence}.

In recent years, there have been several works aimed to justify the addition of transport noise in Euler equations, either by variational principles \cite{Ho15}, homogenization techniques \cite{CoGoHo17}, conservation laws \cite{DrHo20} or Wong-Zakai results \cite{FlPa21}.
More generally, a noise of transport type appears naturally in fluids when considering the effect that turbulence may have on the slow-varying, large-scale structures of the fluid itself, see \cite{FlPa22,DePa22+}.

Moreover, regularization by transport noise has been shown to hold in many different instances, as in transport equation \cite{FlGuPr10}, point vortices dynamics \cite{FlGuPr11}, Vlasov-Poisson equations \cite{DeFlVi14}, Navier-Stokes equations \cite{FlLu21} and other models \cite{FlGaLu21c,Lu21++,Ag22+,La22+}.
This perhaps suggests that a noise of transport type could help in proving well-posedness of Euler equations, or at least some partial result toward this direction.
However, this simply turns out not to be the case.
In this paper we are going to prove the following: 
\begin{itemize}
\item[\emph{i})]
For every constants $\varkappa \in (0,1)$, $T\in (0,\infty)$, and for every suitable \emph{energy profile} $e:[0,\infty) \to (0,\infty)$ given a priori, there exist a stopping time $\mathfrak{t}$ satisfying $\mathbb{P}\{ \mathfrak{t} \geq T\} \geq \varkappa$ and a H\"older continuous, probabilistically strong, analytically weak solution to \eqref{eq:euler} having kinetic energy equal to $e$ up to time $\mathfrak{t}$, namely
\begin{align*}
\int_{\mathbb{T}^3} |u(x,t)|^2 dx = e(t), 
\quad \forall t \in [0,\mathfrak{t}];
\end{align*}
\item[  \emph{ii})]
There exist infinitely many H\"older continuous divergence-free initial data $u(\cdot,0)$ such that the Cauchy problem associated with \eqref{eq:euler} admits non-uniqueness in law.
\end{itemize}
More detailed statements of the previous results are given in \autoref{ssec:main_results} below. 

It is worth mentioning the important progress in the existence theory of stochastic Euler equations implied by these results.
Indeed, in three dimensions only local-in-time well-posedness of regular solutions was known so far, see \cite{CrFlHo19} for transport noise and  \cite{Ki09,GHVi14} for other stochastic perturbations.
On the other hand, our solutions are global-in-time, although we can not prescribe the initial datum as in the cited literature. 
From this point of view, the global existence of probabilistically strong solutions for given initial data is still an open problem.
The more recent \cite{HoZhZh22} and \cite{HoZhZh22b+} provide global existence of analytically weak and probabilistically strong solution, but they are restricted to additive noise perturbations. 

Those in the present paper are the first results of this kind dealing with Euler equations perturbed by transport noise.

\subsection{Flow transformation and reformulation of the problem}

In order to construct solutions to \eqref{eq:euler}, it is convenient to rewrite the equation in a different form, so to ``remove" the transport term $\sum_{k \in I} (\sigma_k \cdot \nabla) u \bullet dB^k$ from the equation, at least apparently.

Therefore, we can consider the flow generated by the noise, given by 
\begin{align} \label{eq:flow}
\phi(x,t) = x + \sum_{k \in I}\int_0^t \sigma_k(\phi(x,s)) \bullet dB^k(s).
\end{align}
By our assumptions on the coefficients $\{\sigma_k\}_{k \in I}$, there exists a unique progressively measurable, probabilistically strong solution $\phi$ to \eqref{eq:flow} such that $\phi(\cdot,t)$ takes values in the class of measure preserving $C^{\infty}$-diffeomorphisms of the torus for every $t \in [0,\infty)$, see \cite{Ku97}.
Applying the change of variables\footnote{With a little abuse of notation, hereafter we denote $\phi^{-1}$ the inverse of the flow with respect to its first argument: $\phi^{-1}(y,t) \coloneqq (\phi(\cdot,t))^{-1}(y)$ for every fixed $y \in \mathbb{T}^3$ and $t \in [0,\infty)$.} 
\begin{align}
v(x,t) &= u(\phi(x,t),t), \label{eq:change_variables}
\quad
q(x,t) = p(\phi(x,t),t),
\\ \label{eq:change_variables_inv}
u(y,t) &= v(\phi^{-1}(y,t),t),
\quad
p(y,t) = q(\phi^{-1}(y,t),t),
\end{align}
we can formally rewrite the SPDE \eqref{eq:euler} as a random PDE:
\begin{align} \label{eq:random_euler}
\begin{cases}
\partial_t v 
+ \dvgphi (v \otimes v)
+ \nabla^\phi q
= 0,
\\
\dvgphi v = 0 ,
\end{cases}
\end{align}
where the symbols $\dvgphi$, $\nabla^\phi$ are abbreviations for the space-time dependent differential operators
\begin{align*}
\dvgphi v = [\dvg(v \circ \phi^{-1})] \circ \phi,
\quad
\nabla^\phi q = [\nabla(q \circ \phi^{-1})] \circ \phi.
\end{align*}
It is worth mentioning at this point that, as already observed in \cite{FlGuPr10,FlGuPr11}, a space-independent noise cannot regularize Euler equations since this would imply $\mbox{div}^\phi=\mbox{div}$ and $\nabla^\phi=\nabla$, and therefore the stochastic system would be equivalent to the deterministic one (which notably admits non-unique solutions). 
However, in most of the examples mentioned above a genuinely space-dependent noise does in fact improve well-posedness, and this is why our analysis is not trivial. 

For the system \eqref{eq:random_euler} we adopt the following notion of solution.
\begin{definition} \label{def:sol_random_euler}
Let $(\Omega,\mathcal{F},\{\mathcal{F}_t\}_{t \geq 0}, \mathbb{P})$ and $\{B^k\}_{k \in I}$ be given as above, and let $\phi$ be the unique stochastic flow of measure preserving diffeomorphisms given by \eqref{eq:flow}.
A progressively measurable stochastic process $(v,q):\Omega \to C_{loc}([0,\infty), L^2(\mathbb{T}^3 ,\R^3 \times \R))$ almost surely is a probabilistically strong, analytically weak solution to \eqref{eq:random_euler} if for every $H \in C^\infty(\T^3)$ it holds almost surely
\begin{align*}
\int_{\mathbb{T}^3} v(x,t) \cdot \nabla^\phi H(x,t) dx &= 0,
\quad
\forall t \in [0,\infty),
\end{align*}
and for every progressively measurable processes $H_0, \{H_k\}_{k \in I} : \Omega \to C_{loc}([0,\infty),C^\infty(\mathbb{T}^3,\R^3))$ and semimartingale $h:\Omega \to C_{loc}([0,\infty),C^\infty(\mathbb{T}^3,\R^3))$ satisfying 
\begin{align*}
dh=H_0\,dt + \sum_{k \in I} H_k \bullet dB^k,
\end{align*}
the process $t \mapsto \int_{\mathbb{T}^3} v(x,t) \cdot h(x,t) \,dx$ is a semimartingale and it holds almost surely
\begin{align*}
d\int_{\mathbb{T}^3} v \cdot h 
&=
\int_{\mathbb{T}^3} v \cdot \left( H_0 + (v \cdot \nabla^\phi) h \right) dt
+
\int_{\mathbb{T}^3} q \,\dvgphi h \,dt
+
\sum_{k \in I} \int_{\mathbb{T}^3} v \cdot  H_k \bullet dB^k.
\end{align*}
\end{definition}

The equivalence of systems \eqref{eq:euler} and \eqref{eq:random_euler} is shown in \autoref{prop:equivalence}. More precisely, we prove that a process $(u,p)$ is a solution to \eqref{eq:euler} in the sense of \autoref{def:sol_euler} if and only if $(v,q)$ given by \eqref{eq:change_variables} is a solution to \eqref{eq:random_euler} in the sense of \autoref{def:sol_random_euler}.
As a consequence, applying the inverse transformation \eqref{eq:change_variables_inv} to a solution $(v,q)$ of the modified equation \eqref{eq:random_euler} produces a solution $(u,p)$ of the original equation \eqref{eq:euler}. 
Like with \eqref{eq:euler}, by continuity of trajectories one can uniquely determine the initial conditions $v_0=v(\cdot,0)$ and $q_0=q(\cdot,0)$, and consider the Cauchy problem associated to \eqref{eq:random_euler}.
Also, since $\phi(\cdot,t)$ is measure preserving, it holds
\begin{align*}
\int_{\mathbb{T}^3} |v(x,t)|^2 dx
=
\int_{\mathbb{T}^3} |u(x,t)|^2 dx
\end{align*}
almost surely for every $t \in [0,\infty)$, and therefore we can state and prove our results in the framework of solutions to the random PDE \eqref{eq:random_euler}.

\subsection{Main results} \label{ssec:main_results}
The main result of the present paper is the following: 

\begin{thm}[Global strong existence] \label{thm:strong_ex}
Assume $\sigma_k$ smooth and divergence-free for every $k \in I$, $I$ finite, and define $\phi$ by \eqref{eq:flow}.
Then there exist $\vartheta>0$ and, for any given $\varkappa \in (0,1)$ and $T\in (0,\infty)$, a stopping time $\mathfrak{t}$ satisfying $\mathbb{P}\{ \mathfrak{t} \geq T\} \geq \varkappa$ with the following property. 

For every function $e$ on $[0,\infty)$ satisfying $\underline{e} \coloneqq \inf_{t \in [0,\infty)} e(t)>0$ and $\overline{e} \coloneqq \|e\|_{C^2_t}<\infty$, there exists a global probabilistically strong, analytically weak solution $(v,q)$ of \eqref{eq:random_euler} of class
\begin{align*}
v: \Omega \to C^{\vartheta}_{loc}([0,\infty), C(\mathbb{T}^3,\R^3)) \cap C_{loc}([0,\infty), C^\vartheta(\mathbb{T}^3,\R^3)),
\\
q: \Omega \to C^{2\vartheta}_{loc}([0,\infty), C(\mathbb{T}^3,\R)) \cap C_{loc}([0,\infty), C^{2\vartheta}(\mathbb{T}^3,\R)),
\end{align*}
almost surely, such that $v$ satisfies the almost sure identity
\begin{align*}
\int_{\mathbb{T}^3} |v(x,t)|^2 dx &= e(t),
\quad
\forall t \in [0,\mathfrak{t}].
\end{align*} 
\end{thm}

As far as we know, this is the first result proving global existence of non-trivial, analytically weak, probabilistically strong solutions to Euler equations perturbed with transport noise. 
The only other example in the stochastic setting we are aware of is the very recent \cite{HoZhZh22b+}, dealing with the case of additive noise. 
We point out the fact that, in terms of space regularity of the solutions constructed, in the present paper we are able to produce even better solutions than those in \cite{HoZhZh22b+} (H\"older continuous versus $H^\vartheta$ Sobolev regular).
We also mention \cite{GHVi14}, where the authors prove that a suitable linear multiplicative noise yields global existence of smooth solutions with large probability.  

As a matter of fact, the proof of \autoref{thm:strong_ex} relies on a convex integration scheme, mostly inspired by the works \cite{DLS13,DLS14} on deterministic Euler equations.
We shall build smooth approximate solutions $(v_n,q_n,\phi_n)$, $n \in \N$ to the random Euler system \eqref{eq:random_euler}. 
We include a smooth approximating flow $\phi_n$ in the solution to take care of the lack of smoothness with respect to time of the Brownian flow $\phi$. 
To measure how far a given $(v_n,q_n,\phi_n)$ is from being a solution of \eqref{eq:random_euler} it is customary to introduce a system of differential equations called Euler-Reynolds system, that in our setting takes the form
\begin{align} \label{eq:random_euler-reynolds}
\partial_t v
+ \dvgphi (v \otimes v)
+ \nabla^\phi q
= \dvgphi \mathring{R}
\end{align}
where the \emph{Reynolds stress} $\mathring{R}$ takes values in the space of $3 \times 3$ symmetric traceless matrices.
In particular, we say that a quadruple $(v_n,q_n,\phi_n,\mathring{R}_n)$ is a solution of the Euler-Reynold system \eqref{eq:random_euler-reynolds} if it is progressively measurable with respect to the filtration $\{\mathcal{F}_t\}_{t \geq 0}$ and \eqref{eq:random_euler-reynolds} holds in strong analytical sense.

Notice that we are \emph{not} imposing any divergence-free condition $\dvg^{\phi_n} v_n = 0$ on the solution. The incompressibility condition on the limit will be restored prescribing the decay of some spatial Besov norm of $\dvg^{\phi_n} v_n$ along the iteration.

However, since for technical reasons that will be clear later we are actually going to construct smooth approximate solutions $(v_n,q_n,\phi_n,\mathring{R}_n)$ also for negative times $t<0$, we require solutions also satisfy \eqref{eq:random_euler-reynolds} in analytically strong sense for all times $t \in \R$. For negative times the flow $\phi_n$ will be defined simply as the identity on the torus, corresponding to absence of noise for $t<0$, and the process $(v_n,q_n,\mathring{R}_n)$ will be deterministic (thus preserving progressive measurability with respect to $\{\mathcal{F}_t\}_{\geq 0}$, extended to negative times identically equal to $\mathcal{F}_0$). 
In addition, we shall always work with solutions $v_n$ with zero spatial average: $\int_{\T^3} v_n = 0$ for every $t \in \R$. 
More details will be given in \autoref{sec:prelim}.

Then, given a sequence $(v_n,q_n,\phi_n,\mathring{R}_n)$ of solutions to the random Euler-Reynolds system \eqref{eq:random_euler-reynolds}, we intend to exhibit a solution $(v,q)$ of \eqref{eq:random_euler} showing the convergences, with respect to suitable topologies:
\begin{align*}
v_n \to v,
\quad
q_n \to q,
\quad
\phi_n \to \phi,
\quad
\mathring{R}_n \to 0,
\quad
\dvg^{\phi_n} v_n \to 0.
\end{align*} 

Let us move to the non-uniqueness issue. \autoref{thm:strong_ex} provides the existence of a solution to \eqref{eq:random_euler} with a prescribed energy profile up to time $\mathfrak{t}$. 
We do not claim this solution to be unique among solutions with the same energy.
However, we can carry on the convex integration scheme outlined above in such a way that solutions given by \autoref{thm:strong_ex} satisfy the following property. 

\begin{thm} \label{thm:non-uniq}
Let $e_1$, $e_2$ be energy profiles satisfying the hypotheses of \autoref{thm:strong_ex} with the same $\underline{e}$, $\overline{e}$, and such that $e_1(t)=e_2(t)$ for every $t\in [0,T/2]$.
Then the two global probabilistically strong solutions $(v_1,q_1)$ and $(v_2,q_2)$ given by \autoref{thm:strong_ex} are such that
 $v_1(x,t)=v_2(x,t)$ for every $x \in \mathbb{T}^3$ and $t \in [0,T/2]$.
\end{thm}

As a consequence, we disprove uniqueness-in-law for solutions $(v,q)$ (with the aforementioned regularity) of the Cauchy problem associated with \eqref{eq:random_euler}. Indeed, we can apply the previous theorem with $e_1, e_2$ such that $e_1(t)>e_2(t)$ for every $t \in (T/2,T]$, yielding $\int_{\mathbb{T}^3} |v_1(x,t)|^2 dx>\int_{\mathbb{T}^3} |v_2(x,t)|^2 dx$ for every $t \in (T/2,T]$ with probability greater than $\varkappa>0$. In particular, for every initial kinetic energy $e(0)>0$ there exists at least an initial datum $v(\cdot,0) :\Omega \to C^\vartheta(\mathbb{T}^3,\R^3)$ with $\int_{\mathbb{T}^3} |v(x,0)|^2 dx = e(0)$ almost surely such that the Cauchy problem associated with \eqref{eq:random_euler} is ill-posed.

Let us explain why \autoref{thm:non-uniq} holds. 
We shall build iteratively (\autoref{prop:it}) a sequence of approximating solutions $(v_n,q_n,\phi_n,\mathring{R}_n)$, $n \in \N$ in such a way that the approximating solution at level $n+1$, evaluated at time $t \in [0,\infty)$, only depends on the restriction to times less or equal than $t$ of the energy profile, the flow \eqref{eq:flow}, and the approximating solutions at levels $k \leq n$.
In formulae, for fixed $x \in \T^3$ and $t \in [0,\infty)$, the quantities $v_{n+1}(x,t)$, $q_{n+1}(x,t)$, $\phi_{n+1}(x,t)$, and $\mathring{R}_{n+1}(x,t)$ only depend upon $e(s)$, $\phi(y,s)$, $v_k(y,s)$, $q_k(y,s)$, $\phi_k(y,s)$, $\mathring{R}_k(y,s)$, for arbitrary $s \leq t$, $k \leq n$, and $y \in \T^3$.
We point out that doing so is almost forced in the stochastic setting, in order to preserve progressive measurability of solutions constructed via the convex integration scheme.
Moreover, we always start the iteration with the same quadruple $(v_0,q_0,\phi_0,\mathring{R}_0)=(0,0,\phi_0,0)$, where $\phi_0$ depends only on $\phi$, $\underline{e}$, $\overline{e}$.

This means, in the setting of \autoref{thm:non-uniq} above, that the two solutions $v_1$ and $v_2$, evaluated at any time $t \in [0,T/2]$, are obtained as the limit of the same sequence (since $e_1(t)=e_2(t)$ for every $t \in [0,T/2]$); therefore, the two solutions coincide. 
Taking $t=0$, we deduce in addition that the initial conditions $v_1(\cdot,0)$ and $v_2(\cdot,0)$ coincide as well.
Since the two solutions can not have the same law for times $t \in (T/2,T]$ (they have different energy with positive probability), this gives non-uniqueness in law for the Cauchy problem.

The property above can be checked looking at the construction outlined in \autoref{sec:convex}.

Finally,  notice that the convenient $C^\infty$-regularity assumption on the coefficients $\{\sigma_k\}_{k \in I}$ can be relaxed to $\sigma_k$ of class $C^\kappa$ for some $\kappa$ sufficiently large, and also we can replace the Brownian motions $\{B^k\}_{k \in I}$ with more general paths, for instance fractional Brownian motions with Hurst parameter $H>1/4$, reinterpreting the equation and the stochastic integral therein in the proper way (cf. \autoref{rmk:H>1/4} below). We omit details for the sake of simplicity. 

\subsection{Bibliographic discussion}
Literature on Euler equations is extremely vast and impossible to sum up exhaustively here.
In two spatial dimensions, well-posedness of (deterministic) Euler equations with initial vorticity (i.e. the curl of the velocity field) in $L^\infty(\T^2)$ is known since the work of Yudovich \cite{Yu63}. 
Uniqueness of solutions may fail in non-regular settings, as proved for instance in \cite{Sc93,Sh97,DLS09}; actually, this phenomenon is generic, in the sense that non-uniqueness holds for every initial datum within the class of square-integrable solutions \cite{Wi11} or in the case of vorticities given by the sum of delta Dirac masses \cite{GrPa22}.
As for the stochastic case, namely with noise of transport type put into the dynamics, well-posedness was proved in \cite{BrFlMa16} for initial vorticity in $L^\infty(\T^2)$, and in \cite{BrMa19} existence of solutions with vorticity in $H^{-1}(\T^2)$ with definite sign was proved. 

In dimension three, one has only local-in-time well-posedness of regular solutions, both in the deterministic case (see \cite{BeMa02} and reference therein) and with transport noise \cite{CrFlHo19}. 
Other stochastic perturbations of 3D Euler equations have been considered in \cite{Ki09,GHVi14}.
Less regular deterministic solutions are not unique by the aforementioned results in two dimensions. Moreover, non-uniqueness holds true in a class of relatively regular solutions (almost $1/3$-H\"older continuous) as a consequence of the series of papers \cite{DLS09,DLS13,DLS14,BuDLIsSz15,Is18} finally solving the Onsager's conjecture.
More recently, in \cite{HoZhZh22} the authors proved ill-posedness (i.e. global existence and non-uniqueness of generalized solutions) even when additive noise is put into the equations. Their result has been successively refined in \cite{HoZhZh22b+} to analytically weak solutions.

\subsection{Main novelties}
It is worth comparing the present paper with previous works using convex integration techniques, both deterministic and stochastic.
The strategy of the proof of \autoref{prop:it} is similar to that in \cite{DLS14}, which is the first result using convex integration techniques to produce H\"older solutions to (deterministic) Euler equations.
Of course, the construction presented here needs several non-trivial adjustments with respect to that in \cite{DLS14}, due to the operators $\dvgphi$, $\nabla^\phi$ present in \eqref{eq:random_euler}:
\begin{enumerate}
\item
They depend also on time, thus compromising some geometric properties of the Euler equations; in addition, the operators $\dvgphi$, $\nabla^\phi$ do not commute with space-time mollifications, and this produces an additional mollification error in the Reynold stress along the iteration;
\item
The flow $\phi$ is non-smooth in time, thus requiring mollification of the flow and \emph{ad hoc} control over the error made in doing so;
\item
Finally, we cannot impose the divergence-free condition on the approximation $v_n$, since we would lose control over the time derivative of the compressibility corrector and the associated estimates on the compressibility error in the Reynold stress. Thus, we just aim at reducing iteratively the size of $\dvg^{\phi_n} v_n$, without requiring it to be zero at every step of the iteration.
\end{enumerate}

In order to deal with the flow $\phi$, and in particular with its growth for large times, in the present paper we introduce a sequence of stopping times $\mathfrak{t}_L \to \infty$ as $L \to \infty$.
However, we point out that contrary to other works introducing stopping time in the convex integration scheme \cite{HoZhZh22,HoZhZh21+,HoZhZh21b+,Ya20,Ya21,ReSc21+,LuZh22}, here we do not produce local-in-time solutions, successively extended to global solutions à la Leray (which is practically impossible due to lack of compactness) or gluing together another convex integration solution (which we cannot do since we are not able to solve the Cauchy problem associated to Euler equations for every initial datum $v_0$). 
Rather, we develop a genuinely stochastic convex integration procedure and construct directly global solutions; we use the stopping times only to control the growth of solutions in suitable H\"older spaces.
This is more similar in the spirit to the recent works \cite{HoZhZh22b+,ChDoZh22}, where global solutions are produced by controlling the growth in expectation of the norms of the solution.

\subsection{Frequently used notation}
We adopt the following notation throughout the paper.
H\"older spaces of functions of time $t \in \R$ or space $x \in \mathbb{T}^3$ with values in some Banach space $E$ are denoted respectively $C^\vartheta_t E$ and $C^\vartheta_x E$, $\vartheta>0$.
When confusion may not arise we simply denote $C^\vartheta_t=C^\vartheta_t E$ and $C^\vartheta_x=C^\vartheta_x E$.
By convention $C^1_t$ and $C^1_x$ denote the spaces of Lipschitz functions (and not the space of continuously differentiable ones). A similar convention holds for $C^n_t$ and $C^n_x$, $n \in \N$, $n > 1$.
H\"older seminorms are denoted $[\,\,\cdot\,\,]_{C^\vartheta_t}$ and $[\,\,\cdot\,\,]_{C^\vartheta_x}$, and H\"older norms $\|\cdot\|_{C^\vartheta_t}$ and $\|\cdot\|_{C^\vartheta_x}$, respectively.

For a stopping time $\mathfrak{t}$, we denote $C^\vartheta_{\leq \mathfrak{t}}$ the space of H\"older functions $f:(-\infty,\mathfrak{t}] \to E$, endowed with associated seminorm $[\,\,\cdot\,\,]_{C^\vartheta_{\leq \mathfrak{t}}}$ and norm $\|\cdot\|_{C^\vartheta_{\leq \mathfrak{t}}}$.
H\"older (semi)norms in $C^\vartheta_{\leq \mathfrak{t}}$ are defined for a generic $f:\R \to E$ upon restriction to $(-\infty,\mathfrak{t}]$. 

For functions of space and time we denote
$\| \cdot \|_{C^\vartheta_{t,x}} \coloneqq \| \cdot \|_{C^{\vartheta}_t C_x} +
\| \cdot \|_{C_t C^{\vartheta}_x}$
and 
$\| \cdot \|_{C^\vartheta_{\leq \mathfrak{t},x}} \coloneqq \| \cdot \|_{C^{\vartheta}_{\leq \mathfrak{t}} C_x} +
\| \cdot \|_{C_{\leq \mathfrak{t}}C^{\vartheta}_x}$ the space-time $\vartheta$-H\"older norms, $\vartheta \in (0,1]$.

We shall sometimes need to work with Besov spaces $B^{\alpha}_{p,q} \coloneqq B^{\alpha}_{p,q}(\T^3)$, 
$\alpha \in \R$, $p,q \in [1,\infty]$ defined as the subset of distributions $u \in \mathcal{S}'(\T^3)$ such that
\begin{align*} 
\| u \|_{B^\alpha_{p,q}}
:=
\left\| \left(2^{j \alpha} \|\Delta_j u \|_{L^p(\T^3)}\right)_{j \geq -1} \right\|_{\ell^q}
< \infty.
\end{align*}
In the previous line, $\{\Delta_j\}_{j \geq -1}$ denotes the Littlewood–Paley blocks corresponding to a dyadic partition of unity, as used for instance in \cite{HoZhZh21c+} or \cite{MoWe17}. 
For $p,q<\infty$ the Besov space $B^\alpha_{p,q}$ is separable and coincides with the closure of $C^\infty(\T^3)$ with respect to $\|\cdot\|_{B^\alpha_{p,q}}$.
Notably $B^{\alpha}_{p,q} = (B^{-\alpha}_{p',q'})^*$, with equivalence of norms, when $1/p+1/p'=1/q+1/q'=1$, and $B^\alpha_{\infty,\infty}=C^\alpha_x$ for every non-integer $\alpha>0$, again with equivalence of norms.
Some useful lemmas on H\"older and Besov spaces are collected in \autoref{sec:schauder}. 

\section*{Acknowledgements}
This project has received funding from the European Research Council (ERC) under the European Union’s Horizon 2020 research and innovation programme (grant agreement No. 949981).

\section{Preliminaries and main iterative proposition} \label{sec:prelim}
In this section we give more details about the strategy of the proof of \autoref{thm:strong_ex}.
We have been strongly inspired by the construction of \cite{DLS14} on deterministic Euler equations.

As already mentioned in the previous section, we need to introduce a sequence $\phi_n$ of smooth approximations of the flow $\phi$, successively extended to negative times via the formula $\phi_n(t)=\phi_n(0)=Id_{\,\T^3}$ for every $t<0$. 
This extension is only for technical reasons, and will be only needed in next \autoref{sec:convex}.  
The smooth approximations are obtained via mollification of the noise, and the procedure is described in details in next \autoref{ssec:moll}. The idea is simply to mollify the driving Brownian motion $B$ with a mollification parameter $\varsigma_n$, and then define the approximating flow $\phi_n$ pathwise as a Riemann-Stieltjes integral with respect to the mollified $B$.
In \autoref{ssec:RoughP} and \autoref{ssec:stop}, making use of some result from Rough Paths theory, we provide estimates on the distance (in suitable spaces) between $\phi_n$ and $\phi$, in terms of the mollification parameter $\varsigma_n$ (here $\phi$ is extended as well to negative times as the identity on the torus).
With these estimates in hand, we can introduce a sequence of stopping times $\mathfrak{t}_L$, $L \in \N$ to control the approximating flow uniformly in $\omega \in \Omega$ and localize the problem, see \autoref{ssec:localize}.

After this preliminary preparation, we state our main \autoref{prop:it} in \autoref{ssec:main_prop}.
The idea of the proposition is the same as in \cite{DLS14}, and it consists in collecting a series of iterative estimates that, if verified, allow to show convergence of the sequence $(v_n,q_n,\phi_n,\mathring{R}_n)$ as $n \to \infty$ towards a limit $(v,q,\phi,0)$ solution of \eqref{eq:random_euler} with desired H\"older regularity.
This is done by controlling, iteratively, the norms in $C_{\leq \mathfrak{t}_L} C_x$, $C_{\leq \mathfrak{t}_L} C^1_x$ and $C^1_{\leq \mathfrak{t}_L,x}$, and using interpolation.

\subsection{Mollification of the noise}  \label{ssec:moll}
We intend to work pathwise, and thus we assume for simplicity that every realisation of the $\R^{|I|}$-valued driving noise $B=(B^k)_{k \in I}$ has $C^\alpha_{loc}$ time regularity, $\alpha \in (1/3,1/2)$.
Since $B$ is of class $C^\alpha_{loc}$ in time, the associated flow $\phi$ is expected to have the same time regularity. Unfortunately, for our purposes this is not sufficient, as we shall need to take time derivatives of the flow.
Therefore, we introduce a sequence of smooth approximations $\{\phi_n\}_{n \in \N}$ as follows.

Fix $n \in \N$ and let $\varsigma_n>0$ be a constant to be properly chosen below, in accordance with other parameters of the convex integration procedure.
For the time being, we shall assume $\varsigma_n$  monotonically decreasing in $n$ and $(n+1)\varsigma_n^{\alpha-\beta} \to 0$ for every $\beta \in (0,\alpha)$. 
 
Next, let $\theta:\R \to \R$ be a smooth mollifier with support contained in $(0,1)$, and define for $t \in \R$:
\begin{align*}
\theta_n(t) \coloneqq\varsigma_n^{-1} \theta(t\varsigma_n^{-1}),
\quad
B_n(t) \coloneqq (B \ast \theta_n) (t)
=
\int_\R B(t-s) \theta_n(s) ds,
\end{align*}
with the convention that $B(t-s)=B(0)=0$ whenever $t-s$ is negative, and the second equation is intended to hold component-wise, in particular $(B_n)^k = B^k_n = (B^k \ast \theta_n)$ for every $k \in I$. 
Notice that $B_n$ is smooth at every time $t \in \R$ and it is identically zero for negatives times, being the mollification one-sided. 
Finally, define $\phi_n$ as the unique solution of the integral equation
\begin{align} \label{eq:def_phi_n}
\phi_n(x,t) \coloneqq x + \sum_{k \in I}\int_0^t \sigma_k(\phi_n(x,s)) \,dB^k_n(s),
\quad
x \in \T^3, \,t \in \R,
\end{align}
where the integral is understood pathwise as a Riemann-Stieltjes integral, and notice that $\phi_n(x,t)= \phi_n(x,0) = x$ for $t<0$.
We extend the flow $\phi$ defined by \eqref{eq:flow} to $t<0$ similarly, imposing $\phi(x,t)= \phi(x,0) = x$ (the resulting flow is locally $C^\alpha$ in time).
With probability one, for every fixed $t \in \R$ the map $\phi_n(\cdot,t)$ is measure preserving since $\sigma_k$ is divergence-free and $dB^k_n = \dot{B}^k_n ds$ is space-independent, as a consequence of Liouville Theorem.

Let us also mention that $\phi_n$ is indeed an approximation of $\phi$ by Wong-Zakai Theorem, and the rate of convergence is made explicit in the next \autoref{lem:flow}.

\subsection{Rough paths preliminaries} \label{ssec:RoughP}
We need to introduce some auxiliary concepts from the theory of Rough Paths. We keep the background at minimum, referring to the books \cite{FrVi10} and \cite{FrHa20} when needed.


Let $\alpha \in (1/3,1/2)$, and for any given stopping time $\mathfrak{t}$ let $\mathscr{C}^\alpha_{g,\mathfrak{t}} \coloneqq \mathscr{C}^\alpha_g([0,\mathfrak{t}],\R^{|I|})$ be the space of $\alpha$-H\"older geometric rough paths $\mathbf{X}=(X,\mathbb{X})$ endowed with the metric
\begin{align*}
\varrho_{\alpha,\mathfrak{t}}(\mathbf{X}, \mathbf{Y})
\coloneqq
\sup_{0\leq s<t \leq \mathfrak{t}} \frac{\|X_{s,t}-Y_{s,t}\|}{|t-s|^\alpha} 
+ 
\sup_{0\leq s<t \leq \mathfrak{t}} \frac{\|\mathbb{X}_{s,t}-\mathbb{Y}_{s,t}\|}{|t-s|^{2\alpha}}.
\end{align*} 

For any smooth path $X:\R \to \R^{|I|}$, let us define the step-$2$ Lyons lift of (the restriction on $[0,\mathfrak{t}]$ of) $X$ as the rough path $\mathbf{X}=(X,\mathbb{X}) \in \mathscr{C}^\alpha_{g,\mathfrak{t}}$ given by the Riemann-Stieltjes integral
\begin{align*}
\mathbb{X}_{s,t} \coloneqq \int_s^t X_{s,r} \otimes dX_r,
\quad
s,t \in [0,\mathfrak{t}].
\end{align*}

It is well known that the Stratonovich lift of $B$ is almost surely an $\alpha$-H\"older geometric rough path, which we denote with the symbol $\mathbf{B}=(B,\mathbb{B})$.
For any $n \in \N$, let us denote $\mathbf{B}_n=(B_n,\mathbb{B}_n) \in \mathscr{C}^\alpha_{g,\mathfrak{t}}$ the step-$2$ Lyons lift of $B_n$.

\begin{lem} \label{lem:lyons}
For every $K > 0$ and $\beta \in (0,\alpha)$, $\alpha \in (1/3,1/2)$ there exists a stopping time $\mathfrak{s}$ with $\mathfrak{s} \leq K$ almost surely and for every $n,m \in \N$, $n \leq m$:
\begin{gather*}
\varrho_{\beta,\mathfrak{s}}(\mathbf{B}_m,\mathbf{B}_n)
\leq
K (n+1) \varsigma_n^{\alpha-\beta},
\\
\varrho_{\beta,\mathfrak{s}}(\mathbf{B},\mathbf{B}_n)
\leq
K (n+1) \varsigma_n^{\alpha-\beta}.
\end{gather*}
Moreover, $\mathfrak{s}$ can be chosen so that $\mathfrak{s} \to \infty$ almost surely as $K \to \infty$.
\end{lem}
\begin{proof}
First of all, notice that up to replacing $K$ with $K/2$, the first inequality descends from the second one and triangle inequality, thus we only focus on the latter.

We apply \cite[Corollary 15.32]{FrVi10} on the time interval $[0,T]$, $T<\infty$ arbitrary, with $\mathbf{X}=\mathbf{B}_m$, $\mathbf{Y}=\mathbf{B}_n$ and $\varepsilon$ given by (cf. Remark 15.33 therein)
\begin{align*}
\varepsilon^\frac{2\alpha}{\alpha-\beta} 
&= 
\max_{i,j \in I}\sup_{s,t \in [0,T]} \mathbb{E}\left[ (B_m^i(t)-B_n^i(t))(B_m^j(s)-B_n^j(s))\right]
\\
&\leq
C \varsigma_n^{2\alpha} \mathbb{E}\left[ \|B\|^2_{C^\alpha_{\leq T}}\right]
\leq
C \varsigma_n^{2\alpha}.
\end{align*}
We point out that the expectation in the line above is finite and bounded by some constant $C$ depending only on $T$ and $\alpha$, see \cite[Proposition 3.5]{FrHa20}.
Thus, by the aforementioned \cite[Corollary 15.32]{FrVi10} and this choice of $\varepsilon$ there exists $C=C(T,\beta,\alpha)$ such that for every finite $q \geq 1$ it holds\footnote{Here $\varrho_{\beta,T}$ denotes the $\beta$-H\"older rough paths distance in $\mathscr{C}^\beta_g([0,T],\R^{|I|})$, i.e. with constant stopping time $\mathfrak{t}=T$. It is important to work in this space, since we cannot apply Corollary 15.32 of \cite{FrVi10} to the stopped processes $\mathbf{B}_n(\cdot \wedge \mathfrak{t})$.}
\begin{align*}
\left\| \varrho_{\beta,T}(\mathbf{B}_m,\mathbf{B}_n) \right\|_{L^q(\Omega)}
\leq
C q^{1/2} \varsigma_n^{\alpha-\beta}.
\end{align*}
Therefore the sequence $\{\mathbf{B}_n\}_{n\in \N}$ is Cauchy in $L^q(\Omega,\mathscr{C}^\beta_{g,T})$, and since the space $\mathscr{C}^\beta_{g,T}$ is complete with respect to the distance $\varrho_{\beta,T}$ by \cite[Theorem 8.13]{FrVi10}, there exists a limit $\mathbf{B}_\infty \in L^q(\Omega,\mathscr{C}^\beta_{g,T})$ such that the previous inequality holds when $\mathbf{B}_m$ is replaced by $\mathbf{B}_\infty$.  
In addition, observe that $\mathbf{B}_\infty = \mathbf{B}$ by well-known Wong Zakai results, see for instance \cite{FrRi14}.
Hence we have proved
\begin{align*}
\left\| \varrho_{\beta,T}(\mathbf{B},\mathbf{B}_n) \right\|_{L^q(\Omega)}
\leq
C q^{1/2} \varsigma_n^{\alpha-\beta}
\end{align*}
for every finite $q \geq 1$. 

Standard estimates now imply that the random variable $c(T,n,\omega) \coloneqq \varrho_{\beta,T}(\mathbf{B},\mathbf{B}_n)\, \varsigma_n^{\beta-\alpha}$ has Gaussian tails for every $n \in \N$ and thus, by Borel-Cantelli theorem, for every $T \in [0,\infty)$ and $\omega \in \Omega$ there exists $N=N(T,\omega)$ such that $c(T,n,\omega)\leq n+1$ for every $n \geq N$.  
In particular, for every $T$ the random variable
\begin{align*}
\tilde{c}(T,\omega) 
\coloneqq 
\sup_{n \in \N} \frac{c(T,n,\omega)}{n+1}
< 
\infty \quad \mbox{almost surely},
\end{align*}
and it is almost surely non-decreasing and lower-semicontinuous with respect to $T$ as a supremum of non-decreasing and lower-semicontinuous functions.

Therefore, recalling that the filtration $\{\mathcal{F}_t\}_{t \geq 0}$ is complete and right-continuous, we can define the stopping time\footnote{To see that $\mathfrak{s}$ is in fact a stopping time, write $\{\mathfrak{s} < t\} = \cup_{q \in \mathbb{Q} \cap (\infty,t)} \{ \tilde{c}(q,\cdot) > K\} \in \mathcal{F}_t$ for every $t \in \R$. Then $\{\mathfrak{s} \leq t\} = \cap_{\epsilon>0} \{\mathfrak{s} < t+\epsilon\} \in \cap_{\epsilon>0} \mathcal{F}_{t+\epsilon} = \mathcal{F}_t$ by right-continuity of the filtration.} $\mathfrak{s}$ as
\begin{align*}
\mathfrak{s}
&\coloneqq
\inf \left\{ s \geq 0 : \tilde{c}(s,\cdot) > K \right\}
\wedge 
K,
\end{align*}
and then for every $n$ we have almost surely:
\begin{align*}
\varrho_{\beta,\mathfrak{s}}(\mathbf{B},\mathbf{B}_n)
\leq
\tilde{c}(\mathfrak{s}) (n+1) \varsigma_n^{\alpha-\beta}
\leq
K (n+1) \varsigma_n^{\alpha-\beta}.
\end{align*} 
Finally, $\mathfrak{s} \to \infty$ almost surely as $K \to \infty$ because $\tilde{c}$ is almost surely finite and non-decreasing.
\end{proof}

\subsection{Choice of the stopping time} \label{ssec:stop}
For an arbitrary constant $K>0$ to be chosen later, define the stopping time
\begin{align} \label{eq:stopping}
\mathfrak{t} 
\coloneqq 
\mathfrak{s}
\wedge
\inf
\left\{ s \geq 0 : \| B \|_{C^\alpha_{\leq s}} > K \right\}
\wedge
\inf
\left\{ s \geq 0 : \| \mathbb{B}\|_{C^{2\alpha}_{\leq s}} > K \right\},
\end{align}
where $\mathfrak{s}$ is given by \autoref{lem:lyons} (and depends on $K$), $\| B \|_{C^\alpha_{\leq s}}$ denotes the $\alpha$-H\"older norm of $B$ on the time interval $(-\infty,s]$ (or equivalently on $[0,s]$ since $B(t)=0$ for $t<0$) and $\| \mathbb{B}\|_{C^{2\alpha}_{\leq s}}$ denotes the $2\alpha$-H\"older norm of the two-indices process $\mathbb{B}$, restricted on the time interval $[0,s]\times[0,s]$.

Observe that $\mathfrak{t} \leq \mathfrak{s} \leq K$ almost surely.
The constant $K$ can be taken large enough so that $\mathfrak{t}$ is large with high probability: namely, given parameters $\varkappa,T$ as in the statement of \autoref{thm:strong_ex}, there exists $K_0=K_0(\varkappa,T)$ sufficiently large such that $\mathbb{P}\{ \mathfrak{t} \geq T\} \geq \varkappa$.
This is because both $\| B \|_{C^\alpha_{\leq s}}$ and $\| \mathbb{B}\|_{C^{2\alpha}_{\leq s}}$ are almost surely finite for every $s<\infty$, and $\mathfrak{s} \to \infty$ almost surely as $K \to \infty$ by construction.
 
Now we state a crucial lemma describing rates of convergence of the approximating flow $\phi_n \to \phi$, as well as its inverse $\phi_n^{-1} \to \phi^{-1}$.
Since $\phi_n$, $\phi$ equal the identity map on the torus for negative times, the convergence is only interesting for positive times.
For technical reasons we must restrict ourselves to time intervals of the form $[0,\mathfrak{t}]$, where $\mathfrak{t}$ is given by \eqref{eq:stopping} and depends on $K>0$.
In particular, we shall see that the approximation becomes worse and worse as $K \to \infty$.

\begin{lem} \label{lem:flow}
For every $n \in \N$ the map $\phi_n(\cdot,t)$ defined by \eqref{eq:def_phi_n} is almost surely a $C^\infty$-diffeomorphism of the torus for every $t \in \R$, and the following hold true.
\begin{itemize}
\item[$i$)]
For every $K > 0$, $\kappa \in \N$ and $\beta \in (0,\alpha)$ there exist constants $C_1,C_2$ such that for every $n \in \N$ it holds 
\begin{gather*}
\| \phi_{n+1} - \phi_n \|_{C^\beta_{\leq \mathfrak{t}} C^\kappa_x}
\leq
C_1 (n+1) \varsigma_n^{\alpha-\beta},
\quad
\| \phi_n \|_{C^\alpha_{\leq \mathfrak{t}} C^\kappa_x}
\leq
C_2,
\\
\| \phi_{n+1}^{-1} - \phi_n^{-1} \|_{C^\beta_{\leq \mathfrak{t}} C^\kappa_x}
\leq
C_1 (n+1) \varsigma_n^{\alpha-\beta}, 
\quad
\| \phi_n^{-1} \|_{C^\alpha_{\leq \mathfrak{t}} C^\kappa_x}
\leq
C_2;
\end{gather*}
Moreover, the same inequalities hold with $\phi_{n+1}$ replaced by $\phi$.
\item[$ii$)]
For every $K > 0$ and $\kappa,r \in \N$ there exists a constant $C_3$ such that for every $n \in \N$ it holds
\begin{align*}
\| \phi_n \|_{C^r_{\leq \mathfrak{t}} C^\kappa_x} 
\leq 
C_3 \varsigma_n^{\alpha-r},
\quad
\| \phi_n^{-1} \|_{C^r_{\leq \mathfrak{t}} C^\kappa_x} 
\leq 
C_3 \varsigma_n^{\alpha-r}. 
\end{align*}
\end{itemize} 
\end{lem}
\begin{proof}[Proof of \autoref{lem:flow}]
First, notice that we can replace the time interval $(-\infty,\mathfrak{t}]$ with $[0,\mathfrak{t}]$ in all the norms on the left-hand-sides. 
By \cite[Theorem 8.5]{FrHa20} (cf. also \cite[Lemma 13]{CrDiFrOb13}), for every $K>0$, $\kappa \in \N$ and $\beta \in (0,\alpha)$ there exist constant $C_1,C_2$ such that for every $n$
\begin{gather*}
\| \phi_{n+1} - \phi_n \|_{C^\beta_{\leq \mathfrak{t}} C^\kappa_x}
\leq
C_1 \varrho_{\beta,\mathfrak{t}}(\mathbf{B}_{n+1},\mathbf{B}_n),
\quad
\| \phi_n \|_{C^\alpha_{\leq \mathfrak{t}} C^\kappa_x}
\leq
C_2,
\\
\| \phi_{n+1}^{-1} - \phi_n^{-1} \|_{C^\beta_{\leq \mathfrak{t}} C^\kappa_x}
\leq
C_1 \varrho_{\beta,\mathfrak{t}}(\mathbf{B}_{n+1},\mathbf{B}_n), 
\quad
\| \phi_n^{-1} \|_{C^\alpha_{\leq \mathfrak{t}} C^\kappa_x}
\leq
C_2,
\end{gather*}
and the same estimates hold true when $\mathbf{B}_{n+1}$ is replaced by $\mathbf{B}$. 
Notice that the mentioned result is stated on finite time intervals $[0,K]$; the case $[0,\mathfrak{t}] \subset [0,K]$ is handled by considering the stopped processes $\mathbf{B}_n(\cdot \wedge \mathfrak{t})$.
Also, it is worth mentioning that we use the bound on $\| \mathbb{B} \|_{C^{2\alpha}_{\leq \mathfrak{t}}}$ coming from the definition of $\mathfrak{t}$ to verify the assumptions of the aforementioned theorem: indeed, the constants in the thesis may depend also on $\| \mathbb{B} \|_{C^{2\alpha}_{\leq \mathfrak{t}}}$ and $\| \mathbb{B}_n \|_{C^{2\alpha}_{\leq \mathfrak{t}}}$ \footnote{To bound $\| \mathbb{B}_n \|_{C^{2\alpha}_{\leq \mathfrak{t}}}$ uniformly in $n$, write $\| \mathbb{B}_n \|_{C^{2\alpha}_{\leq \mathfrak{t}}} \leq \| \mathbb{B} \|_{C^{2\alpha}_{\leq \mathfrak{t}}} + \| \mathbb{B}_n-\mathbb{B} \|_{C^{2\alpha}_{\leq \mathfrak{t}}}$ and use \autoref{lem:lyons}, without loss of generality increasing the values of $\beta,\alpha$ therein.} (as well as on $\| B \|_{C^{\alpha}_{\leq \mathfrak{t}}}$ and $\| B_n \|_{C^{\alpha}_{\leq \mathfrak{t}}}$).

Part $i$) of the thesis then follows by \autoref{lem:lyons} and $\mathfrak{t} \leq \mathfrak{s}$.

For part $ii$) we proceed iteratively. The equations satisfied by the time derivatives of $\phi_n$ for positive times $t>0$ are
\begin{align*}
\phi_n(x,t) 
&= 
x + \sum_{k \in I} \int_0^t \sigma_k(\phi_n(x,s)) \dot{B}^k_n(s) ds, 
\\
\dot{\phi}_n(x,t) 
&= 
\sum_{k \in I} \sigma_k(\phi_n(x,t)) \dot{B}^k_n(t),
\\
\ddot{\phi}_n(x,t)
&=
\sum_{k \in I} \dot{\phi}_n(x,t) \cdot \nabla \sigma_k(\phi_n(x,t))  \dot{B}^k_n(t) + \sigma_k(\phi_n(x,t)) \ddot{B}^k_n(t),
\\
&\cdots
\\
\partial^r_t \phi_n(x,t)
&=
\sum_{k \in I}
\sum_{m=0}^{r-1} C(r,m)
\partial^m_t \sigma_k(\phi_n(x,t))
\partial^{r-1-m}_t \dot{B}^k_n(t),
\end{align*}
giving estimates (use \cite[Proposition 4.1]{DLS14} to bound derivatives of $\sigma_k \circ \phi_n$) 
\begin{align*}
[ \phi_n ]_{C^1_{\leq \mathfrak{t}}C^\kappa_x}
&\leq
\sum_{k \in I}
\| \sigma_k \circ \phi_n \|_{C_{\leq \mathfrak{t}}C^\kappa_x} \|\dot{B}^k_n\|_{C_{\leq \mathfrak{t}}}
\leq 
CK \varsigma_n^{\alpha-1},
\\
[ \phi_n ]_{C^2_{\leq \mathfrak{t}}C^\kappa_x}
&\leq
\sum_{k \in I}
\| \nabla\sigma_k \circ \phi_n \|_{C_{\leq \mathfrak{t}}C^\kappa_x} [ \phi_n ]_{C^1_{\leq \mathfrak{t}}C^\kappa_x} \|\dot{B}^k_n\|_{C_{\leq \mathfrak{t}}} +
\| \sigma_k \circ \phi_n \|_{C_{\leq \mathfrak{t}}C^\kappa_x} \|\ddot{B}^k_n\|_{C_{\leq \mathfrak{t}}}
\\
&\leq
C K^2 \varsigma_n^{2\alpha-2} + C K \varsigma_n^{\alpha-2}
\leq C K^2 \varsigma_n^{\alpha-2},
\\
&\cdots
\\
[ \phi_n ]_{C^r_{\leq \mathfrak{t}}C^\kappa_x}
&\leq C K^r \varsigma_n^{\alpha-r}.
\end{align*}
Finally, estimates for the inverse flow are obtained similarly, noticing that $\phi_n^{-1}$ is driven by $B_n(t-\cdot)$, cf. also the proof of \cite[Proposition 11.13]{FrVi10}, namely
\begin{align*}
\phi_n^{-1}(x,t) 
= 
x - \sum_{k \in I} \int_0^t \sigma_k(\phi_n^{-1}(x,s)) \dot{B}^k_n(t-s) ds.
\end{align*}
\end{proof}

\subsection{Localization} \label{ssec:localize}
Let us fix $\beta \in (0,\alpha)$.
A close inspection of the proof of \autoref{lem:flow} guarantees that the constants $C_1, C_2, C_3$ can be chosen to be increasing with respect to $K, \kappa, r$.
Since throughout our construction we will need only a finite number of derivatives with respect to space and time, we can assume $\kappa, r$ bounded from above by a universal constant, so that we can suppose $C_1, C_2, C_3$ only depend on $K$ (at fixed $\beta$, $\alpha$).
In view of this, for every integer $L \geq 1$ we can find a parameter $K_L$ such that \autoref{lem:flow} holds with $C_1 = C_2 = C_3 = C L$, for some universal constant $C$. Moreover, without loss of generality we can assume the sequence $K_L$ to be increasing and diverging to $\infty$, and $K_0=K_1 \leq K_L $ for every $L \geq 1$ (possibly taking a larger value of $C$ if necessary; recall that $K_0$ was such that $\mathbb{P} \{ \mathfrak{t} \geq T \} \geq \varkappa$).
Thus, given $\mathfrak{s}_L$ as in \autoref{lem:lyons} with $K=K_L$ and defining
\begin{align} \label{eq:stopping_L}
\mathfrak{t}_L 
\coloneqq 
\mathfrak{s}_L
\wedge
\inf
\left\{ s \geq 0 : \| B \|_{C^\alpha_{\leq s}} > K_L \right\}
\wedge
\inf
\left\{ s \geq 0 : \| \mathbb{B}\|_{C^{2\alpha}_{\leq s}} > K_L \right\},
\end{align}
we obtain a sequence of non-decreasing stopping times $\mathfrak{t} = \mathfrak{t}_1 \leq \dots \leq \mathfrak{t}_L \leq \dots$ such that $\mathfrak{t}_L \to \infty$ almost surely as $L \to \infty$ and
\begin{gather*}
\| \phi_{n+1} - \phi_n \|_{C^\beta_{\leq \mathfrak{t}_L} C^\kappa_x}
\leq
CL (n+1) \varsigma_n^{\alpha-\beta},
\quad
\| \phi_n \|_{C^\alpha_{\leq \mathfrak{t}_L} C^\kappa_x}
\leq
CL,
\\
\| \phi_{n+1}^{-1} - \phi_n^{-1} \|_{C^\beta_{\leq \mathfrak{t}_L} C^\kappa_x}
\leq
CL (n+1) \varsigma_n^{\alpha-\beta}, 
\quad
\| \phi_n^{-1} \|_{C^\alpha_{\leq \mathfrak{t}_L} C^\kappa_x}
\leq
CL,
\end{gather*}
as well as
\begin{align*}
\| \phi_n \|_{C^r_{\leq \mathfrak{t}_L} C^\kappa_x} 
\leq 
CL \varsigma_n^{\alpha-r},
\quad
\| \phi_n^{-1} \|_{C^r_{\leq \mathfrak{t}_L} C^\kappa_x} 
\leq 
CL \varsigma_n^{\alpha-r}, 
\end{align*}
for some constant $C$ depending only on $K_0$, $\alpha$ and $\beta$.

\begin{rmk} \label{rmk:H>1/4}
Estimates similar to those given above hold more generally when the collection of Brownian motions $\{B^k\}_{k \in I}$ is replaced by fractional Brownian motions $\{B^{H,k}\}_{k \in I}$ with Hurst parameter $H>1/4$.
The lower threshold comes from the fact that we are using that iterated integrals of the mollified fractional Brownian motions satisfy a Wong-Zakai-type result (cf. \cite{FrRi14}), and therefore one can use this limit as a ``canonical" enhancement of $\{B^{H,k}\}_{k \in I}$ when interpreting the integral \eqref{eq:flow} defining the flow $\phi$.
The space of geometric rough paths and the distance on it must be modified accordingly, including the step-3 iterated integral of the path, and the exponent $\alpha$ must be taken smaller than $H$.
As a consequence, with suitable interpretation of the stochastic integral in \eqref{eq:euler} and the notion of solution to \eqref{eq:euler} and \eqref{eq:random_euler}, we can prove the analogue of our main results also in the fractional case, with minor modifications in the construction and in the choice of parameters carried on in the remainder of the paper.
We omit additional details.
\end{rmk}
 
\subsection{Main proposition} \label{ssec:main_prop}
Throughout the scheme we will make the following choice of parameters: 
\begin{align} \label{eq:parameters}
\delta_n \coloneqq a^{1-b^n},
\quad
D_n \coloneqq a^{cb^n},
\quad
L_n \coloneqq L^{m^{n+1}},
\end{align}
where
\begin{align*}
a \geq 2,
\quad
b=m+\varepsilon,
\quad
c=\frac{b^4(1+\varepsilon)-1/2}{b-1-\varepsilon}>0,
\quad
\varepsilon > 0,
\end{align*} 
and $m\geq 4$ is given by \autoref{prop:it} below. We point out that differently from \cite{DLS14} here $\varepsilon$ cannot be taken arbitrarily small, and it shall be fixed only at the end of \autoref{sec:proof}.
Notice that, at fixed $\varepsilon$ and for large values of $m$, the parameter $c$ is approximately equal to $(1+\varepsilon)b^3$; in particular $D_n$ increases in $n$ as a negative power of $\delta_{n+3}$ (with exponent depending on $\varepsilon$). 
The choice of the parameter $\varsigma_n$, upon which the definition of the mollified flow $\phi_n$ depends, will be given in \eqref{eq:definition_ell}. 

Our goal is to prove the following:
\begin{prop} \label{prop:it}
Let $e$ be a smooth positive function on $\R$ satisfying $\underline{e} \coloneqq \inf_{t \in \R} e(t)>0$ and $\overline{e} \coloneqq \|e\|_{C^2_t}<\infty$. 
Then there exist constants $\varepsilon, m$ as above, $a \geq 2$ depending on $\underline{e},\overline{e},K_0$, a constant $\eta \in (0,1)$ depending on $\underline{e},\overline{e}$, constants $M_v,M_q$ depending only on $\overline{e}$ and a constant $A \in (0,\infty)$ depending on $\underline{e},\overline{e},K_0$ with the following property.

Fix $\phi_n$ as in \eqref{eq:def_phi_n} and let $(v_n,q_n,\phi_n,\mathring{R}_n)$, $n \in \N$, be a solution of \eqref{eq:random_euler-reynolds} satisfying the inductive estimates\footnote{We adopt the convention $\sum_{k=0}^{n-1}\delta_k=0$ for $n=0$.} 
\begin{gather}
\label{eq:est_energy}
\left| 
e(t) (1-\delta_n)- \int_{\T^3} |v_n(x,t)|^2 dx  \right|
\leq 
\tfrac{1}{4} \delta_n e(t),
\qquad \forall t \leq \mathfrak{t},
\end{gather}
and for every $L \in \N, L \geq 1$
\begin{gather}
\|\mathring{R}_n\|_{C_{\leq \mathfrak{t}_L}C_x}   \label{eq:est_Reynolds}
\leq 
\eta L_n \delta_{n+1},
\\
\| q_n\|_{C_{\leq \mathfrak{t}_L}C_x}   \label{eq:est_pres}
\leq M_q L_n \sum_{k=0}^{n-1} \delta_{k},
\\
\| \dvg^{\phi_n} v_n \|_{C_{\leq \mathfrak{t}_L} B^{-1}_{\infty,\infty}}  \label{eq:est_div}
\leq
L_n \delta_{n+2}^{5/4},
\\
\max \left\{  \label{eq:C^1,1}
\|v_n\|_{C^1_{\leq \mathfrak{t}_L,x}},
\|q_n\|_{C^1_{\leq \mathfrak{t}_L,x}},
\|\mathring{R}_n\|_{C_{\leq \mathfrak{t}_L}C^1_x}
\right\}
\leq L_n D_n.
\end{gather}
Then there exists a second quadruple $(v_{n+1},q_{n+1},\phi_{n+1},\mathring{R}_{n+1})$ solution of \eqref{eq:random_euler-reynolds} satisfying \eqref{eq:def_phi_n} and the inductive estimates \eqref{eq:est_energy}, \eqref{eq:est_Reynolds}, \eqref{eq:est_pres}, \eqref{eq:est_div} with $n$ replaced by $n+1$, and for every $L \in \N, L \geq 1$
\begin{align*}
\| v_{n+1} - v_n \|_{C_{\leq \mathfrak{t}_L}C_x} &\leq M_v L_n^4 \delta_n^{1/2},
\\
\| q_{n+1}-q_n \|_{C_{\leq \mathfrak{t}_L}C_x} &\leq M_q L_n  \delta_n,
\\
\max \left\{ 
\|v_{n+1}\|_{C^1_{\leq \mathfrak{t}_L,x}},
\|q_{n+1}\|_{C^1_{\leq \mathfrak{t}_L,x}},
\|\mathring{R}_{n+1}\|_{C_{\leq \mathfrak{t}_L}C^1_x}
\right\}
&\leq 
A L_{n+1} \delta_n^{1/2} \left( \frac{D_n}{\delta_{n+4}} \right)^{1+\varepsilon}.
\end{align*}
Moreover, the quadruple $(v_{n+1},q_{n+1},\phi_{n+1},\mathring{R}_{n+1})$ evaluated at time $t \in [0,\infty)$ only depends upon $e(s)$, $\phi(y,s)$, $v_k(y,s)$, $q_k(y,s)$, $\phi_k(y,s)$, $\mathring{R}_k(y,s)$, for arbitrary $s \leq t$, $k \leq n$, and $y \in \T^3$.
\end{prop}

We point out that, in the statement of the previous proposition, the parameters $a$ and $A$ are allowed to depend on $\underline{e},\overline{e},K_0$, whereas $\varepsilon$ and $m$ are universal. We shall see that the H\"older exponent $\vartheta$ of the solutions we can construct will in fact depend only on $\varepsilon$ and $m$, and therefore is universal as well.
However, the local-in-time $\vartheta$-H\"older norm of the solutions does depend on the parameter $a$ (as well as $M_v$ and $M_q$); in particular it may vary as the energy profile $e$ and the constant $K_0$ change. 

Also, we recall that approximate solutions $(v_n,q_n,\phi_n,\mathring{R}_n)$ in the iteration are defined for all times $t \in \R$. For negative times, they correspond to deterministic approximate solutions of deterministic Euler equations and noise is only inserted for $t>0$ (we shall restrict them to positive times when proving \autoref{thm:strong_ex}). 

Next we show that \autoref{prop:it} above implies \autoref{thm:strong_ex}.
\begin{proof}[Proof of \autoref{thm:strong_ex}]
The proof is mostly inspired by \cite{DLS14}.

\emph{Step 1}.
We define the initial step of the iteration $(v_0,q_0,\mathring{R}_0)$ to be identically zero, which is admissible since \eqref{eq:est_Reynolds}, \eqref{eq:est_pres}, \eqref{eq:est_div} and \eqref{eq:C^1,1} are trivially satisfied and \eqref{eq:est_energy} holds true because by definition $\delta_0=1$.
Applying iteratively the proposition we get a sequence $\{v_n\}_{n \in \N}$ satisfying
\begin{align*}
\| v_n \|_{C_{\leq \mathfrak{t}_L}C_x} 
&\leq 
M_v \sum_{k \leq n} L_k a^{\frac12(1-b^k)}
\leq 
M_v L_n 
\sum_{n \in \N} 2^{\frac12(1-4^n)} 
\leq 
2 M_v L_n,
\\
\| v_n \|_{C^1_{\leq \mathfrak{t}_L,x}} 
&\leq L_n D_n, 
\end{align*}
where in the second line we have used
\begin{align*}
A \delta_n^{1/2} \left( \frac{D_n}{\delta_{n+4}} \right)^{1+\varepsilon}
&=
A a^{1/2-(1+\varepsilon)} 
a^{(-1/2+c(1+\varepsilon)+b^4(1+\varepsilon))b^n}
\\
&=
A a^{1/2-(1+\varepsilon)} a^{cb^{n+1}}
\leq
D_{n+1},
\end{align*}
since $(-1/2+c(1+\varepsilon)+b^4(1+\varepsilon))=cb$ by the very definition of $c$ and taking $a \geq A^{\frac{1}{(1+\varepsilon)-1/2}}$.

\emph{Step 2}.
By assumption, for every $L,n \in \N$, $L \geq 1$ we have
\begin{align*}
\| v_{n+1} - v_n \|_{C_{\leq \mathfrak{t}_L}C_x} 
&\leq 
M_v L_n^4 \delta_n^{1/2}
= 
M_v L^{4 m^{n+1}} a^{\frac12 (1-b^n)},
\\
\| v_{n+1} - v_n \|_{C^1_{\leq \mathfrak{t}_L,x}} 
&\leq 
2 L_{n+1} D_{n+1}
= 
2 L^{m^{n+2}} a^{cb^{n+1}},
\end{align*}
and therefore by interpolation
\begin{align*}
\| v_{n+1} - v_n \|_{C^\vartheta_{\leq \mathfrak{t}_L,x}}
&\leq
2^\vartheta M_v^{1-\vartheta} 
L^{(1-\vartheta)4m^{n+1}}L^{\vartheta m^{n+2}} 
a^{\frac{1-\vartheta}{2}} a^{(\vartheta cb - \frac{(1-\vartheta)}{2})b^n}
\\
&\leq
M_v L^{m^{n+2}} a^{\frac{1-\vartheta}{2}} a^{(\vartheta cb - \frac{(1-\vartheta)}{2})b^n},
\end{align*}
where we have assumed without any loss of generality $M_v \geq 2$.
Now choose 
\begin{align} \label{eq:vartheta}
0<\vartheta 
&< 
\frac{1}{2cb+1} 
=
\frac{m-1}{2(1+\varepsilon)(m+\varepsilon)^5-1-\varepsilon},
\end{align} 
so that $\vartheta cb-\frac{(1-\vartheta)}{2} \eqqcolon -\gamma <0$, and define
\begin{align*}
n_0 = \max \left\{ n \in \N : L^{m^{n+2}}  > a^{\gamma b^{n-1}} \right\}.
\end{align*}
The maximum above exists since $b>m$ and $\gamma>0$, and simple algebraic manipulations show that $n_0$ satisfies
\begin{align*}
n_0 &< \frac{\log \log L + 2 \log m - \log\gamma + \log b - \log\log a}{\log b - \log m}
\\
&\leq C (1+\log\log L)
\end{align*}
for some constant $C$ not depending on $L$.
With this choice of $n_0$ we estimate
\begin{align*}
\sum_{n \in \N} \| v_{n+1} - v_n \|_{C^\vartheta_{\leq \mathfrak{t}_L,x}}
&\leq
\sum_{n \leq n_0}
\| v_{n+1} - v_n \|_{C^\vartheta_{\leq \mathfrak{t}_L,x}}
+
\sum_{n > n_0}
\| v_{n+1} - v_n \|_{C^\vartheta_{\leq \mathfrak{t}_L,x}}
\\
&\leq
C M_v a^{1/2} (L^{m^{n_0+3}}+1)
\leq
C M_v a^{1/2} L^{m^{C(1+\log\log L)}},
\end{align*}
where the constant $C$ may vary from line to line.
Thus $v_n$ converges in $C^{\vartheta}_{\leq \mathfrak{t}_L} C_x \cap C_{\leq \mathfrak{t}_L} C^\vartheta_x$ towards a limit $v$ with uniform bound in $\omega \in \Omega$
\begin{align*}
\|v\|_{C^\vartheta_{\leq \mathfrak{t}_L,x}} 
\leq 
C M_v a^{1/2} L^{m^{C(1+\log\log L)}}.
\end{align*}
Moreover, it is easy to check that $v$ restricted to $(-\infty,\mathfrak{t}_{L-1}]$ is simply the limit of $v_n$ in $C^{\vartheta}_{\leq \mathfrak{t}_{L-1}} C_x \cap C_{\leq \mathfrak{t}_{L-1}} C^\vartheta_x$ for every $L>1$. 
This identifies uniquely a limit object in $C^{\vartheta}_{loc} C_x \cap C_{loc} C^\vartheta_x$, denoted again by $v$, by simply gluing together limits at different values of $L$.
  
Since $v_n(\cdot \wedge \mathfrak{t}_L)$ is progressively measurable for every $n \in \N$ and $L \geq 1$, the limit process $v(\cdot \wedge \mathfrak{t}_L)$ is also progressively measurable; so also $v$ is. 
Moreover, because of \eqref{eq:est_energy} it  satisfies 
\begin{align} \label{eq:energy_vL}
\int_{\mathbb{T}^3}|v(x,t)|^2 dx = e(t), 
\quad
t \leq \mathfrak{t}.
\end{align}

\emph{Step 3}.
Convergence of the pressure $q_n$, as well as the regularity of the limit $q$, can be carried on exactly as in the previous step.
The only difference comes down to the fact that $\| q_{n+1} - q_n \|_{C_{\leq \mathfrak{t}_L}C_x} \leq M_q L_n \delta_n = M_q L^{m^{n+1}} a^{1-b^n}$ and thus
\begin{align*}
\| q_{n+1} - q_n \|_{C^\vartheta_{\leq \mathfrak{t}_L,x}}
\leq
M_q L^{m^{n+2}} a^{1-\vartheta} a^{(\vartheta cb - (1-\vartheta))b^n}.
\end{align*} 
As a consequence, one can produce solutions $q$ with H\"older regularity $\vartheta'$ such that
\begin{align} \label{eq:vartheta'}
0 < \vartheta' < \frac{1}{cb+1},
\end{align}
which is approximately twice the regularity of $v$ for large values of $b$. Up to choosing a smaller $\vartheta$, we can actually suppose $\vartheta'=2\vartheta$ and thus the limit $v$ satisfies the uniform bound (with respect to $\omega \in \Omega$)  
\begin{align*}
\|q\|_{C^{2\vartheta}_{\leq \mathfrak{t}_L,x}} \leq C M_q a L^{m^{C(1+\log\log L)}}.
\end{align*}

\emph{Step 4}.
Finally, notice that for every $L$ we have
$\phi_n \to \phi$, $\phi_n^{-1} \to \phi^{-1}$ almost surely in $C_{\leq \mathfrak{t}_L} C_x^2$ (by \autoref{lem:flow}), $\mathring{R}_n \to 0$ almost surely in $C_{\leq \mathfrak{t}_L} C_x$ and $\dvg^{\phi_n} v_n \to 0$ almost surely in $C_{\leq \mathfrak{t}_L} B^{-1}_{\infty,\infty}$ (by iterative assumptions \eqref{eq:est_Reynolds} and \eqref{eq:est_div}, having chosen $b>m$), and therefore the couple $(v,q)$ is a solution to \eqref{eq:random_euler} on the time interval $[0,\mathfrak{t}_L]$.
Indeed, for every $H \in C^\infty(\T^3)$ and $h \in C_{loc}([0,\infty),C^\infty(\mathbb{T}^3,\R^3))$ we have $\nabla^{\phi_n} H \to \nabla^\phi H$ and $(v_n \cdot \nabla^{\phi_n}) h \to (v \cdot \nabla^\phi)h$, $\dvg^{\phi_n} h \to \dvgphi h$ in $C_{[0,\mathfrak{t}_L]} C_x$ (cf. also the estimates in \autoref{ssec:flow_error}). 
Since $L$ is arbitrary and $\mathfrak{t}_L \to \infty$ as $L \to \infty$, the proof is complete.
\end{proof}

\begin{rmk}[On the H\"older exponent $\vartheta$]
From \eqref{eq:vartheta} one obtains an upper bound for the H\"older regularity of the solution $(v,q)$.
In \cite{DLS14} $\vartheta$-H\"older solutions to the deterministic Euler equations are constructed for every $\vartheta<1/10$; in the series of papers \cite{BuDLIsSz15,Bu15,BuDLIsSz16}, culminating with \cite{Is18,BuDLSzVi19}, the space regularity threshold has been successively extended up to exponent $1/3$, which is optimal since solutions to Euler equations more regular than this necessarily preserve their kinetic energy \cite{CoWeTi94}.

Here, due to the presence of the noise and because $m$ will be taken large, we have to restrict ourselves to much smaller values of $\vartheta$. 
The choice of parameters done in \autoref{sec:proof} only allows values of $\vartheta$ smaller than approximately $2.76 \times 10^{-9}$, which is an extremely low threshold when compared with the deterministic literature on the Onsager's conjecture.
Although we have not optimized our choice of parameters - that is, our threshold is not sharp - we believe one cannot do much better than this following the construction of the present paper, not to mention reaching the same threshold as in \cite{DLS14} $\vartheta<1/10$. 
It is interesting to understand if technical refinements could improve regularity of solutions up to critical exponent $1/3$ (at least when considering space regularity), as it was done in the deterministic case.
However, the most stringent constraint here is the requirement $b>m$, that is a genuinely stochastic problem (the constant $L_n$ is the estimates comes from the growth of the flow $\phi$); thus, it seems not possible to take ideas from more refined deterministic convex integration schemes (as for instance that of \cite{Is18}) in order to improve the resulting regularity.
\end{rmk}

\begin{rmk}[On the local-in-time H\"older norm of solutions]
Also, we would like to mention the following fact.
The estimates
\begin{align*}
\|v\|_{C^\vartheta_{\leq \mathfrak{t}_L,x}}
&\leq 
C M_v a^{1/2} L^{m^{C(1+\log\log L)}},
\quad
\|q\|_{C^{2\vartheta}_{\leq \mathfrak{t}_L,x}} \leq 
C M_q a L^{m^{C(1+\log\log L)}},
\end{align*}
involve the constants $M_v$ and $M_q$, which in turn only depend on $\overline{e}$, and the parameter $a$, that however depends also on $\underline{e}$ and $K_0$.
The latter dependence on $a$ comes from the fact that we started the iteration from the ansatz $(v_0,q_0,\mathring{R}_0)$ identically zero, and therefore we needed $\delta_0=1$ in order to satisfy \eqref{eq:est_energy}.
If we were able to cook up an initial $(v_0,q_0,\mathring{R}_0)$ with ``better" energy profile to start the iteration, we could redefine $\delta_n=a^{-b^n}$ and lose the dependence on $a$ in the H\"older norms of the solution $(v,q)$.
\end{rmk}

To conclude this section, let us mention that \autoref{thm:non-uniq} follows readily from the previous construction and the arguments of \autoref{ssec:main_results}, using the property that $(v_{n+1},q_{n+1},\phi_{n+1},\mathring{R}_{n+1})(x,t)$ depends on $e(s)$, $\phi(y,s)$, $v_k(y,s)$, $q_k(y,s)$, $\phi_k(y,s)$, $\mathring{R}_k(y,s)$, for arbitrary $s \leq t$, $k \leq n$, and $y \in \T^3$.
This last property will be implicitly checked in \autoref{sec:convex} below.

\section{Convex integration scheme} \label{sec:convex}
Recall that our main aim is to construct the quadruple $(v_{n+1},q_{n+1},\phi_{n+1},\mathring{R}_{n+1})$ solution of the Euler-Reynolds system, given $(v_n,q_n,\phi_n,\mathring{R}_n)$ as in the statement of \autoref{prop:it}.
Of course, the candidate approximating flow $\phi_n$ is already given by \eqref{eq:def_phi_n}, so we shall focus on the definition of $v_{n+1}$, $q_{n+1}$ and $\mathring{R}_{n+1}$ only.
Let us work at fixed $n \in \N$.
The construction follows closely that of \cite{DLS14}.

\subsection{Choice of parameters}
For future reference, here we collect all the parameters necessary to the convex integration scheme.
We recall the quantities from \autoref{ssec:main_prop}
\begin{align*} 
\delta_n \coloneqq a^{1-b^n},
\quad
D_n \coloneqq a^{cb^n},
\quad
L_n \coloneqq L^{m^{n+1}},
\end{align*}
where \begin{align*}
b=m+\varepsilon,
\quad
c=\frac{b^4(1+\varepsilon)-1/2}{b-1-\varepsilon}>0,
\end{align*}
and the coefficients $a \geq 2$, $m\geq 4$ and $\varepsilon>0$ are still to be determined and will be fixed only in \autoref{sec:proof}.
According to \autoref{prop:it}, the parameter $\delta_n$ dictates the decay, along the iteration, of the following quantities: the energy error, the Reynolds stress, the divergence of the velocity field, and the velocity and pressure increments. On the other hand, the parameter $D_n$ is used to keep track of the growth of derivatives of $v_n$, $q_n$, and $\mathring{R}_n$. $L_n$ serves to control how the iterative estimates deteriorate on larger and larger time intervals of the form $[0,\mathfrak{t}_L]$, $L\in \N$, $L \geq 1$.

Let us also discuss the choice of the parameter $\eta \in (0,1)$. There are two upper bounds that we need to be satisfied: $\eta \leq \frac{\underline{e}}{8C(\overline{e}^{1/2}+1)}$ from \autoref{lem:bounds_tilde_e} ($C$ is a constant defined therein), and $\eta \leq \frac{r_0 \underline{e}}{40}$ (where $r_0$ is defined in \autoref{ssec:w_o}). We shall fix $\eta \in (0,1)$ according to these constraints.

Next, we need to introduce two mollification parameters for the convex integration scheme.  
Fix $\alpha \in (1/3,1/2)$ close to $1/2$. Let us define the parameters $\ell$ and $\varsigma_n$ by the relations
\begin{align} \label{eq:definition_ell}
\ell^\alpha
\coloneqq 
\frac{c_{n,\ell}}{C_\ell} \frac{\delta_{n+3}^{4/3}}{D_n},
\qquad
\varsigma_n^{\alpha_\star} 
\coloneqq 
\frac{1}{C_\varsigma}
\frac{\delta_{n+3}^{4/3}}{n+1},
\end{align}
for some $C_\ell,C_\varsigma>1$ sufficiently large and $\alpha_\star \in (0,\alpha)$ to be chosen sufficiently close to $\alpha$, and all independent of $n$, whereas $c_{n,\ell} \in (1,2)$ is such that $\ell^{-1}$ is an integer power of $2$ (it will be needed in \autoref{ssec:comp} to apply \autoref{lem:scaling}).
The introduction of the additional parameter $\alpha_\star$ is useful to simplify the estimates in \autoref{ssec:flow_error}, where we need to take an auxiliary parameter $\alpha' \in (\alpha_\star,\alpha)$, see \autoref{prop:R_flow}.
Without loss of generality we may suppose $D_n \ell^{\alpha}\leq \eta \delta_n$ and $D_n \leq \varsigma_{n+1}^{\alpha-1}$.

Given a standard mollifier $\chi \in C^{\infty}_c([-1,1]^3 \times [0,1))$ define $\chi_\ell(x,t) \coloneqq \ell^{-4} \chi(x\ell^{-1},t\ell^{-1})$ and 
\begin{align*}
v_\ell \coloneqq v_n \ast \chi_\ell,
\quad
q_\ell \coloneqq q_n \ast \chi_\ell,
\quad
\mathring{R}_\ell \coloneqq \mathring{R}_n \ast \chi_\ell.
\end{align*}
Here it is important that $v_n,q_n,\mathring{R}_n$ are defined also for negative times $t<0$ and do satisfy \eqref{eq:random_euler-reynolds} in analytically strong sense. Indeed, we shall need in the following that $v_\ell,q_\ell,\mathring{R}_\ell$ satisfy
\begin{align*}
\partial_t v_\ell 
+
\chi_\ell\ast \dvg^{\phi_n}\left( v_n \otimes v_n\right) 
+
\chi_\ell\ast \nabla^{\phi_n} q_n 
= 
\chi_\ell\ast \dvg^{\phi_n} \mathring{R}_n,
\end{align*}
which would not be generally true for times $t \in (0,\ell)$ otherwise.  

Also, let $\lambda,\mu \in \N$ be large parameters (possibly depending on $n$) such that $\lambda/\mu \in \N$.
Notice that the latter condition implies in particular $\lambda \geq \mu$.
These parameters dictate the frequency of space-time oscillations of the building block of the convex integration scheme.
In order to simplify the computations in \autoref{sec:iter}, we shall assume $\mu^2 \varsigma_{n+1}^{\alpha-2} \leq \lambda \leq D_{n+1}$ and for some $r_\star \in \N$ to be suitably chosen (and independent of $n$):
\begin{align} \label{eq:ass_lambda}
\lambda^{r} &\geq \mu^{r+5} \varsigma_{n+1}^{(r+5)(\alpha-1)-2} \left( D_n \ell^{-r-4} + \varsigma_{n+1}^{\alpha-1} \right),
\quad
\forall r \geq r_\star.
\end{align}

In order to satisfy the previous conditions on $\lambda, \mu$ we can choose $\alpha_\star, \alpha$ sufficiently close to $1/2$, $r_\star \geq 7$ and the parameter $\mu$ as
\begin{align*}
\mu \coloneqq c_{n,\mu} C_\mu \ell^{-r_\star}
\end{align*}
for some constant $C_\mu >1$, possibly large but finite and independent of $n$, and $c_{n,\mu} \in (1,2)$ so that $\mu \in \N$, and
\begin{align*}
\lambda \coloneqq c_{n,\lambda} \mu^2 \varsigma_{n+1}^{\alpha-2}
\end{align*}
where $c_{n,\lambda} \in (1,2)$ satisfies $c_{n,\lambda} \varsigma_{n+1}^{\alpha-2} \in \N$, implying $\lambda \in \N$ and $\lambda / \mu \in \N$.
Notice that the condition $\lambda \leq D_{n+1}$ requires the parameters $A$ and $\varepsilon$ in \autoref{prop:it} to be sufficiently large.
More details will be given in \autoref{sec:proof}.

\subsection{Outline of the construction}
As usual in convex integration schemes, the idea is to define $v_{n+1}$ as a perturbation of $v_n$.
However, in order to guarantee smoothness throughout the construction, we in fact construct $v_{n+1}$ as a perturbation of $v_\ell$.
 
The main building block used in this paper to define the principal perturbation $w_o$ is a modified version of Beltrami flows $E$ (see \autoref{ssec:Beltrami} below), obtained via composition with the flow $\phi_{n+1}$.
This implies, among other things, that $E$ is not steady-state and not even a solution to the stochastic Euler equations \eqref{eq:random_euler}.
Because of this time dependence, we need to introduce an additional dependence on $\dot\phi_{n+1}$ in the amplitude coefficients to compensate for the time derivative of $E$.
The key feature of Beltrami flows is that they oscillate at frequency $\lambda \gg 1$ in space, producing cancellations in space integrals of the perturbation.
By stationary phase Lemma \cite[Proposition 5.2]{DLS13} and its stochastic counterpart \autoref{prop:stat_phase_lem}, these cancellations help in controlling H\"older norms of the new velocity field $v_{n+1}$ and Reynold stress $\mathring{R}_{n+1}$.
On top of that, a fine tuning of the amplitude of $E$ permits us to control the amount of kinetic energy introduced in the system at the $n$-th step of the construction, cf. the term $\rho_\ell$ defined in \autoref{ssec:w_o}. 
We take advantage of this to aim at the desired energy profile $e$ when iterating the construction.
However, because our iterative estimates deteriorate for large values of $t$ (corresponding to large values of $L$) and we want to define $\rho_\ell$ independent of $L$ to preserve adaptedness of solutions, we can only reach the desired energy profile up to time $\mathfrak{t}$. For larger times we content ourselves to inject enough energy at each step so that the perturbation $w_o$ is well-defined (cf. the term under square root in the definition of $\rho_\ell$ and \eqref{eq:well_def} below).

Finally, since we want to control the quantity $\dvgphin v_{n+1}$ we need to add a corrector term $w_c$ to make sure the latter decreases fast enough at every step of the iteration (\autoref{ssec:w_c}). This step is far more involved than usual, since we cannot choose $w_c$ simply as the orthogonal Leray projection of $v_\ell + w_o$ because of a lack of control on $\partial_t w_c$.  

Once $v_{n+1} = v_\ell + w_o + w_c$ is defined with the procedure above, the Reynolds stress $\mathring{R}_{n+1}$ and pressure $q_{n+1}$ can be determined inverting the operator $\dvgphin$, similarly to what has been done for the deterministic case in \cite{DLS13}, see \autoref{ssec:R&P}.

\subsection{Modified Beltrami Flows} \label{ssec:Beltrami}
The following construction follows that of \cite[Proposition 3.1]{DLS13}.
For any given vector $k \in \mathbb{Z}^3 \setminus \{\mathbf{0}\}$, denote
\begin{align*}
M_k \coloneqq Id - \frac{k}{|k|} \otimes \frac{k}{|k|}.
\end{align*}
Let $\lambda_0\geq 1$ and let $A_k\in\R^3$ be such that $A_k\cdot k=0$, $|A_k|=\tfrac{1}{\sqrt{2}}$ and $A_{-k}=A_k$ 
for $k\in\Z^3$ with $|k|=\lambda_0$.
Furthermore, let us define
\begin{align*}
E_k \coloneqq A_k+i\frac{k}{|k|}\times A_k\in\C^3.
\end{align*}
For any collection $\{a_k\}_{k \in \mathbb{Z}^3, |k|=\lambda_0}$ of complex numbers $a_k\in\C$ satisfying $\overline{a_k} = a_{-k}$ for every $k$, the vector field
\begin{equation*}
E(x,t) \coloneqq \sum_{|k|=\lambda_0}a_k E_ke^{ik\cdot \phi_{n+1}(x,t)}
\end{equation*}
is real valued and satisfies
\begin{equation} \label{eq:Beltrami_sol}
\dvgphin E = 0,
\quad
\dvgphin (E\otimes E)=\nablaphin \left( \frac{|E|^2}{2} \right).
\end{equation}
Furthermore, since $\phi_{n+1}$ is measure preserving: 
\begin{equation*}
\frac1{(2\pi)^3}\int_{\T^3} E\otimes E\,dx = \frac12 \sum_{|k|=\lambda_0} |a_k|^2 M_k .  
\end{equation*}
We shall call the field $E$ \emph{modified Beltrami wave}.
This is the natural adaptation to our stochastic setting of classical Beltrami waves, that are steady-state solutions of Euler equations.
However, our $E$ \emph{is not} a solution to \eqref{eq:random_euler} because of its dependence of $t$.
We shall overcome this issue taking suitable time dependent amplitude coefficients $a_k=a_k(t)$, see next subsections for details.

\subsection{The transport coefficients $\psi_k^{(j)}$} \label{ssec:psi}
Recall the following construction from \cite{DLS13}.
Let $c_1$, $c_2$ be such that $\frac{\sqrt{3}}{2} < c_1 < c_2 < 1$ and fix a non negative function $\varphi\in C^\infty_c (B_{c _2} (0))$ which is identically equal to $1$ on the ball $B_{c_1} (0)$. 
Next, consider the lattice $\Z^3$ and its quotient by $(2\Z)^3$, and denote by $\mathcal{C}_j$ , $j=1, \ldots, 8$ the eight equivalence classes of $\Z^3/(2\Z)^3$.
For each $k\in \Z^3$ denote by $\varphi_k$ the function $\varphi_k (x):= \varphi (x-k)$.
Observe that, if $k\neq l \in \mathcal{C}_j$, then $|k-l|\geq 2 > 2 c_2$; hence, $\varphi_k$ and  $\varphi_l$ have disjoint supports. 

On the other hand, the function
\begin{align*}
\varphi_\Sigma \coloneqq \sum_{k\in \Z^3} \varphi_k^2
\end{align*}
is smooth, bounded and bounded away from zero. We then define, for $v \in \R^3$ and $\tau \in \R$:
\begin{align*}
\alpha_k(v) 
&\coloneqq \frac{\varphi_k(v)}{\sqrt{\varphi_\Sigma(v)}},
\quad
k \in \Z^3,
\\
\psi^{(j)}_k(v,\tau) 
&\coloneqq 
\sum_{l\in\mathcal{C}_j}\alpha_l(\mu v)e^{-ik\cdot (\frac{l}{\mu})\tau},
\quad
k \in \Z^3,
\,
j=1,\dots,8.
\end{align*}

Since $\alpha_l$ and $\alpha_{\tilde l}$ have disjoint supports for $l\neq \tilde l\in\mathcal{C}_j$, it follows that for all $v,\tau$ as above and $k \in \Z^3$
\begin{align*} 
\sum_{j=1}^8|\psi^{(j)}_k(v,\tau)|^2
&=
\sum_{j=1}^8 \sum_{l\in\mathcal{C}_j}\alpha_l(\mu v)^2
=1,
\end{align*}
and
\begin{align} \label{e:phiestimate0}
\sup_{v,\tau}|D^r_v\psi^{(j)}_k(v,\tau)| &\leq C\mu^r,
\quad
j=1,\dots,8,
\end{align}
with the constant $C=C(r,|k|)$ depending only on $r \in \N$ and $|k|$.
More generally, in \cite[Proposition 4.2]{DLS14} it is proved that for any $k \in \Z^3$ the derivatives of $\psi_k^{(j)}$ with respect to $\tau$ are controlled on the set $|v| \leq V$ by
\begin{align}\label{e:phiestimate3}
\sup_{|v| \leq V,\tau}
\left|D^r_v \partial_{\tau}^h\psi^{(j)}_k \right|
&\leq 
CV^h\mu^r,
\end{align}
where $C=C(r,h,|k|)$ and $V>0$ is any given constant.

Next, we recall the following estimate from \cite{DLS13} on the \emph{material derivative} of $\psi_k^{(j)}$.
For any $r \in \N$ and $k \in \Z^3$ there exists a constant $C=C(r,|k|)$ such that for every $j=1,\dots,8$ it holds
\begin{align}\label{e:phiestimate2}
\sup_{v,\tau}
\left|D^r_v(\partial_{\tau}\psi^{(j)}_k+i(k \cdot v) \psi^{(j)}_k )\right|
&\leq 
C\mu^{r-1}
.
\end{align}

\subsection{The energy pumping term} \label{ssec:pump}
Next, set for $t \in \R$
\begin{align*}
\tilde{e}(t) \coloneqq
\frac{1}{3 (2\pi)^3} \left(e (t) \left(1-\delta_{n+1}\right) - \int_{\T^3} |v_\ell(x,t)|^2 \, dx\right).
\end{align*}
\begin{lem} \label{lem:bounds_tilde_e}
There exists a finite constant $C=C(\chi)$ depending only on the mollifier $\chi$ such that the following inequality holds true almost surely for every $t \leq \mathfrak{t}$:
\begin{align*}
\tilde{e}(t) 
&\leq
\frac{\delta_n}{3(2\pi)^3} \left( \tfrac{5}{4}  \overline{e} + C \eta  \left( \overline{e}^{1/2}+1 \right)\right).
\end{align*}
Moreover, if $\eta \leq \frac{\underline{e}}{8C(\overline{e}^{1/2}+1)}$ then we also have almost surely for every $t \leq \mathfrak{t}$:
\begin{align*}
\tilde{e}(t) 
&\geq
\frac{\delta_n \underline{e}}{24(2\pi)^3}.
\end{align*}
\end{lem}  
\begin{proof}
By triangle inequality, we have $||v_\ell|^2-|v_n|^2| \leq |v_\ell-v_n|^2 + 2|v_n||v_\ell-v_n|$; moreover, $|v_\ell-v_n| \leq C D_n \ell \leq C \eta \delta_n$ for every $t \leq \mathfrak{t}$, by assumption on $\ell$.
In the previous inequality, $C=C(\chi)$ is a finite constant depending only on the mollifier $\chi$. 
By assumption \eqref{eq:est_energy} we have $\int_{\mathbb{T}^3}|v_n(x,t)|^2 dx \leq (1-\tfrac{3}{4} \delta_n) e(t) \leq \overline{e}$ almost surely for every $t \leq \mathfrak{t}$, therefore
\begin{align*}
\int_{\mathbb{T}^3} \left| |v_\ell(x,t)|^2-|v_n(x,t)|^2\right| dx
&\leq
C \eta \delta_n \left( \overline{e}^{1/2}+1 \right),
\end{align*}
almost surely for every $t \leq \mathfrak{t}$.
As a consequence we have the estimate
\begin{align*}
\tilde{e}(t) 
&=
\frac{1}{3 (2\pi)^3} \left(e (t) \left(1-\tfrac{5}{4}\delta_n \right) - \int_{\T^3} |v_n|^2 dx + e (t)\left(\tfrac{5}{4}\delta_n-\delta_{n+1}\right) + \int_{\T^3} (|v_n|^2-|v_\ell|^2) \, dx\right)
\\
&\leq 
\frac{\delta_n}{3(2\pi)^3} \left( \tfrac{5}{4}  \overline{e} + C \eta  \left( \overline{e}^{1/2}+1 \right)\right)
\end{align*}
almost surely for every $t \leq \mathfrak{t}$.
In a similar fashion, one can rewrite
\begin{align*}
\tilde{e}(t) 
&=
\frac{1}{3 (2\pi)^3} \left(e (t) \left(1-\tfrac{3}{4}\delta_n \right) - \int_{\T^3} |v_n|^2 dx + e (t)\left(\tfrac{3}{4}\delta_n-\delta_{n+1}\right) + \int_{\T^3} (|v_n|^2-|v_\ell|^2) \, dx\right)
\\
&\geq 
\frac{\delta_n}{3(2\pi)^3} \left( \tfrac{1}{4}  \underline{e} - C \eta  \left( \overline{e}^{1/2}+1 \right)\right),
\end{align*}
where we have used $\tfrac{3}{4}\delta_n - \delta_{n+1} \geq \tfrac{1}{4}\delta_n$ (recall that the parameter $a,b$ are chosen both greater or equal than $2$).
In order to obtain the thesis it is sufficient to take $\eta$ such that
\begin{align*}
\tfrac{1}{4}  \underline{e} - C \eta  \left( \overline{e}^{1/2}+1 \right)
\geq 
\tfrac{1}{8}  \underline{e},
\end{align*}
which is always possible since we have assumed $\underline{e}>0$, and this corresponds exactly to the condition in the statement of the lemma.
\end{proof}

Hereafter, we shall always assume $\eta$ satisfies the condition of previous \autoref{lem:bounds_tilde_e}.
Let $\gamma_n:\R \to (0,\infty)$ be an adapted process with twice-differentiable trajectories satisfying almost surely   
\begin{align} \label{eq:conditions_gamma1}
\gamma_n (t) 
&=
\tilde{e}(t),
\quad \forall t \leq \mathfrak{t},
\\ \label{eq:conditions_gamma2}
\tfrac12 \tilde{e}(\mathfrak{t})
&\leq \gamma_n (t)
\leq \tfrac32 \tilde{e}(\mathfrak{t}),
\quad \forall t \geq \mathfrak{t}, 
\\ \label{eq:conditions_gamma3}
\|\gamma_n\|_{C^k_t} &\leq C \|\tilde{e}_{\leq \mathfrak{t}}\|_{C^k_t}, \quad \forall k \leq 2,
\end{align}
for some unimportant constant $C$, and for some arbitrary $r_0>0$ define
\begin{align*}
\rho_\ell (x,t) &\coloneqq
\frac{2}{r_0} \sqrt{\eta^2 \delta_{n+1}^2 + |\mathring{R}_\ell(x,t)|^2} + \gamma_n(t), 
\\
R_\ell (x,t) &\coloneqq \rho_\ell (x,t) Id - \mathring{R}_\ell (x,t).
\end{align*}

The parameter $r_0$ will be fixed in next \autoref{ssec:w_o} according to a geometric lemma from \cite{DLS13}.
Condition \eqref{eq:conditions_gamma1} above is imposed to reduce the error (up to time $\mathfrak{t}$) between the kinetic energy of $v_{n+1}$ and the desired energy profile $e$.
After time $\mathfrak{t}$, we are no longer able to control $\tilde{e}$ from below for large times (in particular if $t>\mathfrak{t}_L$ with $L$ large), and thus we inject energy in the system according to $\gamma_n(t)$ and not aiming for any prescribed energy profile.
Notice that $\gamma_n$ does not depend on the parameter $L$, in particular we do not use any stopping time of the form $\mathfrak{t}_L$, $L>1$ in the definition of $\gamma_n$, see \autoref{rmk:gamma_n} below.
Conditions \eqref{eq:conditions_gamma2} and \eqref{eq:conditions_gamma3} are needed to show convergence of the so-obtained iterative scheme for times $t \geq \mathfrak{t}$.
Let us also mention that by \eqref{eq:conditions_gamma1}, \eqref{eq:conditions_gamma2} and \autoref{lem:bounds_tilde_e} we have $c \delta_n \leq \gamma_n(t) \leq C \delta_n$ for every $t \in \R$, for some constants $c=c(\underline{e})>0$ and $C=C(\overline{e})<\infty$.

\begin{rmk} \label{rmk:gamma_n}
We point out that it is always possible to construct a process $\gamma_n$ enjoying the properties \eqref{eq:conditions_gamma1}, \eqref{eq:conditions_gamma2} and \eqref{eq:conditions_gamma3}. 
For instance, an explicit construction goes as follows. 
Let $F:\R \to (-1,1)$ be given by $F(t)\coloneqq\frac{2}{\pi}\arctan(t)$.
For every $k\in \{0,1,2\}$ let $\tilde{e}^{(k)}(\mathfrak{t})$ denote the derivative of $\tilde{e}$ of order $k$ evaluated at time $t=\mathfrak{t}$. 
For every $\omega \in \Omega$, let $f=f_\omega$ be the unique polynomial of order $2$ such that $f(\mathfrak{t})=\tilde{e}(\mathfrak{t})$ and 
\begin{align*}
\left.
\frac{\tilde{e}(\mathfrak{t})}{2} \frac{d^k}{dt^k}  F\left( \frac{f-\tilde{e}(\mathfrak{t})}{\|\tilde{e}_{\leq \mathfrak{t}}\|_{C^2_t}} \right) \right|_{t=\mathfrak{t}} = \tilde{e}^{(k)}(\mathfrak{t}), 
\qquad
k=1,2,
\end{align*}
which is uniquely determined by 
\begin{align*}
\left.
\frac{d}{dt} F \left( \frac{f-\tilde{e}(\mathfrak{t})}{\|\tilde{e}_{\leq \mathfrak{t}}\|_{C^2_t}} \right) \right|_{t=\mathfrak{t}}
&=
\frac{F'(0) f'(\mathfrak{t}) }{\|\tilde{e}_{\leq \mathfrak{t}}\|_{C^2_t}},
\\
\left.
\frac{d^2}{dt^2} F \left( \frac{f-\tilde{e}(\mathfrak{t})}{\|\tilde{e}_{\leq \mathfrak{t}}\|_{C^2_t}} \right) \right|_{t=\mathfrak{t}}
&=
\frac{F'(0) f''(\mathfrak{t}) }{\|\tilde{e}_{\leq \mathfrak{t}}\|_{C^2_t}}
+
\frac{F''(0)f'(\mathfrak{t})^2}{\|\tilde{e}_{\leq \mathfrak{t}}\|_{C^2_t}^2}.
\end{align*}
Notice that the coefficients of $f$ are $\mathcal{F}_\mathfrak{t}$-measurable real valued random variables, and defining 
\begin{align*}
\gamma_n(t) \coloneqq 
\begin{cases}
\tilde{e}(t), &\mbox{ if } t \leq \mathfrak{t},
\\
\tilde{e}(\mathfrak{t}) \left( 1 + \frac{1}{2} F\left( \frac{f(t)-\tilde{e}(\mathfrak{t})}{\|\tilde{e}_{\leq \mathfrak{t}}\|_{C^2_t}} \right) \right), &\mbox{ if } t \geq \mathfrak{t},
\end{cases}
\end{align*}
we have a $\gamma_n$ satisfying \eqref{eq:conditions_gamma1}, \eqref{eq:conditions_gamma2} and \eqref{eq:conditions_gamma3} with $C$ depending only on $F$. Moreover, $\gamma_n$ is adapted since $v_n$ (and hence $\tilde{e}$) is so. In particular, $\gamma_n$ is deterministic for negative times.
\end{rmk}

\subsection{The oscillatory term $w_o$} \label{ssec:w_o}
We recall (cf. \cite[Lemma 3.2]{DLS13}) that for every $N\in\N$ there exist $r_0>0$ and $\lambda_0>1$ with the following property: there exist pairwise disjoint subsets 
$
\Lambda_j\subset\{k\in \Z^3:\,|k|=\lambda_0\}$, for 
$
j\in \{1, \dots, N\},
$
and smooth positive functions 
\begin{align*}
\gamma^{(j)}_k\in C^{\infty}\left(B_{r_0} (Id)\right), 
\quad 
j\in \{1,\dots, N\}, \, k\in\Lambda_j,
\end{align*}
such that $k\in \Lambda_j$ implies both $-k\in \Lambda_j$ and $\gamma^{(j)}_k = \gamma^{(j)}_{-k}$, and for each $R\in B_{r_0} (Id)$ we have the identity
\begin{equation*}
R 
= 
\frac12 \sum_{k\in\Lambda_j} 
\left(\gamma^{(j)}_k(R)\right)^2 M_k,
\quad \forall R\in B_{r_0}(Id)\, .
\end{equation*}

We apply the previous lemma with $N=8$ to obtain $\lambda_0>1$, $r_0>0$ and pairwise disjoint families $\Lambda_j$ together with corresponding
functions $\gamma^{(j)}_k\in C^{\infty}\left(B_{r_0}(Id)\right)$. 

Finally, let us define
\begin{align}
w_o(x,t) \nonumber
&\coloneqq
W(x,t,\lambda \phi_{n+1}(x,t),\lambda t),
\\
W(y,s,\xi,\tau) \label{eq:def_W}
&\coloneqq
\sum_{k \in \Lambda}
a_k(y,s,\tau) E_ke^{ik\cdot \xi},
\end{align}
where $\Lambda \coloneqq \cup_{j} \Lambda_j$ and the \emph{amplitude coefficients} $a_k$, $k \in \Lambda$ are defined by
\begin{align} \label{eq:a_k}
a_k(y,s,\tau)
&\coloneqq
\mathbf{1}_{\{k \in \Lambda_j\}} 
\sqrt{\rho_\ell(y,s)}\, 
\gamma_k^{(j)}\left(\frac{R_\ell(y,s)}{\rho_\ell(y,s)}\right)\psi_k^{(j)} \left(\tilde{v}(y,s),\tau\right),
\\
\nonumber
\tilde{v}(y,s)
&\coloneqq 
v_\ell(y,s)
+
\dot{\phi}_{n+1}(y,s).
\end{align}

Notice that the maps $\gamma_k^{(j)}\left(\frac{R_\ell(y,s)}{\rho_\ell(y,s)}\right)$ are well-defined because (recall $\gamma_n \geq c \delta_n \geq 0$)
\begin{align} \label{eq:well_def}
\left\| Id - \frac{R_\ell(y,s)}{\rho_\ell(y,s)} \right\|_{C_t C_x }
=
\left\| \frac{\mathring{R}_\ell(y,s)}{\rho_\ell(y,s)} \right\|_{C_t C_x} \leq \frac{r_0}{2}.
\end{align}

\begin{rmk}
The auxiliary velocity $\tilde{v}$ is needed because, in order to control the transport error below (see \eqref{eq:decomposition_R} and \autoref{prop:R_trans}), we will need cancellations for the quantities
\begin{align*} 
\partial_\tau a_k(x,t,\lambda t) + i k \cdot [v(x,t)+\dot{\phi}_{n+1}(x,t)]\, a_k(x,t,\lambda t), \quad
k \in \Lambda.
\end{align*} 
However, some care is needed since $\dot\phi_{n+1}$ is typically large, in particular it diverges as $\varsigma_{n+1}^{\alpha-1}$ as $n \to \infty$, thus affecting both space and time regularity of the coefficients $a_k$. 
\end{rmk}

Also, as a consequence of \eqref{eq:conditions_gamma1},  \eqref{eq:conditions_gamma2}, \autoref{lem:bounds_tilde_e} and the iterative assumption \eqref{eq:est_Reynolds}, there exists a constant $C$ depending only on $\overline{e}$ such that for every $x \in \mathbb{T}^3$ and $t \leq \mathfrak{t}_L$
\begin{align} \label{eq:w_o_bound}
|w_o(x,t)| 
\leq
C L_n^{1/2} \delta_n^{1/2}.
\end{align}
In particular, from \eqref{eq:w_o_bound} we also deduce
\begin{align*}
|w_o(x,t)| 
\leq
C \delta_n^{1/2},
\quad
\forall t \leq \mathfrak{t}.
\end{align*}

\subsection{The correction $w_c$} \label{ssec:w_c}
Let $\mathcal{P} = Id - \mathcal{Q}$ be the classical Leray projector on zero-average, divergence free velocity fields.
Define $\mathcal{P}^{\phi_{n+1}}, \mathcal{Q}^{\phi_{n+1}}$ as the operators acting on a given $v \in C^\infty(\T^3,\R^3)$ as
\begin{align*}
\mathcal{P}^{\phi_{n+1}} v \coloneqq
[\mathcal{P}(v \circ \phi_{n+1}^{-1})] \circ \phi_{n+1},
\quad
\mathcal{Q}^{\phi_{n+1}} v \coloneqq
[\mathcal{Q}(v \circ \phi_{n+1}^{-1})] \circ \phi_{n+1}.
\end{align*}

It holds
\begin{align*}
\dvgphin \mathcal{P}^{\phi_{n+1}} v
=
[\dvg(\mathcal{P}^{\phi_{n+1}} v \circ \phi_{n+1}^{-1})] \circ \phi_{n+1}
=
[\dvg(\mathcal{P}(v \circ \phi_{n+1}^{-1})] \circ \phi_{n+1}
=
0,
\end{align*}
and 
\begin{align*}
\dvgphin \mathcal{Q}^{\phi_{n+1}} v
=
[\dvg(\mathcal{Q}^{\phi_{n+1}} v \circ \phi_{n+1}^{-1})] \circ \phi_{n+1}
=
[\dvg(\mathcal{Q}(v \circ \phi_{n+1}^{-1})] \circ \phi_{n+1}
=
\dvgphin v.
\end{align*}

Moreover, denoting $\Delta^{\phi_{n+1}} v \coloneqq \dvgphin \nabla^{\phi_{n+1}} v = [\Delta(v \circ \phi_{n+1}^{-1})] \circ \phi_{n+1}$ and introducing the zero-mean solution $\psi$ of the Poisson equation $\Delta^{\phi_{n+1}} \psi = \dvgphin v$ we have the alternative expression
\begin{align*}
\mathcal{Q}^{\phi_{n+1}} v 
= 
\nabla^{\phi_{n+1}} \psi
+
\frac{1}{(2\pi)^3}\int_{\T^3} v,
\end{align*}
provided $v \in C^\infty(\T^3,\R^3)$.

Recall that we want a control on $\dvgphin v_{n+1}$, where $v_{n+1}=v_\ell + w_o + w_c$ and $w_c$ is still to be found. 
Let us compute
\begin{align} \label{eq:div v_ell}
\dvgphin v_\ell
&=
\dvgphin v_\ell - \left( \dvgphin v_n \right) \ast \chi_\ell
\\
&\quad+ \nonumber
\left( \dvgphin v_n \right) \ast \chi_\ell
-
\left(\dvg^{\phi_n}v_n\right)   \ast \chi_\ell
\\
&\quad+ \nonumber
\left(\dvg^{\phi_n}v_n\right) \ast \chi_\ell.
\end{align}
We shall see that the first two terms on the right-hand-side of \eqref{eq:div v_ell} are already small and compatible with the inductive assumption on the divergence \eqref{eq:est_div}, without any compressibility correction needed.
The only term to really ``compensate" for in \eqref{eq:div v_ell} is $\left(\dvg^{\phi_n}v_n\right) \ast \chi_\ell$, and the idea is to compensate with some $w^1_c$ such that $\partial_t w^1_c$ is easy to control.
This happens when $w^1_c$ is a convolution, because one can put time derivatives on the mollification kernel.
Therefore, let us define
\begin{align*}
w^1_c 
\coloneqq 
-(\mathcal{Q}^{\phi_n}v_n) \ast \chi_\ell,
\end{align*}
which is such that 
\begin{align} \label{eq:div v_n+1}
\left(\dvg^{\phi_n}v_n\right) \ast \chi_\ell
+
\dvgphin w^1_c
&=
\left(\dvg^{\phi_n} \mathcal{Q}^{\phi_n} v_n\right) \ast \chi_\ell
-
\dvg^{\phi_n} \left( (\mathcal{Q}^{\phi_n}v_n) \ast \chi_\ell\right)
\\
&\quad+
\dvg^{\phi_n} \left( (\mathcal{Q}^{\phi_n}v_n) \ast \chi_\ell\right)
-
\dvgphin \left( (\mathcal{Q}^{\phi_n}v_n) \ast \chi_\ell\right) \nonumber
\end{align}
can be controlled as the first two lines in \eqref{eq:div v_ell}.
As for the principal perturbation $w_o$, we can take care of its divergence by adding
\begin{align*}
w^2_c
\coloneqq 
-
\mathcal{Q}^{\phi_{n+1}} w_o
\end{align*}
so that $w_o+w^2_c = \mathcal{P}^{\phi_{n+1}} w_o$.
We shall see that there is no problem in controlling $\partial_t w^2_c$ because of the particular geometric structure of the perturbation $w_o$, cf. \autoref{ssec:comp} for details.

The vector field $v_{n+1}$ is then defined as the sum 
\begin{align*}
v_{n+1} 
\coloneqq
v_\ell + w_o + w_c
\coloneqq
v_\ell + w_o + w^1_c + w^2_c,
\end{align*}
and has zero space average by construction.

\subsection{The Reynold stress $\mathring{R}_{n+1}$ and the new pressure $q_{n+1}$} \label{ssec:R&P}

Let us preliminarily recall the following result from \cite{DLS13} on the left inverse of the operator $\mbox{div}$, here adapted to deal with the operator $\dvgphin$.
\begin{lem} \label{lem:inverse}
Let $v \in C^\infty(\T^3,\R^3)$ and $\mathcal{R}v$ be the matrix-valued function defined in \cite[Definition 4.2]{DLS13}, so that $\mathcal{R}v$ takes values in the space of symmetric trace-free matrices and $\dvg\mathcal{R}v = v - \frac{1}{(2\pi)^3}\int_{\T^3}v$.
Then the operator $\mathcal{R}^{\phi_{n+1}}$ defined as
\begin{align*}
\mathcal{R}^{\phi_{n+1}} v \coloneqq
[\mathcal{R}(v \circ \phi_{n+1}^{-1})] \circ \phi_{n+1}
\end{align*} 
satisfies $\dvgphin (\mathcal{R}^{\phi_{n+1}} v) = v - \frac{1}{(2\pi)^3}\int_{\T^3}v$.
\end{lem}

Suppose we are given the new pressure $q_{n+1}$. Since we have defined $v_{n+1} = v_\ell + w_o + w_c$ and we are looking for a solution to the Euler-Reynolds system, we shall choose the new Reynold stress $\mathring{R}_{n+1}$ such that 

\begin{align} \label{eq:R&q}
\mathring{R}_{n+1} \coloneqq
\mathcal{R}^{\phi_{n+1}}
\left( \partial_t v_{n+1} + \dvgphin(v_{n+1} \otimes v_{n+1}) + \nablaphin q_{n+1} \right).
\end{align}

Notice that $\partial_t v_{n+1} + \dvgphin(v_{n+1} \otimes v_{n+1}) + \nablaphin q_{n+1}$ has average zero whatever the choice of $q_{n+1}$. Indeed, $v_{n+1}$ has average zero by construction, and thus so does
$\partial_t v_{n+1}$; on the other hand, the vector field $\dvgphin(v_{n+1} \otimes v_{n+1}) + \nablaphin q_{n+1}$ has average zero because $\phi_{n+1}$ is measure preserving and thus
\begin{align*}
\int_{\T^3}
\dvgphin(v_{n+1} \otimes v_{n+1}) + \nablaphin q_{n+1}
&=
\int_{\T^3}
\dvgphin\left( v_{n+1} \otimes v_{n+1} + q_{n+1} Id \right)
\\
&=
\int_{\T^3}
\dvg \left[\left( v_{n+1} \otimes v_{n+1} + q_{n+1} Id \right)\circ \phi_{n+1}^{-1} \right] = 0.
\end{align*}

Therefore from \eqref{eq:R&q} and \autoref{lem:inverse} it follows that $\mathring{R}_{n+1}$ is symmetric and trace-free, and the following identity holds
\begin{align*}
\partial_t v_{n+1} + \dvgphin(v_{n+1}\otimes v_{n+1})
+ 
\nabla^{\phi_{n+1}} q_{n+1} = \dvgphin \mathring{R}_{n+1};
\end{align*}
otherwise said, $(v_{n+1},q_{n+1},\phi_{n+1},\mathring{R}_{n+1})$ is a solution of the Euler-Reynolds system at level $n+1$.
Let us now see how to choose the new pressure term $q_{n+1}$. 
Recalling
\begin{align*}
\partial_t v_\ell 
+
\chi_\ell\ast \dvg^{\phi_n}\left( v_n \otimes v_n\right) 
+
\chi_\ell\ast \nabla^{\phi_n} q_n 
= 
\chi_\ell\ast \dvg^{\phi_n} \mathring{R}_n,
\end{align*}
and $v_{n+1} = v_\ell + w_o + w_c$ we have for any $\tilde{\rho}_\ell$
\begin{align} \label{eq:decomposition_R}
\partial_t v_{n+1} 
&+ 
\dvgphin(v_{n+1}\otimes v_{n+1})
+
\nabla^{\phi_{n+1}} q_\ell
-
\nabla^{\phi_{n+1}} \frac12\left(|w_o|^2-\tilde{\rho}_\ell \right)
\\
&= \nonumber
\underbrace{\left[\partial_t w_o + \dvgphin(w_o \otimes v_\ell) - w_o\, \dvgphin v_\ell\right]}_{=\,transport\, error}
\\
&\quad+ \nonumber
\underbrace{\left[ \dvgphin(v_\ell \otimes v_\ell - (v_n \otimes v_n) \ast \chi_\ell)\right]}_{=\,mollification\, error I}
\\
&\quad+ \nonumber
\underbrace{\left[ \dvg^{\phi_n}\left( \left( v_n \otimes v_n\right) \ast \chi_\ell  + q_\ell Id - \mathring{R}_\ell  \right)
-
\left( \dvg^{\phi_n} \left( v_n \otimes v_n + q_n Id - \mathring{R}_n \right) \right) \ast \chi_\ell
 \right]}_{=\,mollification\, error\, II}
\\
&\quad+ \nonumber
\underbrace{\left[\dvgphin\left(w_o \otimes w_o - \frac{1}{2}\left(|w_o|^2-\tilde{\rho}_\ell \right) Id +\mathring{R}_\ell\right)\right]}_{=\,oscillation\, error}
\\
&\quad+ \nonumber
\underbrace{\left[\left(\dvgphin-\dvg^{\phi_n}\right)
((v_n \otimes v_n) \ast \chi_\ell - \mathring{R}_\ell + q_\ell Id) + w_o\, \dvgphin v_\ell \right]}_{=\,flow\, error}
\\
&\quad+ \nonumber
\underbrace{\left[\partial_t w_c + \dvgphin(v_{n+1} \otimes w_c + w_c \otimes v_{n+1} - w_c \otimes w_c +v_\ell \otimes w_o) \right]}_{=\,compressibility\, error}.
\end{align}
The oscillation error above can be controlled when we take $\tilde{\rho}_\ell$ equal to the energy pumping term defined in \autoref{ssec:pump}, cf. the computations in \autoref{ssec:osc}.
This suggests the choice
\begin{align*}
q_{n+1} &\coloneqq q_\ell - \frac12 \left(|w_o|^2-\tilde{\rho}_\ell \right),
\\
\tilde{\rho}_\ell(x,t) &\coloneqq \frac{2}{r_0} \sqrt{\eta^2  \delta_{n+1}^2 + |\mathring{R}_\ell(x,t)|^2}.
\end{align*}
Notice that the part with $\gamma_n$ in the definition of $\rho_\ell$ does not appear here since it is $x$-independent. 
Incidentally, in view of \eqref{eq:R&q} and since the operator $\mathcal{R}^{\phi_{n+1}}$ is linear, we can split the Reynolds stress into five parts. Each terms will be controlled separately in \autoref{sec:iter}.

\section{Iterative estimates} \label{sec:iter}

In the present section we collect all the necessary estimates on the several error terms from \eqref{eq:decomposition_R}.
Recall the definition of $\tilde{v}$ \eqref{eq:a_k};
as a preliminary step, we need to give good bounds for H\"older norms of the amplitude coefficients $a_k$, their material derivatives $\partial_\tau a_k + i(k \cdot \tilde{v})a_k$, and time derivatives (i.e. with respect to both the variables $s$ and $\tau$) of both of them.
This is done in forthcoming \autoref{prop:aux_est}, whose proof is given in \autoref{app:proof_prop}. 
As a consequence of that we have \autoref{cor:U}, describing H\"older regularity of the Fourier coefficients (with respect to $\xi \in \T^3$) of the matrix-valued field $W \otimes W$.
 
By interpolation between H\"older spaces $C^r_x$, $r \in [1,r_\star+5]$ the results of \autoref{prop:aux_est} and \autoref{cor:U}, here stated for $r\geq 1$, $r \in \N$ still hold true when $r$ is not an integer.
That will be used in the forthcoming subsections to control, respectively, the transport error (\autoref{ssec:trans}), the oscillation error (\autoref{ssec:osc}), the compressibility error (\autoref{ssec:comp}), as well as to give bounds on the pressure (\autoref{ssec:pres}) and the energy (\autoref{ssec:energy}).

\begin{prop} \label{prop:aux_est}
Let $a_k$ be given by \eqref{eq:a_k}, $k \in \Lambda$, and let $s \leq \mathfrak{t}_L$, $\tau \leq \lambda \mathfrak{t}_L$ be fixed, $L \geq 1$. Then there exists a constant $C$ depending only on $\overline{e},K_0,\eta,M_v$ such that for every $r \in [0,1]$:
\begin{align}
\|a_k(\cdot,s,\tau)\|_{C^r_x} 
&\leq \label{est:a_k,r_small}
C L_n^{7/2} \delta_n^{1/2}  \mu^r \varsigma_{n+1}^{r(\alpha-1)},
\\
\|\partial_\tau a_k(\cdot,s,\tau)\|_{C^r_x} 
&\leq \label{est:Dtau(a_k),r_small}
C L_n^{7/2} \delta_n^{1/2}  \mu^r \varsigma_{n+1}^{(r+1)(\alpha-1)},
\\
\|(\partial_\tau a_k + i(k \cdot \tilde{v})a_k)(\cdot,s,\tau)\|_{C^r_x}
&\leq \label{est:Dmat(a_k),r_small}
C L_n^{7/2} \delta_n^{1/2}  \mu^{r-1} \varsigma_{n+1}^{r(\alpha-1)}
,
\\
\|\partial_s a_k(\cdot,s,\tau)\|_{C_x}
&\leq \label{est:Ds(a_k)0}
CL_n^{7/2} \mu \delta_n^{1/2} \varsigma_{n+1}^{\alpha-2}.
\end{align}
Moreover, there exists a constant $C$ depending only on $\overline{e},K_0,\eta,M_v,r_\star$ such that for every $r \in \N$, $2 \leq r \leq r_\star + 5$:
\begin{align}
\|a_k(\cdot,s,\tau)\|_{C^r_x} 
&\leq \label{est:a_k,r_large}
C L_n^{3r}  \delta_n^{1/2} \mu^r \varsigma_{n+1}^{(r-1)(\alpha-1)} (D_n \ell^{1-r} + \varsigma_{n+1}^{\alpha-1}),
\\
\|\partial_\tau a_k(\cdot,s,\tau)\|_{C^r_x} 
&\leq \label{est:Dtau(a_k),r_large}
C L_n^{3r}  \delta_n^{1/2} \mu^r \varsigma_{n+1}^{r(\alpha-1)} (D_n \ell^{1-r} + \varsigma_{n+1}^{\alpha-1}),
\\
\|(\partial_\tau a_k + i(k \cdot \tilde{v})a_k)(\cdot,s,\tau)\|_{C^r_x}
&\leq \label{est:Dmat(a_k),r_large}
C L_n^{3r}  \delta_n^{1/2} \mu^{r-1} \varsigma_{n+1}^{(r-1)(\alpha-1)} (D_n \ell^{1-r} + \varsigma_{n+1}^{\alpha-1})
,
\\
\|\partial_s a_k(\cdot,s,\tau)\|_{C^r_x}
&\leq \label{est:Ds(a_k)}
C L_n^{3r+5/2} \delta_n^{1/2}  \mu^{r+1} \varsigma_{n+1}^{r(\alpha-1)-1} (D_n \ell^{1-r}+\varsigma_{n+1}^{\alpha-1}).
\end{align}
%
%
%
\end{prop}

Let us denote $\mathcal{S}^{3 \times 3}$ the space of symmetric $3 \times 3$ matrices.
Arguing as in \cite[Proposition 6.1]{DLS13} we can also deduce the following:

\begin{cor} \label{cor:U}
Let $W=W(y,s,\xi,\tau)$ be defined by \eqref{eq:def_W}. Then the matrix-valued field $W \otimes W$ can be written as
\begin{align*}
W \otimes W (y,s,\xi,\tau)
=
R_\ell (y,s) 
+
\sum_{1 \leq |k| \leq 2\lambda_0} U_k(y,s,\tau) e^{i k \cdot \xi},
\end{align*}
with $U_k \in C^\infty_{loc}(\mathbb{T}^3 \times \R^2, \mathcal{S}^{3 \times 3})$, $k \in \Lambda$, satisfying $U_k k = \tfrac{1}{2} Tr(U_k)k$. 
In addition, for every fixed $s \leq \mathfrak{t}_L$, $\tau \leq \lambda \mathfrak{t}_L$, $L \geq 1$, and for every $r \in (0,1]$ there exists a constant $C$ depending only on $\overline{e},K_0,\eta,M_v$ such that:
\begin{align}
\|U_k(\cdot,s,\tau)\|_{C^r_x} 
&\leq \label{est:U_k,r_small}
C L_n^4 \delta_n  \mu^r \varsigma_{n+1}^{r(\alpha-1)}
.
\end{align}
Moreover, there exists $C$ depending on $\overline{e},K_0,\eta,M_v,r_\star$ such that for every $r \in \N$, $2 \leq r \leq r_\star+5$:
\begin{align}
\|U_k(\cdot,s,\tau)\|_{C^r_x} 
&\leq \label{est:U_k,r_large}
C L_n^{3r+5/2}  \delta_n \mu^r \varsigma_{n+1}^{(r-1)(\alpha-1)} (D_n \ell^{1-r} + \varsigma_{n+1}^{\alpha-1})
.
\end{align} 
\end{cor}

In the remaining part of the present section we prove the iterative estimates necessary to the proof of our main \autoref{prop:it}. The underlying rationale consists in giving individual H\"older bounds to the error terms appearing in the decomposition \eqref{eq:decomposition_R}.
Some of these errors - namely, the transport error, the oscillation error and the compressibility error - are controlled via the same stationary phase Lemma (cf. \cite[Proposition 4.4]{DLS14}) previously used in the deterministic setting, or a slight modification of it detailed in \autoref{prop:stat_phase_lem}.
Worth to mention is the presence of the mollification error II and flow error terms, respectively due to the fact that the space-time dependent differential operator $\dvg^{\phi_n}$ does not commute with convolutions and the mollification of the noise $\phi_n$. These terms are very specific to our construction and impose a relatively fast decay of $\varsigma_n$ to be dealt with.
Since we are going to use \autoref{prop:aux_est} and \autoref{cor:U} in the following, the constants $C$ in the remainder of this section may depend on parameters $\overline{e},K_0,\eta,M_v,r_\star$ (without mentioning explicitly).

\subsection{Estimate on the transport error} \label{ssec:trans}

The introduction of the modified Beltrami flows, as well as the exact form of the amplitude coefficients $a_k$, $k \in \N$ in previous subsections, is justified by the following observation. 
Let us rewrite the transport error in \eqref{eq:decomposition_R} as
\begin{align*}
\partial_t w_o + \dvgphin(w_o \otimes v_\ell) - w_o\, \dvgphin v_\ell
&=
\partial_t w_o + (v_\ell \cdot \nabla^{\phi_{n+1}}) w_o.
\end{align*}
Denoting $\Omega_k (\xi) \coloneqq E_k e^{ik\cdot \xi}$ and recalling the expression of $w_o$:
\begin{align*}
w_o(x,t) 
&= 
\sum_{k \in \Lambda} a_k(x,t,\lambda t)
\,
\Omega_k (\lambda \phi_{n+1}(x,t)),
\end{align*}
we can compute explicitly
\begin{align*}
\partial_t w_o (x,t)
&=
\sum_{k \in \Lambda}
\partial_s a_k(x,t,\lambda t)
\,
\Omega_k(\lambda \phi_{n+1}(x,t))
\\
&\quad
+
\lambda \sum_{k \in \Lambda}
\partial_\tau a_k(x,t,\lambda t)
\,
\Omega_k(\lambda \phi_{n+1}(x,t))
\\
&\quad
+
\lambda \sum_{k \in \Lambda}
(i k \cdot \dot{\phi}_{n+1}(x,t)) a_k (x,t,\lambda t)
\,
\Omega_k(\lambda \phi_{n+1}(x,t)),
\end{align*}
and
\begin{align*}
(v_\ell \cdot \nabla^{\phi_{n+1}}) w_o (x,t)
&=
\sum_{k \in \Lambda}
v_\ell(x,t) \cdot \nabla^{\phi_{n+1}}a_k (x,t,\lambda t)
\,\Omega_k(\lambda \phi_{n+1}(x,t))
\\
&\quad +
\lambda \sum_{k \in \Lambda}
(i k \cdot v_\ell(x,t)) a_k(x,t,\lambda t)\,\Omega_k(\lambda \phi_{n+1}(x,t)).
\end{align*}
Therefore, recalling the definition of $\tilde{v} \coloneqq v_\ell + \dot{\phi}_{n+1}$ we have
\begin{align*}
\partial_t w_o + (v_\ell \cdot \nabla^{\phi_{n+1}}) w_o
&=
\lambda
\sum_{k \in \Lambda}
\left( \partial_\tau a_k + i (k \cdot \tilde{v}) a_k \right) 
\Omega_k
+
\sum_{k \in \Lambda}
(v_\ell \cdot \nabla^{\phi_{n+1}}a_k + \partial_s a_k ) \Omega_k
,
\end{align*}
where in the expression above the left-hand-side is evaluated at $(x,t)$, whereas the right-hand-side is evaluated at the quadruple $(y,s,\xi,\tau)=(x,t, \lambda\phi_{n+1}(x,t),\lambda t)$.

\begin{prop} \label{prop:R_trans}
Let us denote $\mathring{R}^{trans} \coloneqq \mathcal{R}^{\phi_{n+1}} (\partial_t w_o + (v_\ell \cdot \nabla^{\phi_{n+1}}) w_o) $. Then for every $r \geq r_\star+2$ and $\delta>0$ sufficiently small there exists a constant $C$ such that almost surely for every $L \in \N$, $L \geq 1$
\begin{align*}
\|\mathring{R}^{trans}\|_{C_{\leq \mathfrak{t}_L} C_x} 
&\leq
CL_n^{3r+5}\lambda^{\delta} \mu^{-1} \delta_n^{1/2},
\\
\|\mathring{R}^{trans}\|_{C_{\leq \mathfrak{t}_L} C^1_x} 
&\leq
C L_n^{3r+8} \lambda^{\delta} \mu \delta_n^{1/2} \varsigma_{n+1}^{\alpha-1}
.
\end{align*}
\end{prop}
\begin{proof}
It is convenient to divide $\mathring{R}^{trans}$ into three terms:
\begin{align*}
\mathring{R}^{trans} = \mathring{R}^{trans}_1 + \mathring{R}^{trans}_2 + \mathring{R}^{trans}_3,
\end{align*}
where, denoting $b_k \coloneqq (\partial_\tau a_k + i(k\cdot \tilde{v})a_k) \circ \phi_{n+1}^{-1}$ and $\Omega_k^\lambda \coloneqq \Omega_k(\lambda\, \cdot)$, with a slight abuse of notation we have denoted:
\begin{align*}
\mathring{R}^{trans}_1  
&\coloneqq
\lambda \mathcal{R}^{\phi_{n+1}} \left( \sum_{k \in \Lambda} 
(b_k \,\Omega_k^\lambda) \circ \phi_{n+1} \right),
\\
\mathring{R}^{trans}_2  
&\coloneqq
\mathcal{R}^{\phi_{n+1}} \left( \sum_{k \in \Lambda} 
(v_\ell \circ \phi_{n+1}^{-1} \cdot \nabla_y (a_k\circ \phi_{n+1}^{-1}) \,\Omega_k^\lambda) \circ \phi_{n+1} \right),
\\
\mathring{R}^{trans}_3  
&\coloneqq
\mathcal{R}^{\phi_{n+1}} \left( \sum_{k \in \Lambda} 
(\partial_s a_k \circ \phi_{n+1}^{-1} \,\Omega_k^\lambda) \circ \phi_{n+1} \right).
\end{align*}

Let us control each term separately, starting from $\mathring{R}^{trans}_1$.
\autoref{prop:aux_est} will be implicitly used throughout this section.
Applying \autoref{lem:C0_Cm} and \autoref{prop:stat_phase_lem} with $r \geq r_\star+2$ and $\delta$ sufficiently small, we obtain for every $L \in \N$, $L \geq 1$
\begin{align*}
\| \mathring{R}^{trans}_1\|_{C_{\leq \mathfrak{t}_L}C^\delta_x}
&\leq
CL^\delta  \sum_{k \in \Lambda} \left( 
\lambda^{\delta} \| b_k \|_{C_x}
+
\lambda^{1+\delta-r} [b_k]_{C^r_x}
+
\lambda^{1-r} [b_k]_{C^{r+\delta}_x} \right)
\\
&\leq 
C L^\delta L_n^{7/2} \lambda^{\delta} \mu^{-1} \delta_n^{1/2}  
\\
&\quad+
C L^{r+\delta} L_n^{3r} \lambda^{1+\delta-r} \mu^{r-1}
\delta_n^{1/2}  \varsigma_{n+1}^{(r-1)(\alpha-1)} (D_n \ell^{1-r} + \varsigma_{n+1}^{\alpha-1})
\\
&\quad+
C L^{r+2\delta} L_n^{3r+3\delta} \lambda^{1-r} \mu^{r+\delta-1}  
\delta_n^{1/2} \varsigma_{n+1}^{(r+\delta-1)(\alpha-1)} (D_n \ell^{1-r-\delta} + \varsigma_{n+1}^{\alpha-1})
\\
&\leq
C L^{r+2\delta} L_n^{3r+3\delta}
\lambda^{\delta} \mu^{-1} \delta_n^{1/2},
\end{align*}
where we have used the relation $\lambda^{r-2}
\geq
\mu^r \varsigma_{n+1}^{r(\alpha-1)} (D_n \ell^{-r} + \varsigma_{n+1}^{\alpha-1})$ coming from \eqref{eq:ass_lambda} and the choice of $r$.

To estimate the $C^1_x$ norm, we need to take the space derivative of $\mathring{R}^{trans}_1$ in the first place.
With a slight abuse on notation, we can write
\begin{align} \label{eq:aux_R_1}
\mathring{R}^{trans}_1
&= 
\lambda \mathcal{R}^{\phi_{n+1}} \left( \sum_{k \in \Lambda} 
(b_k \Omega_k^\lambda) \circ \phi_{n+1} \right)
=
\lambda \left( \mathcal{R} \sum_{k \in \Lambda} b_k \Omega_k^\lambda \right) \circ \phi_{n+1},
\end{align}
where the second identity comes from the very definition of $\mathcal{R}^{\phi_{n+1}}$, $\mathcal{R}$ being the inverse divergence operator defined in \cite{DLS13}.
Therefore, by \autoref{lem:C0_Cm} the following inequality holds true:
\begin{align*}
\| \mathring{R}^{trans}_1 \|_{C_{\leq \mathfrak{t}_L}C^1_x}
\leq
C L \lambda \left\| \mathcal{R} \sum_{k \in \Lambda} b_k \Omega_k \right\|_{C_{\leq \mathfrak{t}_L}C^1_x},
\end{align*}
and since $\mathcal{R}$ commutes with every directional derivative\footnote{\label{note:Q}To see this, use the explicit formula from \cite{DLS13}
\begin{align*}
\mathcal{R}v 
= 
\frac14 \left( \nabla \mathcal{P} u + (\nabla \mathcal{P} u)^T \right)
+
\frac34 \left( \nabla u + (\nabla u)^T \right)
-
\frac12 (\dvg u) Id,
\quad
v \in C^\infty(\mathbb{T}^3,\R^3),
\end{align*}
where $\mathcal{P} = Id - \mathcal{Q}$ is the Leray projector and $u \in C^\infty(\mathbb{T}^3,\R^3)$ is the unique zero-average solution of $\Delta u = v - (2\pi)^{-3}\int_{\mathbb{T}^3} v$.
The operator $\mathcal{Q}$ commutes with any directional space derivative since $\mathcal{Q} v = \nabla A + (2\pi)^{-3}\int_{\mathbb{T}^3} v$ for every $v \in C^\infty(\mathbb{T}^3,\R^3)$, where $A \in C^\infty(\mathbb{T}^3)$ is the unique zero-average solution of $\Delta A = \dvg v$.   
} $\partial_{x_i}$, $i=1,2,3$:
\begin{align*}
\partial_{x_i} \mathcal{R} \sum_{k \in \Lambda} b_k \Omega_k^\lambda
&=
\mathcal{R} \sum_{k \in \Lambda} \partial_{x_i}(b_k \Omega_k^\lambda)
=
\mathcal{R} \sum_{k \in \Lambda} (\partial_{x_i}b_k) \Omega_k^\lambda
+
i \lambda k_i \mathcal{R} \sum_{k \in \Lambda} b_k \Omega_k^\lambda,
\end{align*}
we can use the stationary phase Lemma \cite[Proposition 4.4]{DLS14} to get
\begin{align*}
\| \mathring{R}^{trans}_1\|_{C_{\leq \mathfrak{t}_L}C^1_x}
&\leq
CL \lambda\|\mathring{R}^{trans}_1 \|_{C_{\leq \mathfrak{t}_L}C_x}
+
CL \lambda \max_{i=1,2,3}\left\| \mathcal{R} \sum_{k \in \Lambda} (\partial_{x_i} b_k) \Omega_k^\lambda \right\|_{C_{\leq \mathfrak{t}_L}C_x}
\\
&\leq
CL \lambda\|\mathring{R}^{trans}_1 \|_{C_{\leq \mathfrak{t}_L}C_x}
\\
&\quad+
CL \sum_{k \in \Lambda} \left( 
\lambda^{\delta} [ b_k ]_{C^1_x}
+
\lambda^{1+\delta-r} [b_k]_{C^{r+1}_x}
+
\lambda^{1-r} [b_k]_{C^{r+1+\delta}_x} \right)
\\
&\leq 
C L^{r+1+2\delta} L_n^{3r+3\delta}
\lambda^{\delta} \mu^{-1} \delta_n^{1/2}
\\
&\quad+
CL^2L_n^3 \lambda^{\delta} \delta_n^{1/2} \varsigma_{n+1}^{\alpha-1}
+
CL^{r+2}L_n^{3r+3} \lambda^{1+\delta-r} \mu^{r} \delta_n^{1/2}  \varsigma_{n+1}^{r(\alpha-1)} (D_n \ell^{-r} + \varsigma_{n+1}^{\alpha-1})
\\
&\quad+
CL^{r+2+\delta}L_n^{3r+3\delta+3} \lambda^{1-r}\mu^{r+\delta} \delta_n^{1/2}  \varsigma_{n+1}^{(r+\delta)(\alpha-1)} (D_n \ell^{-r-\delta} + \varsigma_{n+1}^{\alpha-1})
\\
&\leq
CL^{r+2+\delta}L_n^{3r+3\delta+3} \lambda^{\delta} \delta_n^{1/2} \varsigma_{n+1}^{\alpha-1}.
\end{align*}

Let us move to the term $\mathring{R}^{trans}_2$.
Spatial H\"older norms are dealt with using again the stationary phase Lemma. We have, omitting details for the sake of brevity (estimates on $a_k$ are as usual given by \autoref{prop:aux_est})
\begin{align*}
\| \mathring{R}^{trans}_2 \|_{C_{\leq \mathfrak{t}_L}C^\delta_x}
&\leq
CL^{1+\delta}  \sum_{k \in \Lambda} \left( 
\lambda^{\delta-1} \| v_\ell \cdot \nabla_y a_k \|_{C_x}
+
L^r\lambda^{\delta-r} [v_\ell \cdot \nabla_y a_k]_{C^r_x}
+
L^{r+\delta}\lambda^{-r} [v_\ell \cdot \nabla_y a_k]_{C^{r+\delta}_x} \right)
\\
&\leq
C L^{r+2+2\delta}L_n^{3r+3\delta+4} \lambda^{\delta-1} \mu \delta_n^{1/2} \varsigma_{n+1}^{\alpha-1},
\end{align*}
as well as (the space derivative is taken as for the term $\mathring{R}^{trans}_1$)
\begin{align*}
\| \mathring{R}^{trans}_2\|_{C_{\leq \mathfrak{t}_L}C^1_x}
&\leq
CL^2 \lambda \| \mathring{R}^{trans}_2 \|_{C_{\leq \mathfrak{t}_L}C_x}
\\
&\quad+
C L^{2+\delta}\sum_{k \in \Lambda} \left( 
\lambda^{\delta-1} [v_\ell \cdot \nabla_y a_k]_{C^1_x}
+
L^r \lambda^{\delta-r} [v_\ell \cdot \nabla_y a_k]_{C^{r+1}_x}
+
L^{r+\delta}\lambda^{-r} [v_\ell \cdot \nabla_y a_k]_{C^{r+1+\delta}_x} \right)
\\
&\leq 
C L^{r+4+2\delta}L_n^{3r+3\delta+7} \lambda^{\delta} \mu \delta_n^{1/2} \varsigma_{n+1}^{\alpha-1}.
\end{align*}

We have used relations $\lambda^{r-1}
\geq 
\mu^{r+3} \varsigma_{n+1}^{(r+3)(\alpha-1)} (D_n \ell^{-r-2} + \varsigma_{n+1}^{\alpha-1})$ and $\lambda
\geq 
\mu (D_n \ell^{-1} + \varsigma_{n+1}^{\alpha-1})$ to highlight one single term on the right-hand-side of each of the previous inequalities.

Finally, for $\mathring{R}^{trans}_3$ we have
\begin{align*}
\| \mathring{R}^{trans}_3 \|_{C_{\leq \mathfrak{t}_L}C^\delta_x}
&\leq
CL^\delta \sum_{k \in \Lambda} \left(
\lambda^{\delta-1} \| \partial_s a_k \|_{C_x}
+
L^r \lambda^{\delta-r} [\partial_s a_k]_{C^r_x}
+
L^{r+\delta}\lambda^{-r} [\partial_s a_k]_{C^{r+\delta}_x} \right)
\\
&\leq 
C L^{r+2\delta} L_n^{3r+3\delta+5/2} \lambda^{\delta-1} \mu \delta_n^{1/2} \varsigma_{n+1}^{\alpha-2},
\end{align*}
and 
\begin{align*}
\| \mathring{R}^{trans}_3 \|_{C_{\leq \mathfrak{t}_L}C^1_x}
&\leq
CL \lambda \| \mathring{R}^{trans}_3 \|_{C_{\leq \mathfrak{t}_L}C_x}
\\
&\quad+
CL^{1+\delta} \sum_{k \in \Lambda} \left(
\lambda^{\delta-1} [ \partial_s a_k ]_{C^1_x}
+
L^r \lambda^{\delta-r} [\partial_s a_k]_{C^{r+1}_x}
+
L^{r+\delta} \lambda^{-r} [\partial_s a_k]_{C^{r+\delta+1}_x} \right)
\\
&\leq 
C L^{r+1+2\delta} L_n^{3r+3\delta+11/2} \lambda^{\delta-1}  \mu^{2} \delta_n^{1/2}  \varsigma_{n+1}^{2\alpha-3}.
\end{align*}

Let us recollect what we have proved. For $\|\mathring{R}^{trans}\|_{C_{\leq \mathfrak{t}_L}C_x}$ we have for $\delta$ sufficiently small and $m \geq 2r$
\begin{align*}
\|\mathring{R}^{trans}\|_{C_{\leq \mathfrak{t}_L}C_x}
&\leq
C L^{r+2\delta} L_n^{3r+3\delta}
\lambda^{\delta} \mu^{-1} \delta_n^{1/2}
+
C L^{r+2+2\delta}L_n^{3r+3\delta+4} \lambda^{\delta-1} \mu \delta_n^{1/2} \varsigma_{n+1}^{\alpha-1}
\\
&\quad+
C L^{r+2\delta} L_n^{3r+3\delta+5/2} \lambda^{\delta-1} \mu \delta_n^{1/2} \varsigma_{n+1}^{\alpha-2}
\\
&\leq
CL_n^{3r+5}\lambda^{\delta} \mu^{-1} \delta_n^{1/2},
\end{align*}
where we have used the relation $\lambda
\geq 
\mu^2  \varsigma_{n+1}^{\alpha-2}$.
Moreover, the same inequality also implies 
\begin{align*}
\|\mathring{R}^{trans}\|_{C_{\leq \mathfrak{t}_L}C^1_x}
&\leq
CL^{r+2+\delta}L_n^{3r+3\delta+3} \lambda^{\delta} \delta_n^{1/2} \varsigma_{n+1}^{\alpha-1}
+
C L^{r+4+2\delta}L_n^{3r+3\delta+7} \lambda^{\delta} \mu \delta_n^{1/2} \varsigma_{n+1}^{\alpha-1}
\\
&\quad+
C L^{r+1+2\delta} L_n^{3r+3\delta+11/2} \lambda^{\delta-1}  \mu^{2} \delta_n^{1/2}  \varsigma_{n+1}^{2\alpha-3}
\\
&\leq
C L_n^{3r+8} \lambda^{\delta} \mu \delta_n^{1/2} \varsigma_{n+1}^{\alpha-1}.
\end{align*}
\end{proof}

\subsection{Estimate on the mollification error} \label{ssec:moll_error}

The mollification error in \eqref{eq:decomposition_R} is divided into two contributions.
The first one is due to the fact that we have replaced $v_n$ with $v_\ell$ in the construction of $v_{n+1}$, but convolution with $\chi_\ell$ does not commute with the tensor product.
The second one, instead, comes from the fact that the space-time differential operator $\dvg^{\phi_n}$ does not commute with the convolution with $\chi_\ell$, which is strikingly different from what happens in the deterministic case.

Let us denote $\mathring{R}^{moll} \coloneqq \mathring{R}^{moll}_1 +\mathring{R}^{moll}_2$, where we define
\begin{align*}
\mathring{R}^{moll}_1 
&\coloneqq  
\mathcal{R}^{\phi_{n+1}}\dvgphin(v_\ell \otimes v_\ell - (v_n \otimes v_n) \ast \chi_\ell),
\\
\mathring{R}^{moll}_2
&\coloneqq
\mathcal{R}^{\phi_{n+1}}
\left(
\dvg^{\phi_n}\left( \left( v_n \otimes v_n\right) \ast \chi_\ell  + q_\ell Id - \mathring{R}_\ell  \right)
-
\left( \dvg^{\phi_n} \left( v_n \otimes v_n + q_n Id - \mathring{R}_n \right) \right) \ast \chi_\ell
\right).
\end{align*}

To better control the second term we will need the following.
\begin{lem} \label{lem:commuting_moll}
There exists a constant $C$ depending only on $K_0$ and $\chi$ with the following property.
Let $n \in \N$ and $L \in\N$, $L\geq 1$ be fixed and let $f:\T^3 \times \R \to \R^3$ be of class $C_{\leq \mathfrak{t}_L} C^1_x$.
Then, denoting $G \coloneqq \dvg^{\phi_n} (f\ast \chi_\ell)
- (\dvg^{\phi_n} f)\ast \chi_\ell  $
it holds
\begin{align*}
\| G \|_{C_{\leq \mathfrak{t}_L}C_x}
&\leq
CL^2 \|f\|_{C_{\leq \mathfrak{t}_L}C^1_x} \ell^{\alpha},
\\
\| G \|_{C_{\leq \mathfrak{t}_L}C^1_x}
&\leq
CL^2 \|f\|_{C_{\leq \mathfrak{t}_L}C^1_x} \ell^{-1}.
\end{align*}
\end{lem}

\begin{proof}
For fixed $x \in \T^3$, $t\in \R$ we have (here we denote $\mbox{div}_y$, $\partial_{y_k}$ etc. the derivatives with respect to the space variable)
\begin{align*}
\left( \dvg^{\phi_n} f \right) (x,t)
&\coloneqq
\left( \dvg_y(f(\phi_{n}^{-1}(y,t),t)) \right) \mid_{y=\phi_n(x,t)}
\\
&=
\left( \sum_{k=1}^3 \partial_{y_k}(f(\phi_{n}^{-1}(y,t),t)) \right) \mid_{y=\phi_n(x,t)}
\\
&=
\left( \sum_{k,j=1}^3 (\partial_{y_j}f)(\phi_{n}^{-1}(y,t),t) (\partial_{y_k} (\phi_n^{-1})^j)(y,t) \right) \mid_{y=\phi_n(x,t)}
\\
&=
\sum_{k,j=1}^3 (\partial_{y_j}f)(x,t) (\partial_{y_k} (\phi_n^{-1})^j)(\phi_n(x,t),t) .
\end{align*}

Therefore we can compute
\begin{align*}
\left( (\dvg^{\phi_n} f)\ast \chi_\ell \right) (x,t)
&=
\int_{\T^3 \times \R} \left( \dvg^{\phi_n} f \right) (z,s)
\chi_\ell(x-z,t-s) dzds
\\
&=
\int_{\T^3 \times \R} 
\sum_{k,j=1}^3 (\partial_{y_j}f)(z,s) (\partial_{y_k} (\phi_n^{-1})^j)(\phi_n(z,s),s) 
\chi_\ell(x-z,t-s) dzds,
\end{align*}
whereas on the other hand
\begin{align*}
\left( \dvg^{\phi_n} (f\ast \chi_\ell) \right) (x,t)
&=
\sum_{k,j=1}^3 (\partial_{y_j}(f\ast \chi_\ell))(x,t) (\partial_{y_k} (\phi_n^{-1})^j)(\phi_n(x,t),t)
\\
&=
\sum_{k,j=1}^3 (\partial_{y_j}f\ast \chi_\ell)(x,t) (\partial_{y_k} (\phi_n^{-1})^j)(\phi_n(x,t),t)
\\
&=
\int_{\T^3 \times \R} \sum_{k,j=1}^3 (\partial_{y_j}f)(z,s) (\partial_{y_k} (\phi_n^{-1})^j)(\phi_n(x,t),t)
\chi_\ell(x-z,t-s) dzds. 
\end{align*}
Thus, it holds for every $x\in \T^3$ and $t \in \R$
\begin{align*}
| G(x,t) |
&\leq
\int_{\T^3 \times \R} \sum_{k,j=1}^3 
|(\partial_{y_j}f)(z,s)| 
\left| (\partial_{y_k} (\phi_n^{-1})^j)(\phi_n(x,t),t)
-(\partial_{y_k} (\phi_n^{-1})^j)(\phi_n(z,s),s) \right|
\chi_\ell(x-z,t-s) dzds
\\
&\leq
C 
\|f\|_{C_{\leq \mathfrak{t}_L}C^1_x}
\int_{\T^3 \times \R} \sum_{k,j=1}^3 
\left| (\partial_{y_k} (\phi_n^{-1})^j)(\phi_n(x,t),t)
-(\partial_{y_k} (\phi_n^{-1})^j)(\phi_n(z,t),t) \right|
\chi_\ell(x-z,t-s) dzds
\\
&\quad+
C 
\|f\|_{C_{\leq \mathfrak{t}_L}C^1_x}
\int_{\T^3 \times \R} \sum_{k,j=1}^3 
\left| (\partial_{y_k} (\phi_n^{-1})^j)(\phi_n(z,t),t)
-(\partial_{y_k} (\phi_n^{-1})^j)(\phi_n(z,s),s) \right|
\chi_\ell(x-z,t-s) dzds
\\
&\leq
C 
\|f\|_{C_{\leq \mathfrak{t}_L}C^1_x}
\int_{\T^3 \times \R} \sum_{k,j=1}^3 
\left\| (\partial_{y_k} (\phi_n^{-1})^j) \circ \phi_n 
\right\|_{C_{\leq \mathfrak{t}_L}C^1_x} |x-z|
\chi_\ell(x-z,t-s) dzds
\\
&\quad+
C 
\|f\|_{C_{\leq \mathfrak{t}_L}C^1_x}
\int_{\T^3 \times \R} \sum_{k,j=1}^3 
\left\| (\partial_{y_k} (\phi_n^{-1})^j) \circ \phi_n 
\right\|_{C^{\alpha}_{\leq \mathfrak{t}_L}C_x} |t-s|^\alpha
\chi_\ell(x-z,t-s) dzds.
\end{align*}

By \autoref{lem:C0_Cm}, \autoref{lem:Cbeta_C0} and \autoref{lem:flow} we have
\begin{align*}
\left\| (\partial_{y_k} (\phi_n^{-1})^j) \circ \phi_n 
\right\|_{C_{\leq \mathfrak{t}_L}C^1_x}
&\leq
CL
\left\| \partial_{y_k} (\phi_n^{-1})^j 
\right\|_{C_{\leq \mathfrak{t}_L}C^1_x}
\leq
CL
\left\| \phi_n^{-1} 
\right\|_{C_{\leq \mathfrak{t}_L}C^2_x}
\leq
CL^2,
\\
\left\| (\partial_{y_k} (\phi_n^{-1})^j) \circ \phi_n 
\right\|_{C^{\alpha}_{\leq \mathfrak{t}_L}C_x} 
&\leq
CL
\left\| \partial_{y_k} (\phi_n^{-1})^j
\right\|_{C_{\leq \mathfrak{t}_L}C^1_x}
+
\left\| \partial_{y_k} (\phi_n^{-1})^j 
\right\|_{C^{\alpha}_{\leq \mathfrak{t}_L}C_x}
\\
&\leq
CL
\left\| \phi_n^{-1}
\right\|_{C_{\leq \mathfrak{t}_L}C^2_x}
+
\left\| \phi_n^{-1}
\right\|_{C^{\alpha}_{\leq \mathfrak{t}_L}C^1_x}
\leq
CL^2,
\end{align*}
and since $|x-z|,|t-s| \leq \ell$ in the support of $\chi_\ell$ and $\int \chi_\ell =1 $ we get the first inequality.

Let us bound $G$ in $C_{\leq \mathfrak{t}_L}C^1_x$.
It holds for every $x,y \in \T^3$ and $t \in \R$
\begin{align*}
G(x,t)&-G(y,t)
\\
&=
\int_{\T^3 \times \R} \sum_{k,j=1}^3 
(\partial_{y_j}f)(z,s)
\left((\partial_{y_k} (\phi_n^{-1})^j)(\phi_n(x,t),t)
-(\partial_{y_k} (\phi_n^{-1})^j)(\phi_n(z,s),s) \right)
\chi_\ell(x-z,t-s) dzds
\\
&\quad-
\int_{\T^3 \times \R} \sum_{k,j=1}^3 
(\partial_{y_j}f)(z,s)
\left((\partial_{y_k} (\phi_n^{-1})^j)(\phi_n(y,t),t)
-(\partial_{y_k} (\phi_n^{-1})^j)(\phi_n(z,s),s) \right)
\chi_\ell(y-z,t-s) dzds
\\
&=
\int_{\T^3 \times \R} \sum_{k,j=1}^3 
(\partial_{y_j}f)(z,s)
\left((\partial_{y_k} (\phi_n^{-1})^j)(\phi_n(x,t),t)
-(\partial_{y_k} (\phi_n^{-1})^j)(\phi_n(y,t),t) \right)
\chi_\ell(x-z,t-s) dzds
\\
&\quad-
\int_{\T^3 \times \R} \sum_{k,j=1}^3 
(\partial_{y_j}f)(z,s)
\left((\partial_{y_k} (\phi_n^{-1})^j)(\phi_n(y,t),t)
-(\partial_{y_k} (\phi_n^{-1})^j)(\phi_n(z,s),s) \right)
\\
&\qquad\qquad\qquad \times
(\chi_\ell(y-z,t-s)-\chi_\ell(x-z,t-s) )dzds,
\end{align*}
implying, since the measure of the support of $\chi_\ell(y-z,t-s)-\chi_\ell(x-z,t-s)$ is of order $\ell^4$, 
\begin{align*}
|G(x,t)-G(y,t)|
&\leq
C\|f\|_{C_{\leq \mathfrak{t}_L}C^1_x} 
\left\| (\partial_{y_k} (\phi_n^{-1})^j) \circ \phi_n
\right\|_{C_{\leq \mathfrak{t}_L}C^1_x} |x-y|
\\
&\quad+
C\|f\|_{C_{\leq \mathfrak{t}_L}C^1_x} 
\left\| (\partial_{y_k} (\phi_n^{-1})^j) \circ \phi_n
\right\|_{C_{\leq \mathfrak{t}_L}C_x}
\| \chi_\ell \|_{C_{\leq \mathfrak{t}_L}C^1_x} \ell^4 |x-y|
\\
&\leq
CL^2\|f\|_{C_{\leq \mathfrak{t}_L}C^1_x} |x-y|
+
CL\|f\|_{C_{\leq \mathfrak{t}_L}C^1_x} \ell^{-1} |x-y|.
\end{align*}
Thus $[ G ]_{C_{\leq \mathfrak{t}_L}C^1_x}
\leq
CL^2 \|f\|_{C_{\leq \mathfrak{t}_L}C^1_x} \ell^{-1}$ and the proof is complete.
\end{proof}

We are now ready to prove the following:
\begin{prop} \label{prop:R_moll}
For every $\delta \in (0,1)$ there exists a constant $C$ such that almost surely for every $L \in \N$, $L \geq 1$
\begin{align*}
\|\mathring{R}^{moll}\|_{C_{\leq \mathfrak{t}_L} C_x}
&\leq
C L_n^3 D_n \ell^{\alpha} ,
\\
\|\mathring{R}^{moll}\|_{C_{\leq \mathfrak{t}_L} C^1_x}
&\leq
C L_n^3 D_n \ell^{-\delta}
.
\end{align*}
\end{prop}

\begin{proof}
By \eqref{e:Schauder_Rdiv} for every $\delta>0$ there exists a constant $C$ such that almost surely 
\begin{align*}
\|\mathring{R}^{moll}_1\|_{C_{\leq \mathfrak{t}_L} C_x}
&\leq
CL^{2\delta} \|v_\ell \otimes v_\ell - (v_n \otimes v_n) \ast \chi_\ell\|_{C_{\leq \mathfrak{t}_L} C^{\delta}_x}
\\
&\leq
CL^{2\delta} \ell^{1-\delta} \|v_n \otimes v_n\|_{C^1_{\leq \mathfrak{t}_L,x}}
\\
&\leq
CL^{2\delta}\ell^{1-\delta} \|v_n\|_{C_{\leq \mathfrak{t}_L}C_x}\|v_n\|_{C^1_{\leq \mathfrak{t}_L,x}}
\\
&\leq
CL^{2\delta}L_n^2 \ell^{1-\delta} D_n,
\\
\|\mathring{R}^{moll}_1\|_{C_{\leq \mathfrak{t}_L} C^1_x}
&\leq
CL^{2+2\delta} \|v_\ell \otimes v_\ell - (v_n \otimes v_n) \ast \chi_\ell\|_{C_{\leq \mathfrak{t}_L}C^{1+\delta}_x}
\\
&\leq
CL^{2+2\delta} \ell^{-\delta} \|v_n \otimes v_n\|_{C^1_{\leq \mathfrak{t}_L,x}}
\\
&\leq
CL^{2+2\delta}L_n^2 \ell^{-\delta} D_n .
\end{align*}
To control $\mathring{R}^{moll}_2$, we apply \autoref{lem:commuting_moll} with $f\coloneqq v_n \otimes v_n + q_n Id - \mathring{R}_n$ (or rather the rows of the same matrix field) and use \eqref{e:Schauder_R_Besov}, the inequality $\|G\|_{C_{\leq \mathfrak{t}_L} B^{\beta/\alpha-1}_{\infty,\infty}} \leq C \|G\|_{C_{\leq \mathfrak{t}_L} C_x}$ (see \cite[Remark A.3]{MoWe17}), and \eqref{e:Schauder_R} 
respectively to get
\begin{align*}
\| \mathring{R}^{moll}_2 \|_{C_{\leq \mathfrak{t}_L}C^\delta_x}
&\leq
CL^{5+4\delta} \|G\|_{C_{\leq \mathfrak{t}_L}B^{\delta-1}_{\infty,\infty}}
\leq
CL^{5+4\delta} \|G\|_{C_{\leq \mathfrak{t}_L}C_x}
\leq
CL^{7+4\delta} L_n^2 D_n \ell^{\alpha},
\\
\| \mathring{R}^{moll}_2 \|_{C_{\leq \mathfrak{t}_L}C^{1+\delta}_x}
&\leq
CL^{1+2\delta} \|G\|_{C_{\leq \mathfrak{t}_L}C^\delta_x}
\leq
CL^{3+2\delta} L_n^2 D_n
.
\end{align*}
\end{proof}

\subsection{Estimate on the oscillation error} \label{ssec:osc}
In the decomposition \eqref{eq:decomposition_R} above, the introduction of the incremental pressure term $- \tfrac{1}{2}\nabla^{\phi_{n+1}}(|w_o|^2-\tilde{\rho}_\ell) = -\tfrac{1}{2}\dvgphin((|w_o|^2-\tilde{\rho}_\ell)Id)$ may seem arbitrary.
We shall see now that this choice contributes to produce better estimates for the (inverse divergence of the) oscillation error term, thus justifying our choice of the pressure $q_{n+1}$.

Let $W=W(y,s,\xi,\tau)$ be defined by \eqref{eq:def_W}, and recall from \autoref{cor:U} the decomposition
\begin{align*}
W \otimes W (y,s,\xi,\tau)
=
R_\ell (y,s) 
+
\sum_{1 \leq |k| \leq 2\lambda_0} U_k(y,s,\tau) e^{i k \cdot \xi},
\end{align*}
with $U_k \in C^\infty_{loc}(\mathbb{T}^3 \times \R^2, \mathcal{S}^{3 \times 3})$, $k \in \Lambda$, satisfying $U_k k = \tfrac{1}{2} Tr(U_k)k$. 
Using the identity $w_o(x,t) = W(x,t,\lambda \phi_{n+1}(x,t),\lambda t)$ we can rewrite the previous line as
\begin{align*}
w_o \otimes w_o (x,t)
=
R_\ell (x,t) 
+
\sum_{1 \leq |k| \leq 2\lambda_0} U_k(x,t,\lambda t)\, \Omega_k( \lambda \phi_{n+1}(x,t)).
\end{align*}

Therefore,
\begin{align*}
\dvgphin&\left(w_o \otimes w_o - \frac{1}{2}\left(|w_o|^2-\tilde{\rho}_\ell \right) Id +\mathring{R}_\ell\right)
\\
&=
\dvgphin\left(w_o \otimes w_o - \frac{1}{2}\left(|w_o|^2-{\rho}_\ell \right) Id +\mathring{R}_\ell\right)
\\
&=
\dvgphin\left(w_o \otimes w_o - R_\ell - \frac{1}{2}\left(|w_o|^2-Tr(R_\ell) \right) Id \right)
\\
&=
\sum_{1 \leq |k| \leq 2\lambda_0}
\dvgphin \left(U_k - \tfrac{1}{2} Tr(U_k) Id \right) (x,t,\lambda t) \Omega_k( \lambda \phi_{n+1}(x,t)),
\end{align*}
where in the last line we have used the relation $U_k k = \tfrac{1}{2} Tr(U_k)k$.

\begin{prop} \label{prop:R_osc}
Let us denote $\mathring{R}^{osc} \coloneqq \mathcal{R}^{\phi_{n+1}} \dvgphin(w_o \otimes w_o - \frac{1}{2}\left(|w_o|^2-\tilde{\rho}_\ell \right) Id +\mathring{R}_\ell)$. Then for every $r \geq r_\star+1$ and $\delta>0$ sufficiently small there exists a constant $C$ such that almost surely for every $L \in \N$, $L \geq 1$
\begin{align*}
\| \mathring{R}^{osc} \|_{C_{\leq \mathfrak{t}_L} C_x} 
&\leq 
CL_n^{3r+6} \lambda^{\delta-1} \mu  \delta_n   \varsigma_{n+1}^{\alpha-1},
\\
\| \mathring{R}^{osc} \|_{C_{\leq \mathfrak{t}_L} C^1_x} 
&\leq
C L_n^{3r+9} 
\lambda^{\delta} \mu  \delta_n  \varsigma_{n+1}^{\alpha-1}
.
\end{align*}
\end{prop}
\begin{proof}
Let us start with the estimate in $C_{\leq \mathfrak{t}_L}C_x$. 
We apply \autoref{prop:stat_phase_lem} with $r \geq r_\star+1$ and $\delta>0$ sufficiently small to obtain for every $L \in \N$, $N \geq 1$
\begin{align*}
\| \mathring{R}^{osc} \|_{C_{\leq \mathfrak{t}_L}C^\delta_x}
&\leq
CL^{\delta+1} \sum_{1 \leq |k| \leq 2\lambda_0} \left( \lambda^{\delta-1} [U_k]_{C^1_x} 
+ 
L^r \lambda^{\delta-r} [U_k]_{C^{r+1}_x} 
+ 
L^{r+\delta} \lambda^{-r} [U_k]_{C^{r+1+\delta}_x} \right),
\end{align*}
where we denote $[U_k]_{C^\cdot_x}=[U_k(\cdot,t,\lambda t)]_{C^\cdot_x}$ the spatial H\"older seminorms at fixed $t \leq \mathfrak{t}_L$.
By \autoref{cor:U} and interpolation inequality
\begin{align*}
\| \mathring{R}^{osc} \|_{C_{\leq \mathfrak{t}_L}C^\delta_x}
&\leq
CL^{r+2\delta+1} L_n^{3r+3\delta+11/2} \lambda^{\delta-1} \mu \delta_n   \varsigma_{n+1}^{\alpha-1}. 
\end{align*}
The previous bound makes use of $\lambda^{r-1} \geq \mu^{r+1} \varsigma_{n+1}^{r(\alpha-1)}  (D_n \ell^{-r-1} + \varsigma_{n+1}^{\alpha-1})$.

The same arguments used to control the $C^1_x$ norm of $\mathring{R}^{trans}$ also imply (use that for our choice of $r$ it holds $\lambda^{r-1} \geq \mu^{r+3} \varsigma_{n+1}^{(r+3)(\alpha-1)}(D_n \ell^{1-r}+\varsigma_{n+1}^{\alpha-1}) \geq \mu^{r+2} \varsigma_{n+1}^{(r+2)(\alpha-1)}(D_n \ell^{1-r}+\varsigma_{n+1}^{\alpha-1})^2$ )
\begin{align*}
\| \mathring{R}^{osc} \|_{C_{\leq \mathfrak{t}_L}C^1_x}
&\leq
CL \lambda \| \mathring{R}^{osc} \|_{C_{\leq \mathfrak{t}_L}C_x}
\\
&\quad+
CL^2 \sum_{1 \leq |k| \leq 2\lambda_0}\left( 
\lambda^{\delta-1} [U_k]_{C^2_x} 
+ 
L^r \lambda^{\delta-r} [U_k]_{C^{r+2}_x} 
+ 
L^{r+\delta} \lambda^{-r} [U_k]_{C^{r+2+\delta}_x} \right)
\\
&\leq
C L_n^{3r+9} \lambda^{\delta} \mu  \delta_n \varsigma_{n+1}^{\alpha-1}.
\end{align*}

\end{proof}

\subsection{Estimate on the flow error} \label{ssec:flow_error}
The flow error in the decomposition \eqref{eq:decomposition_R} is peculiar of the particular construction carried on in the present paper.
It is due to the fact that the Euler-Reynolds systems \eqref{eq:random_euler-reynolds} at level $n$ and $n+1$ are obtained composing with flows $\phi_n$ and $\phi_{n+1}$, respectively.

The key lemma we are going to use in this subsection is \autoref{lem:G_holder} below.

\begin{lem} \label{lem:G_holder}
For every $\delta \in (0,1)$, $\alpha' \in (0,\alpha)$ there exists a constant $C$ with the following property. For every $n \in \N$, given any smooth vector field $v \in C^\infty(\mathbb{T}^3,\R^3)$ on the torus and denoting $G \coloneqq \left(\dvgphin-\dvg^{\phi_n}\right) v$, almost surely for every $L \in \N$, $L \geq 1$ it holds
\begin{align*}
\|G\|_{C_{\leq \mathfrak{t}_L} B^{\delta-1}_{\infty,\infty}}
&\leq 
C L^3 (n+1)\varsigma_n^{\alpha'}
\| v \|_{C_{\leq \mathfrak{t}_L} C^{\delta}_x},
\\
\|G\|_{C_{\leq \mathfrak{t}_L} C^\delta_x}
&\leq 
C L^3 (n+1)\varsigma_n^{\alpha'}
\| v \|_{C_{\leq \mathfrak{t}_L} C^{1+\delta}_x}
,
\\
\|G\|_{C_{\leq \mathfrak{t}_L} C^{1+\delta}_x}
&\leq 
C L^4 (n+1)\varsigma_n^{\alpha'}
\| v \|_{C_{\leq \mathfrak{t}_L} C^{2+\delta}_x}
.
\end{align*}
\end{lem}
\begin{proof}
For $\phi=\phi_n$ or $\phi=\phi_{n+1}$ it holds
\begin{align*}
\dvgphi v 
= 
\sum_{j=1}^3 \partial_{x_j} (v^{j} \circ \phi^{-1}) \circ \phi
=
\sum_{j,k=1}^3 (\partial_{x_k} v^{j}) (\partial_{x_j} (\phi^{-1})^k \circ \phi).
\end{align*}
Therefore
\begin{align*}
G
&=
\sum_{j,k=1}^3 (\partial_{x_k} v^{j}) 
\left( \partial_{x_j} (\phi_{n+1}^{-1})^k \circ \phi_{n+1}
- \partial_{x_j} (\phi_{n}^{-1})^k \circ \phi_{n+1} \right)
\\
&\quad+
\sum_{j,k=1}^3 (\partial_{x_k} v^{j}) 
\left( \partial_{x_j} (\phi_{n}^{-1})^k \circ \phi_{n+1}
- \partial_{x_j} (\phi_{n}^{-1})^k \circ \phi_{n} \right).
\end{align*}
Thus by paraproduct estimates in Besov spaces \cite[Proposition A.7]{MoWe17}, \autoref{lem:C0_Cm} and \autoref{lem:flow} we have 
\begin{align*}
\| G \|_{C_{\leq \mathfrak{t}_L} B^{\delta-1}_{\infty,\infty}}
&\leq C 
\| v \|_{C_{\leq \mathfrak{t}_L} C^{\delta}_x}
\| \partial_{x_j} (\phi_{n+1}^{-1})^k \circ \phi_{n+1}
- \partial_{x_j} (\phi_{n}^{-1})^k \circ \phi_{n+1} \|_{C_{\leq \mathfrak{t}_L} C^1_x}
\\
&\quad+ C
\| v \|_{C_{\leq \mathfrak{t}_L} C^{\delta}_x}
\| \partial_{x_j} (\phi_{n}^{-1})^k \circ \phi_{n+1}
- \partial_{x_j} (\phi_{n}^{-1})^k \circ \phi_{n} \|_{C_{\leq \mathfrak{t}_L} C^1_x}
\\
&\leq C 
\| v \|_{C_{\leq \mathfrak{t}_L} C^{\delta}_x}
L\| \phi^{-1}_{n+1}-\phi^{-1}_n \|_{C_{\leq \mathfrak{t}_L} C^{2}_x}
\\
&\quad +C
\| v \|_{C_{\leq \mathfrak{t}_L} C^\delta_x} 
\left( 
L\| \phi_{n+1}-\phi_n \|_{C_{\leq \mathfrak{t}_L} C^{1}_x}
+
L^2\| \phi^{-1}_{n+1}-\phi^{-1}_n \|_{C_{\leq \mathfrak{t}_L} C_x} 
\right),
\end{align*}
where in the first line we have used that $\|\partial_{x_k}v\|_{C_{\leq \mathfrak{t}_L} B^{\delta-1}_{\infty,\infty}}
\leq C \|v\|_{C_{\leq \mathfrak{t}_L} B^{\delta}_{\infty,\infty}} = C \|v\|_{C_{\leq \mathfrak{t}_L} C^{\delta}_x}
$.
To conclude, recall that again by \autoref{lem:flow} it holds
\begin{align*}
\|\phi_{n+1}^{-1}-\phi_n^{-1}\|_{C_{\leq \mathfrak{t}_L} C^2_x}
&\leq 
CL (n+1) \varsigma_n^{\alpha'},
\\
\|\phi_{n+1}-\phi_n\|_{C_{\leq \mathfrak{t}_L} C^2_x}
&\leq 
CL (n+1) \varsigma_n^{\alpha'}.
\end{align*}
The inequalities concerning the $C_{\leq \mathfrak{t}_L} C^{\delta}_x$ and $C_{\leq \mathfrak{t}_L} C^{1+\delta}_x$ norms are completely analogous, and we omit their proof.
\end{proof}

Let us denote $F_n
\coloneqq 
v_n \otimes v_n - \mathring{R}_n + q_n Id$ and
\begin{align*}
F_\ell
\coloneqq 
F_n \ast \chi_\ell
=
(v_n \otimes v_n) \ast \chi_\ell - \mathring{R}_\ell + q_\ell Id.
\end{align*}
Notice that for every $r \geq 0$ and almost surely for every $L \in \N$, $L \geq 1$ it holds:
\begin{align*}
\|F_\ell\|_{C_{\leq \mathfrak{t}_L}C_x^r}
\leq
C
\ell^{-r} \|F_n \|_{C_{\leq \mathfrak{t}_L} C_x}
\leq
C L_n^2 \ell^{-r},
\end{align*} 
where the constant $C$ depends only on the mollifier $\chi$ and $r$. 

Finally, define $G_\ell \coloneqq \left(\dvgphin-\dvg^{\phi_n}\right) F_\ell$ and let us split $\mathring{R}^{flow} \coloneqq \mathring{R}^{flow}_1 + \mathring{R}^{flow}_2$, where
\begin{align*}
\mathring{R}^{flow}_1 
&\coloneqq 
\mathcal{R}^{\phi_{n+1}} G_\ell ,
\\
\mathring{R}^{flow}_2 
&\coloneqq 
\mathcal{R}^{\phi_{n+1}} \left( w_o \,\dvgphin v_\ell \right).
\end{align*}

\begin{prop} \label{prop:R_flow}
Let $\mathring{R}^{flow} \coloneqq \mathring{R}^{flow}_1 + \mathring{R}^{flow}_2$ be defined as above. 
Then for every $r \geq r_\star +2$, $\delta \in (0,1)$ sufficiently small and $\alpha' \in (\alpha_\star,\alpha)$ there exists a constant $C$ such that almost surely for every $L \in \N$, $L \geq 1$
\begin{align*}
\left\|  \mathring{R}^{flow} \right\|_{C_{\leq \mathfrak{t}_L} C_x}
&\leq
C L_n^{3r+5} \ell^{-\delta} (n+1)\varsigma_n^{\alpha'}
\\
\left\|  \mathring{R}^{flow} \right\|_{C_{\leq \mathfrak{t}_L} C^1_x}
&\leq
CL_n^{3r+8} \ell^{-1-\delta} (n+1)\varsigma_n^{\alpha'}
.
\end{align*} 
\end{prop}

\begin{proof}

Let us focus first on the term $\mathring{R}^{flow}_1$.
By \eqref{e:Schauder_R_Besov} and \autoref{lem:G_holder}, for  every $\delta>0$ one has
\begin{align*}
\|\mathring{R}^{flow}_1 \|_{C_{\leq \mathfrak{t}_L} C^\delta_x}
&\leq
CL^{5+4\delta}
\|G_\ell\|_{B^{\delta-1}_{\infty,\infty}}
\\
&\leq
CL^{8+4\delta}
(n+1)\varsigma_n^{\alpha'} \|F_\ell\|_{C_{\leq \mathfrak{t}_L} C^\delta_x}
\\
&\leq
CL^{8+4\delta}L_n^2 \ell^{-\delta}
(n+1)\varsigma_n^{\alpha'},
\end{align*}  
and similarly, using \eqref{e:Schauder_R} instead of \eqref{e:Schauder_R_Besov}
\begin{align*}
\|\mathring{R}^{flow}_1 \|_{C_{\leq \mathfrak{t}_L} C^{1+\delta}_x}
&\leq
CL^{9+3\delta}  (n+1)\varsigma_n^{\alpha'}
\|F_\ell\|_{C_{\leq \mathfrak{t}_L} C^{1+\delta}_x}
\\
&\leq
CL^{9+3\delta} L_n^2 \ell^{-1-\delta} (n+1)\varsigma_n^{\alpha'}.
\end{align*}


Moving to the second term $\mathring{R}^{flow}_2$, we notice that 
\begin{align} \label{eq:w_o div v_ell}
w_o\, \dvgphin v_\ell 
&=
\left(
\sum_{k \in \Lambda}
a_k\circ \phi_{n+1}^{-1}\, \mbox{div}\left(v_\ell\circ \phi_{n+1}^{-1}\right) \Omega_k^\lambda
\right) \circ \phi_{n+1},
\end{align}
so that by \autoref{prop:stat_phase_lem}, \autoref{lem:C0_Cm} and taking $r \geq r_\star + 2$
\begin{align*}
\|\mathring{R}^{flow}_2 \|_{C_{\leq \mathfrak{t}_L}C^{\delta}_x}
&\leq
C L^\delta \sum_{k \in \Lambda}
\lambda^{\delta-1} \|a_k\|_{C_x} \|v_\ell\|_{C^1_x}
\\
&\quad+
C L^{r+1+\delta} \sum_{k \in \Lambda} \sum_{r_1+r_2=r}
\lambda^{\delta-r} \|a_k\|_{C^{r_1}_x} \|v_\ell\|_{C^{r_2+1}_x}
\\
&\quad+
C L^{r+2+\delta} \sum_{k \in \Lambda} \sum_{r_1+r_2=r+1}
\lambda^{-r} \|a_k\|_{C^{r_1}_x} \|v_\ell\|_{C^{r_2+1}_x}
\\
&\leq
CL^{r+3}L_n^{3r+4}
\lambda^{\delta-1} \delta_n^{1/2} D_n.
\end{align*}
In the line above we have used, as usual, the factors $\lambda^{-r}$ to compensate, up to a power of $L_n$, for the quantities $\|a_k\|_{C^{r_1}_x} \|v_\ell \|_{C^{r_2+1}_x}$.
Taking the space derivative in \eqref{eq:w_o div v_ell} as usual, we obtain in a similar fashion
\begin{align*}
\|\mathring{R}^{flow}_2 \|&_{C_{\leq \mathfrak{t}_L}C^{1+\delta}_x}
\leq
CL \lambda \|\mathring{R}^{flow}_2 \|_{C_{\leq \mathfrak{t}_L}C^{\delta}_x} 
\\
&\quad+
CL^3 \sum_{k \in \Lambda}
\left(
\lambda^{\delta-1} [a_k]_{C^1_x} [v_\ell]_{C^1_x}
+
\lambda^{\delta-1} \|a_k\|_{C_x} [v_\ell]_{C^2_x}
\right)
\\
&\quad+
C L^{r+2+\delta} \sum_{k \in \Lambda} \sum_{r_1+r_2=r+1}
\lambda^{\delta-r} \|a_k\|_{C^{r_1}_x} \|v_\ell \|_{C^{r_2+1}_x}
\\
&\quad+
C L^{r+3+\delta} \sum_{k \in \Lambda} \sum_{r_1+r_2=r+2}
\lambda^{-r} \|a_k\|_{C^{r_1}_x} \|v_\ell \|_{C^{r_2+1}_x}
\\
&\leq
CL^{r+4}L_n^{3r+7}
\lambda^{\delta} \delta_n^{1/2} D_n.
\end{align*}

\end{proof}

\subsection{Estimate on the compressibility error}
\label{ssec:comp}
Let us finally move to the last term to control in the decomposition \eqref{eq:decomposition_R}, namely the compressibility error term.
It is due to the fact that we have added a compressibility corrector $w_c$ to the oscillatory term $w_o$ when defining the new velocity $v_{n+1} = v_\ell + w_o + w_c$, in order to satisfy the condition on the divergence \eqref{eq:est_div} at level $n+1$.
In particular, recall we have defined $w_c = w^1_c + w^2_c$, where $w^1_c = - (\mathcal{Q}^{\phi_n} v_n ) \ast\chi_\ell$ and $w^2_c = -  \mathcal{Q}^{\phi_{n+1}} w_o$.
We will need the following preliminary:
\begin{lem} \label{lem:w_c}
For every $\delta \in (0,1)$ sufficiently small there exists a constant $C$ such that, almost surely for every $L \in \N$, $L \geq 1$:
\begin{align*}
\|w_c^1\|_{C_{\leq \mathfrak{t}_L}C^{\delta}_x}
&\leq
CL^{7} L_n \delta_{n+2}^{6/5},
\\
\|w_c^1\|_{C_{\leq \mathfrak{t}_L}C^{1+\delta}_x}
&\leq
CL^{7} L_n \delta_{n+2}^{6/5} \ell^{-1},
\\
\| w^2_c \|_{C_{\leq \mathfrak{t}_L}C^\delta_x}
&\leq
C L^{1+2\delta} L_n^{3+3\delta} \lambda^{\delta-1} \mu \delta_n^{1/2}  \varsigma_{n+1}^{\alpha-1},
\\
\| w^2_c \|_{C_{\leq \mathfrak{t}_L}C^{1+\delta}_x}
&\leq
CL^{3+2\delta} L_n^{6+3\delta}\lambda^{\delta} \mu \delta_n^{1/2}  \varsigma_{n+1}^{\alpha-1}.
\end{align*}
\end{lem}

\begin{proof}
First, notice that $\|w_c^1\|_{C_{\leq \mathfrak{t}_L}C^{\delta}_x} \leq \|\mathcal{Q}^{\phi_n}v_n\|_{C_{\leq \mathfrak{t}_L}C^{\delta}_x}$ and $\|w_c^1\|_{C_{\leq \mathfrak{t}_L}C^{1+\delta}_x} \leq \ell^{-1}\|\mathcal{Q}^{\phi_n}v_n\|_{C_{\leq \mathfrak{t}_L}C^{\delta}_x}$.
In addition, since $v_n$ is zero mean it holds $Q^{\phi_n}v_n = \nabla^{\phi_n} \psi$, where $\psi$ solves $\Delta^{\phi_n} \psi = \dvg^{\phi_n} v_n$ and $\int_{\T^3} \psi = 0$. Therefore, assuming $\delta$ sufficiently small, by iterative assumption \eqref{eq:est_div}, Schauder estimates in Besov spaces and interpolation we have
\begin{align} \label{eq:Qv_n} 
\|Q^{\phi_n}v_n\|_{C_{\leq \mathfrak{t}_L}C^\delta_x}
&\leq
CL^{1+2\delta} \|\psi\|_{C_{\leq \mathfrak{t}_L}C^{1+\delta}_x}
\\
&\leq \nonumber
CL^{6+2\delta} \|\dvg^{\phi_n} v_n \|_{C_{\leq \mathfrak{t}_L}B^{\delta-1}_{\infty,\infty}}
\\
&\leq \nonumber
CL^{6+2\delta} 
\|\dvg^{\phi_n} v_n \|^{1-2\delta}_{C_{\leq \mathfrak{t}_L}B^{-1}_{\infty,\infty}}
\|\dvg^{\phi_n} v_n \|^{2\delta}_{C_{\leq \mathfrak{t}_L}B^{-1/2}_{\infty,\infty}}
\\
&\leq \nonumber
CL^{6+10\delta} L_n \delta_{n+2}^{5(1-2\delta)/4} D_n^\delta
\leq
CL^{7} L_n \delta_{n+2}^{6/5},
\end{align} 
provided $\delta$ is taken sufficiently small.

As for the inequalities involving $w_c^2$, the proof is similar to \cite[Lemma 6.2]{DLS13} and \cite[Proposition 6.1]{DLS14}.
Let us define the field $u_c : \mathbb{T}^3 \times \R \to \R^3$ as
\begin{align*}
u_c(x,t) 
&\coloneqq 
i \sum_{k \in\Lambda} {\nablaphin a_k (x,t,\lambda t)} \times \frac{k}{|k|^2} \times \Omega_k(\lambda \phi_{n+1}(x,t));
\end{align*}
since $k \cdot E_k = 0$ (recall that $\Omega_k(\xi) = E_k e^{ik \cdot \xi}$) it is easy to check by direct verification
\begin{align*}
w_c^2
&=  
\frac{1}{\lambda} \mathcal{Q}^{\phi_{n+1}} u_c
+
\frac{1}{\lambda} \mathcal{Q}^{\phi_{n+1}} \nabla^{\phi_{n+1}}\times
\left( \sum_{k \in \Lambda} -i a_k \frac{k}{|k|^2}\times (\Omega_k^\lambda \circ \phi_{n+1}) \right)
\\
&=  
\frac{1}{\lambda} \mathcal{Q}^{\phi_{n+1}} u_c
=
\frac{1}{\lambda} \mathcal{Q}(u_c \circ \phi_{n+1}^{-1}) \circ \phi_{n+1},
\end{align*}
where the last line is justified by the identity $\mathcal{Q}^{\phi_{n+1}} (\nabla^{\phi_{n+1}} \times v)= 0$ for every smooth vector field $v$.
Schauder estimate then gives for $r=0,1$ and $\delta \in (0,1)$
\begin{align*}
\| w_c^2 \|_{C_{\leq \mathfrak{t}_L}C^{r+\delta}_x}
&\leq
CL^{r+\delta}\lambda^{-1}
\| u_c \circ \phi_{n+1}^{-1} \|_{C_{\leq \mathfrak{t}_L}C^{r+\delta}_x},
\end{align*}
and thus, recalling \autoref{prop:aux_est} we get
\begin{align*}
\| w_c^2 \|_{C_{\leq \mathfrak{t}_L}C^{\delta}_x}
&\leq
CL^{1+\delta}\lambda^{-1} \sum_{k \in \Lambda} \left(
[a_k]_{C^1_x} \|\Omega_k^\lambda\|_{C^\delta_x} 
+
L^{\delta}\|a_k\|_{C^{1+\delta}_x} \|\Omega_k^\lambda \|_{C_x}\right)
\\
&\leq
CL^{1+2\delta} L_n^{3+3\delta}\lambda^{\delta-1} \mu \delta_n^{1/2}  \varsigma_{n+1}^{\alpha-1},
\end{align*}
and
\begin{align*}
\| w_c^2 \|_{C_{\leq \mathfrak{t}_L}C^{1+\delta}_x}
&\leq
CL^{2+\delta}\lambda^{-1} \sum_{k \in \Lambda} \left(
[a_k]_{C^1_x} \|\Omega_k^\lambda\|_{C^{1+\delta}_x} 
+
L^{1+\delta}\|a_k\|_{C^{2+\delta}_x} \|\Omega_k^\lambda \|_{C_x}\right)
\\
&\leq
CL^{3+2\delta}L_n^{6+3\delta}\lambda^{\delta} \mu \delta_n^{1/2}  \varsigma_{n+1}^{\alpha-1}.
\end{align*}
\end{proof}

We will also need the following:
\begin{lem} \label{lem:w_o}
Fix $r\geq r_\star + 2$. Then there exists a constant $C$ such that, almost surely for every $L \in \N$, $L \geq 1$:
\begin{align*}
\| w_o \|_{C_{\leq \mathfrak{t}_L}C^\delta_x}
&\leq
C L^{2\delta} L_n^{7/2} \lambda^\delta \delta_n^{1/2},
\\
\| w_o \|_{C_{\leq \mathfrak{t}_L}C^{1+\delta}_x}
&\leq
C L^{2+2\delta} L_n^{3+3\delta} \lambda^{1+\delta} \delta_n^{1/2} \mu^\delta \varsigma_{n+1}^{\delta(\alpha-1)}
,
\\
\| \partial_t w_o \|_{C_{\leq \mathfrak{t}_L} C_x}
&\leq
C L_n^{3r+9} \lambda^{1+\delta} \delta_n^{1/2} \mu^\delta \varsigma_{n+1}^{\delta(\alpha-1)}
.
\end{align*}
\end{lem}

\begin{proof}
It holds by the very definition $w_o \circ \phi_{n+1}^{-1}= \sum_{k \in \Lambda} (a_k \circ \phi_{n+1}^{-1}) \Omega_k^\lambda$, and thus by \autoref{lem:C0_Cm} and \autoref{prop:aux_est}
\begin{align*}
\| w_o \|_{C_{\leq \mathfrak{t}_L}C^\delta_x}
&\leq
C L^\delta  \sum_{k \in \Lambda} \left(
\|a_k\|_{C_x} \|\Omega_k^\lambda\|_{C^\delta_x}
+
L^\delta\|a_k\|_{C^\delta_x} \|\Omega_k^\lambda\|_{C_x} \right)
\\
&\leq
C L^{2\delta} L_n^{7/2} \lambda^\delta \delta_n^{1/2},
\\
\| w_o \|_{C_{\leq \mathfrak{t}_L}C^{1+\delta}_x}
&\leq
C L^{1+\delta}  \sum_{k \in \Lambda} \left(
L^\delta \|a_k\|_{C^\delta_x} \|\Omega_k^\lambda\|_{C^{1+\delta}_x}
+
L^{1+\delta}\|a_k\|_{C^{1+\delta}_x} \|\Omega_k^\lambda\|_{C^\delta_x} \right)
\\
&\leq
C L^{2+2\delta} L_n^{3+3\delta} \lambda^{1+\delta} \delta_n^{1/2} \mu^\delta \varsigma_{n+1}^{\delta(\alpha-1)}.
\end{align*}

This proves the first two bounds. As for the second one, recall the decomposition of the transport error $\mathring{R}^{trans}$ from \autoref{ssec:trans}. By \autoref{lem:inverse} it holds
\begin{align*}
\dvgphin \mathring{R}^{trans}
&=
\partial_t w_o
+
(v_\ell \cdot \nabla^{\phi_{n+1}})w_o
-
\int_{\T^3} \partial_t w_o
-
\int_{\T^3} (v_\ell \cdot \nabla^{\phi_{n+1}})w_o,
\end{align*}
implying (use \eqref{e:average} with $r=1$ to control the average of $\partial_t w_o$)
\begin{align*}
\| \partial_t w_o \|_{C_{\leq \mathfrak{t}_L} C_x}
&\leq
\|  \dvgphin \mathring{R}^{trans}\|_{C_{\leq \mathfrak{t}_L} C_x}
+ 
C\| (v_\ell \cdot \nabla^{\phi_{n+1}})w_o \|_{C_{\leq \mathfrak{t}_L} C_x}
+
\left\| \int_{\T^3} \partial_t w_o \right\|_{C_{\leq \mathfrak{t}_L}}
\\
&\leq
CL \| \mathring{R}^{trans} \|_{C_{\leq \mathfrak{t}_L} C^1_x}
+
CL_n\|  w_o \|_{C_{\leq \mathfrak{t}_L} C^1_x}
+
\left\| \int_{\T^3} \partial_t w_o \right\|_{C_{\leq \mathfrak{t}_L}} 
\\
&\leq
CL_n^{3r+9} \lambda^{1+\delta} \delta_n^{1/2} \mu^\delta \varsigma_{n+1}^{\delta(\alpha-1)},
\end{align*}
completing the proof of the lemma.
\end{proof}

We are finally ready to prove the following:
\begin{prop} \label{prop:R_comp}
Let us denote 
\[
\mathring{R}^{comp} \coloneqq \mathcal{R}^{\phi_{n+1}} (\partial_t w_c + \dvgphin(v_{n+1} \otimes w_c + w_c \otimes v_{n+1} - w_c \otimes w_c +v_\ell \otimes w_o)).
\]
Then for every $r \geq r_\star + 1$, $\delta>0$ sufficiently small there exists a constant $C$ such that almost surely for every $L\in \N$, $L \geq 1$
\begin{align*}
\| \mathring{R}^{comp} \|_{C_{\leq \mathfrak{t}_L} C_x}
&\leq
C L_n^{3r+6} \lambda^\delta \delta_n^{1/2} \delta_{n+2}^{6/5},
\\ 
\| \mathring{R}^{comp} \|_{C_{\leq \mathfrak{t}_L} C^1_x}
&\leq
CL_n^{3r+10} \lambda^{1+\delta} \mu^\delta \delta_n^{1/2} \delta_{n+2}^{6/5} \varsigma_{n+1}^{\delta(\alpha-1)}
.
\end{align*}
\end{prop}

\begin{proof}
We split $\mathring{R}^{comp}$ into four terms:
\begin{align*}
\mathring{R}^{comp}
=
\mathring{R}^{comp}_1
+
\mathring{R}^{comp}_2
+
\mathring{R}^{comp}_3
+
\mathring{R}^{comp}_4,
\end{align*}
where we define
\begin{align*}
\mathring{R}^{comp}_1
&\coloneqq
\mathcal{R}^{\phi_{n+1}} \partial_t w^1_c ,
\\
\mathring{R}^{comp}_2
&\coloneqq
\mathcal{R}^{\phi_{n+1}} \partial_t w^2_c ,
\\
\mathring{R}^{comp}_3
&\coloneqq
\mathcal{R}^{\phi_{n+1}}\dvgphin \left( v_{n+1} \otimes w_c + w_c \otimes v_{n+1} - w_c \otimes w_c \right),
\\
\mathring{R}^{comp}_4
&\coloneqq
\mathcal{R}^{\phi_{n+1}} \dvgphin (v_\ell \otimes w_o).
\end{align*}

We shall provide estimates for each term separately, starting from $\mathring{R}^{comp}_1$.
Since $\chi_\ell$ is supported in $[0,\ell]^3 \times [0,\ell] \subset [0,2\pi]^3 \times [0,\ell]$ and $Q^{\phi_n}v_n$ has zero spatial average at every fixed time, one has the alternative expression 
\begin{align} \label{eq:conv w^1_c}
-\partial_t w^1_c=(Q^{\phi_n}v_n) \ast \partial_t\chi_\ell 
&=
\int_{\R^3 \times \R}
(Q^{\phi_n}v_n)(\cdot-y,\cdot-s) \partial_t\chi_\ell (y,s) dyds
\\
&= \nonumber
\int_{[0,2\pi]^3 \times [0,\ell]}
(Q^{\phi_n}v_n)(\cdot-y,\cdot-s) \partial_t\chi_\ell (y,s) dyds 
\\
&= \nonumber
\int_{[0,2\pi]^3 \times [0,\ell]}
(Q^{\phi_n}v_n)(\cdot-y,\cdot-s) \partial_t\chi^0_\ell (y,s) dyds,
\end{align}
where we have denoted $\partial_t\chi^0_\ell \coloneqq \partial_t\chi_\ell - (2\pi)^{-3 }\int_{\T^3} \partial_t\chi_\ell$ the zero-mean version (on the torus) of the convolution kernel $\partial_t\chi_\ell$.
We denote the spatial convolution on the torus by the symbol $\ast_{\T^3}$, i.e. for every fixed $t \in \R$
\begin{align*}
\int_0^\ell (Q^{\phi_n}v_n)(\cdot,t-s) \ast_{\T^3} \partial_t\chi^0_\ell(\cdot,s) ds
\coloneqq
\int_{[0,2\pi]^3 \times [0,\ell]}
(Q^{\phi_n}v_n)(\cdot-y,t-s) \partial_t\chi^0_\ell (y,s) dyds.
\end{align*}

Using \eqref{e:Schauder_R_Besov} we compute
\begin{align*}
\| \mathring{R}^{comp}_1 \|_{C_{\leq \mathfrak{t}_L}C^\delta_x}
&=
\| \mathcal{R}^{\phi_{n+1}} \partial_t w^1_c \|_{C_{\leq \mathfrak{t}_L}C^\delta_x}
\leq
C L^{5+4\delta}
\| \partial_t w^1_c \|_{C_{\leq \mathfrak{t}_L}B^{\delta-1}_{\infty,\infty}},
\end{align*}
and by \eqref{eq:conv w^1_c} above and \autoref{lem:conv_Besov} we have for any $p>1$
\begin{align*}
\| \partial_t w^1_c \|_{C_{\leq \mathfrak{t}_L}B^{\delta-1}_{\infty,\infty}}
&= \sup_{t \leq \mathfrak{t}_L}
\int_0^\ell
\| (Q^{\phi_n}v_n)(\cdot,t-s) \ast_{\T^3} \partial_t\chi^0_\ell(\cdot,s) \|_{B^{\delta-1}_{\infty,\infty}} ds
\\
&\leq
C\ell
\| Q^{\phi_n}v_n\|_{C_{\leq \mathfrak{t}_L}L^\infty_x}
\| \partial_t\chi^0_\ell \|_{C_{\leq \mathfrak{t}_L}B^{2\delta-1}_{p,\infty}}.
\end{align*}
Now observe that 
$\partial_t \chi^0_\ell(\cdot,s)=
\ell^{-2} \left(\ell^{-3}\partial_t \chi(\cdot/\ell,s/\ell) - (2\pi)^{-3}\int_{\T^3} \partial_t \chi (y,s/\ell) dy\right)$
for every $s \in \R$. 
Therefore by \autoref{lem:scaling} and taking $p$ sufficiently close to $1$ we have
\begin{align*}
\|\partial_t \chi^0_\ell(\cdot,s)\|_{B^{2\delta-1}_{p,\infty}} 
&=
\ell^{-2} \left\|\ell^{-3}\partial_t \chi(\cdot/\ell,s/\ell) - (2\pi)^{-3}\int_{\T^3} \partial_t \chi (y,s/\ell) dy\right\|_{B^{2\delta-1}_{p,\infty}} 
\\
&=
\ell^{-2}\ell^{3(1/p-1)+1-2\delta}\|\partial_t \chi(\cdot,s/\ell)\|_{B^{2\delta-1}_{p,\infty}} 
\\
&\leq 
\ell^{-1-3\delta}\|\partial_t \chi\|_{C_{\leq \mathfrak{t}_L}B^{2\delta-1}_{p,\infty}}
\leq C \ell^{-1-3\delta}.
\end{align*}
Recalling \eqref{eq:Qv_n}, we arrive to
\begin{align*}
\| \mathring{R}^{comp}_1 \|_{C_{\leq \mathfrak{t}_L}C^\delta_x}
\leq
C L^{12+4\delta} L_n \delta_{n+2}^{6/5} \ell^{-3\delta}.
\end{align*}

As for the $C_{\leq \mathfrak{t}_L} C^1_x$ norm of $\mathring{R}^{comp}_1 $ we use \eqref{e:Schauder_Rdiv} and similar arguments to get
\begin{align*}
\|\mathring{R}^{comp}_1 \|_{C_{\leq \mathfrak{t}_L} C^1_x}
&\leq
CL^{1+2\delta}
\| \partial_t w^1_c \|_{C_{\leq \mathfrak{t}_L} C^\delta_x}
\leq
C L^{8+2\delta} L_n \delta_{n+2}^{6/5} \ell^{-1-3\delta}.
\end{align*}

As for the term involving $\partial_t w^2_c$, the analysis is more involved since $w_c^2$ is not a convolution, and we first need the following observation.
Recall the definition of $u_c$ from \autoref{lem:w_c}.
We have
\begin{align*}
\partial_t w^2_c
&=
\frac{1}{\lambda} [\mathcal{Q} \partial_t (u_c \circ \phi_{n+1}^{-1})] \circ \phi_{n+1}
+
\frac{1}{\lambda} \dot{\phi}_{n+1}\cdot
[(\nabla\mathcal{Q}(u_c \circ \phi_{n+1}^{-1})) \circ \phi_{n+1}]
\\
&=
\frac{1}{\lambda} [\mathcal{Q} \partial_t (u_c \circ \phi_{n+1}^{-1})] \circ \phi_{n+1}
+
\dot{\phi}_{n+1}\cdot
\nabla^{\phi_{n+1}} w^2_c.
\end{align*}
As a consequence, applying the operator $\mathcal{R}^{\phi_{n+1}}$ to both sides of the previous equation we get
\begin{align*}
\mathring{R}^{comp}_2
=
\frac{1}{\lambda} [\mathcal{R}\mathcal{Q} \partial_t (u_c \circ \phi_{n+1}^{-1})] \circ \phi_{n+1}
+
\mathcal{R}^{\phi_{n+1}} \left( \dot{\phi}_{n+1}\cdot
\nabla^{\phi_{n+1}} w^2_c \right). 
\end{align*}
Now we proceed as usual; first, by \autoref{lem:C0_Cm}, \eqref{e:Schauder_R_Besov}, the paraproduct estimates for Besov spaces \cite[Proposition A.7]{MoWe17}, and \autoref{lem:w_c}
\begin{align*}
\| \mathring{R}^{comp}_2 \|_{C_{\leq \mathfrak{t}_L} C^\delta_x}
&\leq
CL^\delta \lambda^{-1}
\| \mathcal{R} \mathcal{Q} \partial_t(u_c\circ \phi_{n+1}^{-1}) \|_{C_{\leq \mathfrak{t}_L} C^\delta_x}
+
CL^{5+4\delta}
\left\| \dot{\phi}_{n+1}\cdot\nabla^{\phi_{n+1}} w^2_c \right\|_{C_{\leq \mathfrak{t}_L} B^{\delta-1}_{\infty,\infty}}
\\
&\leq
CL^\delta \lambda^{-1}
\| \mathcal{R} \mathcal{Q} \partial_t(u_c\circ \phi_{n+1}^{-1}) \|_{C_{\leq \mathfrak{t}_L} C^\delta_x}
+
CL^{11+4\delta} \varsigma_{n+1}^{\alpha-1}
\left\| w^2_c \right\|_{C_{\leq \mathfrak{t}_L} C^{\delta}_x}
\\
&\leq
CL^\delta \lambda^{-1}
\| \mathcal{R} \mathcal{Q} \partial_t(u_c\circ \phi_{n+1}^{-1}) \|_{C_{\leq \mathfrak{t}_L} C^\delta_x}
+
CL^{12+6\delta}L_n^{3+3\delta} \lambda^{\delta-1} \mu \delta_n^{1/2} \varsigma_{n+1}^{2\alpha-2}.
\end{align*}
Since
\begin{align*}
\partial_t(u_c\circ \phi_{n+1}^{-1})
&=
i 
\sum_{k \in \Lambda} {\nabla_y \partial_t (a_k\circ \phi_{n+1}^{-1})} \times \frac{k}{|k|^2}\times \Omega_k^\lambda,
\end{align*}
the stationary phase Lemma implies, together with the assumption \eqref{eq:ass_lambda}
\begin{align*}
\| \mathring{R}^{comp}_2 \|_{C_{\leq \mathfrak{t}_L}C^\delta_x}
&\leq 
C L^{r+1+2\delta} \sum_{k \in \Lambda} \left(
\lambda^{\delta-2}[\partial_s a_k]_{C^1_x}
+
\lambda^{\delta-r-1}[\partial_s a_k]_{C^{r+1}_x}
+
\lambda^{-r-1}[\partial_s a_k]_{C^{r+1+\delta}_x}
\right)
\\
&\quad+
C L^{r+1+2\delta} \sum_{k \in \Lambda} \left(
\lambda^{\delta-1}[\partial_\tau a_k]_{C^1_x}
+
\lambda^{\delta-r}[\partial_\tau a_k]_{C^{r+1}_x}
+
\lambda^{-r}[\partial_\tau a_k]_{C^{r+1+\delta}_x}
\right)
\\
&\quad+
{C L^{r+3+2\delta} \varsigma_{n+1}^{\alpha-1}  \sum_{k \in \Lambda} \left(
\lambda^{\delta-1}[a_k]_{C^2_x}
+
\lambda^{\delta-r}[a_k]_{C^{r+2}_x}
+
\lambda^{-r}[a_k]_{C^{r+2+\delta}_x}
\right)}
\\
&\quad+
CL^{12+6\delta}L_n^{3+3\delta} \lambda^{\delta-1} \mu \delta_n^{1/2} \varsigma_{n+1}^{2\alpha-2}
\\
&\leq
CL^\delta L_n^{3r+3\delta+11/2} \lambda^{\delta-1} \mu \delta_n^{1/2}  \varsigma_{n+1}^{2\alpha-2},
\end{align*}
and with usual arguments one deduces also, for $\lambda^{r-1} \geq \mu^{r+4} \varsigma_{n+1}^{(r+4)(\alpha-1)}(D_n \ell^{-r-3} + \varsigma_{n+1}^{\alpha-1})$:
\begin{align*}
\| \mathring{R}^{comp}_2 \|_{C_{\leq \mathfrak{t}_L}C^1_x}
&\leq
CL \lambda 
\| \mathring{R}^{comp}_2 \|_{C_{\leq \mathfrak{t}_L}C_x}
\\
&\quad+
C L^{r+3+\delta} \sum_{k \in \Lambda} \left(
\lambda^{\delta-2}[\partial_s a_k]_{C^2_x}
+
\lambda^{\delta-r-1}[\partial_s a_k]_{C^{r+2}_x}
+
\lambda^{-r-1}[\partial_s a_k]_{C^{r+2+\delta}_x}
\right)
\\
&\quad+
C L^{r+3+\delta} \sum_{k \in \Lambda} \left(
\lambda^{\delta-1}[\partial_\tau a_k]_{C^2_x}
+
\lambda^{\delta-r}[\partial_\tau a_k]_{C^{r+2}_x}
+
\lambda^{-r}[\partial_\tau a_k]_{C^{r+2+\delta}_x}
\right)
\\
&\quad+
C L^{r+5+2\delta} \varsigma_{n+1}^{\alpha-1}  \sum_{k \in \Lambda} \left(
\lambda^{\delta-1}[a_k]_{C^3_x}
+
\lambda^{\delta-r}[a_k]_{C^{r+3}_x}
+
\lambda^{-r}[a_k]_{C^{r+3+\delta}_x}
\right)
\\
&\quad+
CL^{15+6\delta}L_n^{6+3\delta} \lambda^\delta \mu \delta_{n}^{1/2} \varsigma_{n+1}^{2\alpha-2}
\\
&\leq
CL L_n^{3r+3\delta+17/2} \lambda^{\delta} \mu \delta_n^{1/2}  \varsigma_{n+1}^{2\alpha-2}.
\end{align*}

Moving to $\mathring{R}^{comp}_3$, we have by \autoref{lem:schauder}, \autoref{lem:w_c} and \autoref{lem:w_o}
\begin{align*}
\| \mathring{R}^{comp}_3 \|_{C_{\leq \mathfrak{t}_L}C^\delta_x}
&\leq 
C L^{2\delta}
\| v_{n+1} \otimes w_c + w_c \otimes v_{n+1} - w_c \otimes w_c \|_{C_{\leq \mathfrak{t}_L}C^\delta_x}
\\
&\leq
C L^{2\delta}\left( 
\|v_{n+1}\|_{C_{\leq \mathfrak{t}_L}C^\delta_x} \|w_c\|_{C_{\leq \mathfrak{t}_L}C^\delta_x}
+
\|w_c\|_{C_{\leq \mathfrak{t}_L}C^\delta_x}^2
\right)
\\
&\leq
C L^{2\delta}\left( 
\|v_\ell\|_{C_{\leq \mathfrak{t}_L}C^\delta_x} \|w_c\|_{C_{\leq \mathfrak{t}_L}C^\delta_x}
+
\|w_o\|_{C_{\leq \mathfrak{t}_L}C^\delta_x} \|w_c\|_{C_{\leq \mathfrak{t}_L}C^\delta_x}
+
\|w_c\|_{C_{\leq \mathfrak{t}_L}C^\delta_x}^2
\right)
\\
&\leq 
C L_n^{13/2+4\delta} \lambda^\delta\delta_n^{1/2} \delta_{n+2}^{6/5},
\end{align*}
and
\begin{align*}
\| \mathring{R}^{comp}_3 \|_{C_{\leq \mathfrak{t}_L}C^{1+\delta}_x}
&\leq 
C L^{2+2\delta}
\| v_{n+1} \otimes w_c + w_c \otimes v_{n+1} - w_c \otimes w_c \|_{C_{\leq \mathfrak{t}_L}C^{1+\delta}_x}
\\
&\leq
C L^{2+2\delta}\left( 
\|v_{n+1}\|_{C_{\leq \mathfrak{t}_L}C^{\delta}_x} \|w_c\|_{C_{\leq \mathfrak{t}_L}C^{1+\delta}_x}
+
\|v_{n+1}\|_{C_{\leq \mathfrak{t}_L}C^{1+\delta}_x} \|w_c\|_{C_{\leq \mathfrak{t}_L}C^{\delta}_x}
\right)
\\
&\quad+
CL^{2+2\delta}
\|w_c\|_{C_{\leq \mathfrak{t}_L}C^{\delta}_x} 
\|w_c\|_{C_{\leq \mathfrak{t}_L}C^{1+\delta}_x}
\\
&\leq
C L^{2+2\delta}\left( 
\|v_\ell\|_{C_{\leq \mathfrak{t}_L}C^{\delta}_x} \|w_c\|_{C_{\leq \mathfrak{t}_L}C^{1+\delta}_x}
+
\|v_\ell\|_{C_{\leq \mathfrak{t}_L}C^{1+\delta}_x} \|w_c\|_{C_{\leq \mathfrak{t}_L}C^{\delta}_x}
\right)
\\
&\quad+
C L^{2+2\delta}\left( 
\|w_o\|_{C_{\leq \mathfrak{t}_L}C^{\delta}_x} \|w_c\|_{C_{\leq \mathfrak{t}_L}C^{1+\delta}_x}
+
\|w_o\|_{C_{\leq \mathfrak{t}_L}C^{1+\delta}_x} \|w_c\|_{C_{\leq \mathfrak{t}_L}C^{\delta}_x}
\right)
\\
&\quad+
CL^{2+2\delta}
\|w_c\|_{C_{\leq \mathfrak{t}_L}C^{\delta}_x} 
\|w_c\|_{C_{\leq \mathfrak{t}_L}C^{1+\delta}_x}
\\
&\leq
C L_n^{10} \lambda^{1+\delta} \mu^\delta \delta_n^{1/2} \delta_{n+2}^{6/5} \varsigma_{n+1}^{\delta(\alpha-1)}.
\end{align*}

We have nothing left but $\mathring{R}^{comp}_4$. 
To better control this term, we rewrite
\begin{align*}
\mathcal{R}^{\phi_{n+1}} \dvgphin (v_\ell \otimes w_o)
&=
\mathcal{R}^{\phi_{n+1}} \left( 
v_\ell \, \dvgphin (w_o)  \right)
+ 
\mathcal{R}^{\phi_{n+1}} \left( 
(w_o \cdot \nabla^{\phi_{n+1}}) v_\ell \right)
\\
&=
\mathcal{R} \left( \, \sum_{k \in \Lambda} 
(v_\ell \circ \phi_{n+1}^{-1})\, \mbox{div} (a_k\circ \phi_{n+1}^{-1}) \, \Omega_k 
\right) \circ \phi_{n+1}
\\
&\quad+ 
\mathcal{R} \left( \, \sum_{k \in \Lambda} 
((a_k\circ \phi_{n+1}^{-1}) \cdot \nabla) (v_\ell \circ \phi_{n+1}^{-1}) \Omega_k \right)
\circ \phi_{n+1},
\end{align*}
so that we can apply stationary phase Lemma and \autoref{lem:C0_Cm} to get (details omitted):
\begin{align*}
\|\mathring{R}^{comp}_4 \|_{C_{\leq \mathfrak{t}_L}C^\delta_x}
&\leq
CL_n^{3r+6} \lambda^{\delta-1} \mu \delta_n^{1/2}   \varsigma_{n+1}^{\alpha-1},
\end{align*}
and 
\begin{align*}
\|\mathring{R}^{comp}_4 \|_{C_{\leq \mathfrak{t}_L}C^{1+\delta}_x}
&\leq
CL_n^{3r+7} \lambda^{\delta-1} \mu^2 \delta_n^{1/2}  \varsigma_{n+1}^{\alpha-1} (D_n \ell^{-1} + \varsigma_{n+1}^{\alpha-1}).
\end{align*}

\end{proof}

\subsection{Estimate on the divergence} \label{ssec:div}
Recall that  by construction it holds $\dvgphin v_{n+1} 
= \dvgphin v_\ell + \dvgphin w^1_c$, and recalling \eqref{eq:div v_ell} and \eqref{eq:div v_n+1}
\begin{align} \label{eq:divergence} 
\dvgphin v_{n+1}
&= 
\dvgphin v_\ell - \left( \dvgphin v_n \right) \ast \chi_\ell
\\ 
&\quad+ \nonumber
\left( \dvgphin v_n \right) \ast \chi_\ell
-
\left(\dvg^{\phi_n}v_n\right)   \ast \chi_\ell
\\
&\quad+ \nonumber
\left(\dvg^{\phi_n}Q^{\phi_n}v_n\right) \ast \chi_\ell
-
\dvg^{\phi_n} \left( (Q^{\phi_n}v_n) \ast \chi_\ell \right)
\\
&\quad+  \nonumber
\dvg^{\phi_n} \left( (Q^{\phi_n}v_n) \ast \chi_\ell \right)
-
\dvgphin \left( (Q^{\phi_n}v_n) \ast \chi_\ell \right).
\end{align}

\begin{prop} \label{prop:diverg}
For every $\delta$ sufficiently small and $\alpha' \in (0,\alpha)$ there exists a constant $C$ such that for all $L \in \N$, $L \geq 1$ it holds almost surely
\begin{align*}
\|\dvgphin v_{n+1}\|_{C_{\leq \mathfrak{t}_L} B^{-1}_{\infty,\infty}} 
\leq 
C L^{10} L_n (D_n^{1+2\delta}\ell^\alpha + D_n^{\delta} (n+1) \varsigma_n^{\alpha'}).
\end{align*}
\end{prop}

\begin{proof}
We refer to decomposition \eqref{eq:divergence} above.
By \autoref{lem:commuting_moll}, \autoref{lem:schauder} and iterative assumption \eqref{eq:C^1,1} it holds
\begin{align*}
\left\|\dvgphin v_\ell - \left( \dvgphin v_n \right) \ast \chi_\ell \right\|_{C_{\leq \mathfrak{t}_L}C_x}
&\leq
CL^2 \|v_n\|_{C_{\leq \mathfrak{t}_L}C^1_x} \ell^{\alpha}
\\
&\leq
CL^2L_n D_n \ell^{\alpha},
\\
\left\|
\left(\dvg^{\phi_n}Q^{\phi_n}v_n\right) \ast \chi_\ell
-
\dvg^{\phi_n} \left( (Q^{\phi_n}v_n) \ast \chi_\ell \right) \right\|_{C_{\leq \mathfrak{t}_L}C_x}
&\leq
CL^2 \|Q^{\phi_n} v_n\|_{C_{\leq \mathfrak{t}_L}C^1_x} \ell^{\alpha}
\\
&\leq
CL^{4+2\delta} \|v_n\|_{C_{\leq \mathfrak{t}_L}C^{1+\delta}_x} \ell^{\alpha}
\\
&\leq
CL^{4+2\delta}L_n D_n^{1+2\delta} \ell^{\alpha}.
\end{align*}
In the last line above we have used the bound $\|v_n\|_{C_{\leq \mathfrak{t}_L}C^{1+\delta}_x} \leq L_n D_n^{4/3}$, justified by the following observation. If $n=0$ we have defined $v_n=0$, thus let us assume $n \geq 1$ without any loss of generality. Then, $v_n$ was constructed from $v_{n-1},q_{n-1},\mathring{R}_{n-1}$ by the formula $v_n = v_{n-1} \ast \chi_{\ell_{n-1}} + w_{o,n-1} + w_{c,n-1}$, and \autoref{lem:w_c}, \autoref{lem:w_o} and the assumption $\lambda_{n-1} \leq D_n$ imply for $\delta$ sufficiently small
\begin{align*}
\|v_n\|_{C_{\leq \mathfrak{t}_L}C^{1+\delta}_x}
\leq
C L^{1+\delta} L_{n-1}^{6+3\delta} \lambda_{n-1}^{1+\delta} \delta_{n-1}^{1/2} \mu_{n-1}^\delta \varsigma_n^{\delta(\alpha-1)}
\leq
C L_n \lambda_{n-1}^{1+2\delta}
\leq
C L_n D_n^{1+2\delta}.
\end{align*}

Let us move to the remaining terms in \eqref{eq:divergence}. Using \autoref{lem:G_holder} and the convolution inequality \autoref{lem:conv_Besov} 
\begin{align*}
\left\|
\left( \dvgphin v_n \right) \ast \chi_\ell
-
\left(\dvg^{\phi_n}v_n\right)   \ast \chi_\ell
\right\|_{C_{\leq \mathfrak{t}_L}B^{-1}_{\infty,\infty}}
&\leq C
\left\|
\left( \dvgphin - \dvg^{\phi_n}\right)v_n
\right\|_{C_{\leq \mathfrak{t}_L}B^{\delta-1}_{\infty,\infty}}
\\
&\leq
CL^3\|v_n\|_{C_{\leq \mathfrak{t}_L}C^\delta_x}(n+1)\varsigma_n^{\alpha'}
\\
&\leq
CL^3L_nD_n^\delta(n+1)\varsigma_n^{\alpha'},
\end{align*}
and
\begin{align*}
&\left\|
\dvg^{\phi_n} \left( (Q^{\phi_n}v_n) \ast \chi_\ell \right)
-
\dvgphin \left( (Q^{\phi_n}v_n) \ast \chi_\ell \right)
\right\|_{C_{\leq \mathfrak{t}_L}B^{-1}_{\infty,\infty}}
\\
&\qquad\leq
CL^3 \| (Q^{\phi_n}v_n) \ast \chi_\ell \|_{C_{\leq \mathfrak{t}_L} C^\delta_x}(n+1)\varsigma_n^{\alpha'}
\\
&\qquad\leq
CL^{10} L_n \delta_{n+2}^{6/5}(n+1) \varsigma_n^{\alpha'}.
\end{align*}

\end{proof}

\subsection{Estimate on the pressure} \label{ssec:pres}
In the previous subsections we have collected several estimates that together allow to control iteratively the Reynold stress and the velocity field.
In order to prove our main iterative proposition we still need to provide suitable bounds on the pressure term, which we intend to do in this subsection.

Recall the definition of the new pressure $q_{n+1}$ and the energy pumping term $\tilde{\rho}_\ell$
\begin{align*}
q_{n+1} 
&= 
q_\ell - \frac12 \left(|w_o|^2-\tilde{\rho}_\ell \right),
\\
\tilde{\rho}_\ell(x,t) 
&=
\frac{2}{r_0} \sqrt{\eta^2  \delta_{n+1}^2 + |\mathring{R}_\ell(x,t)|^2}.
\end{align*}

We have:
\begin{prop} \label{prop:it_pres}
There exists a constant $M_q$ as in the statement of \autoref{prop:it} such that for every $L \in \N,$ $L \geq 1$ it holds almost surely
\begin{align*}
\| q_{n+1}-q_n \|_{C_{\leq \mathfrak{t}_L}C_x}
&\leq
M_q L_n \delta_n.
\end{align*}
Moreover, Let $r\geq r_\star + 2$. Then, for every $\delta>0$ there exists a constant $C$ such that for every $L \in \N,$ $L \geq 1$ it holds almost surely
\begin{align*}
\| q_{n+1}-q_n \|_{C^1_{\leq \mathfrak{t}_L,x}}
&\leq
C L_n^{3r+10} \lambda^{1+\delta} \delta_n \mu^\delta \varsigma_{n+1}^{\delta(\alpha-1)}.
\end{align*}
\end{prop}

\begin{proof}
Let us rewrite $q_{n+1}-q_n = q_\ell - q_n - \frac12 \left(|w_o|^2-\tilde{\rho}_\ell \right)$.
Standard mollification estimates and the iterative assumption \eqref{eq:C^1,1} yield
\begin{align*}
\| q_\ell-q_n \|_{C_{\leq \mathfrak{t}_L}C_x}
&\leq
CL_n \ell D_n ,
\\
\| q_\ell-q_n \|_{C^1_{\leq \mathfrak{t}_L,x}}
&\leq
C L_n D_n.
\end{align*}
The previous inequalities hold with a constant $C$ depending only on the mollifier $\chi$, which however can be thought of as fixed.
Notice that $\ell D_n \leq \delta_n$ by assumption.

In addition, the energy pumping term $\tilde{\rho}_\ell$ is easily controlled with estimates
\begin{align*}
\| \tilde{\rho}_\ell  \|_{C_{\leq \mathfrak{t}_L}C_x}
&\leq
\tfrac{3}{r_0} L_n \delta_{n+1},
\\
\| \tilde{\rho}_\ell \|_{C^1_{\leq \mathfrak{t}_L,x}}
&\leq
C L_n^2 \delta_{n+1} \ell^{-1},
\end{align*}
where $r_0$ has been defined in \autoref{ssec:w_o} and we have used $\eta<1$. 
Also, by \eqref{eq:w_o_bound} we have 
\begin{align*}
\| |w_o|^2  \|_{C_{\leq \mathfrak{t}_L}C_x}
&\leq
C L_n \delta_n,
\end{align*}
where the constant $C$ may depend on $\overline{e}$.
Thus, we have proved
\begin{align*}
\| q_{n+1}-q_n \|_{C_{\leq \mathfrak{t}_L}C_x}
&\leq
M_q L_n \delta_n
\end{align*}
for some constant $M_q$ depending only on $\overline{e}$.
Finally, by \autoref{lem:w_o} it holds
\begin{align*}
\| |w_o|^2  \|_{C^1_{\leq \mathfrak{t}_L,x}}
&\leq
C L_n^{3r+10} \lambda^{1+\delta} \delta_n \mu^\delta \varsigma_{n+1}^{\delta(\alpha-1)},
\end{align*}
completing the proof.
\end{proof}

\subsection{Estimate on the energy} \label{ssec:energy}
Finally, we have to check the iterative condition on the energy. This is the content of the following:
\begin{prop} \label{prop:it_energy}
Recall the definition of $r_0$ from \autoref{ssec:w_o}.
Up to choosing $C_\varsigma, C_\mu$ large enough, the following holds true almost surely:
\begin{align*}
\left| e(t)(1-\delta_{n+1}) - \int_{\mathbb{T}^3} |v_{n+1}(x,t)|^2 dx \right| 
\leq 
\frac{10\, \eta}{r_0}  \delta_{n+1},
\quad
\forall t \leq \mathfrak{t}.
\end{align*}
\end{prop}
\begin{proof}
Rewrite 
\begin{align} \label{eq:v_{n+1}_squared}
|v_{n+1}|^2
&=
|v_\ell|^2
+
|w_o|^2
+
|w_c|^2
+
2 v_\ell \cdot w_o
+
2 v_\ell \cdot w_c
+
2 w_o \cdot w_c.
\end{align}
Since $\chi_\ell \geq 0$ and by assumption \eqref{eq:est_energy}, H\"older inequality yields 
\begin{align*}
\int_{\mathbb{T}^3} |v_\ell(x,t)|^2 dx
&=
\int_{\mathbb{T}^3} \left| \int_{\mathbb{T}^3 \times \R} v_n(x-y,t-s) \chi_\ell(y,s) dy ds\right|^2 dx
\\
&\leq
\int_{\mathbb{T}^3} \int_{\mathbb{T}^3 \times \R} |v_n(x-y,t-s)|^2 \chi_\ell(y,s) dy ds  dx
\\
&\leq
\int_{\mathbb{T}^3 \times \R} e(t-s) \chi_\ell(y,s) dy ds 
\leq 
\overline{e},
\end{align*}
implying by \autoref{lem:w_c}
\begin{align*}
\left| \int_{\mathbb{T}^3} v_\ell(x,t) \cdot w_c(x,t) dx \right|
\leq
C \overline{e}^{1/2} \|w_c\|_{C_{\leq \mathfrak{t}}C_x}
\leq 
C \delta_{n+2}^{6/5},
\quad
\forall t \leq \mathfrak{t},
\end{align*}
where the constant $C$ depends only on $\overline{e}$.
Moreover, recalling 
\begin{align*}
w_o(x,t)
=
\sum_{|k|=\lambda_0}
a_k(x,t,\lambda t) E_ke^{ik\cdot \lambda\phi_{n+1}(x,t)}
\end{align*}
and applying \eqref{e:average} and \autoref{lem:C0_Cm} we obtain almost surely for every $t \leq \mathfrak{t}$
\begin{align*}
\left| \int_{\mathbb{T}^3} v_\ell(x,t) \cdot w_o(x,t) dx \right|
&\leq 
C\lambda^{-1}\sum_{|k|=\lambda_0}
[v_\ell \cdot a_k]_{C^1_x}
\\
&\leq
C\lambda^{-1}\sum_{|k|=\lambda_0}
\left(
\|v_\ell\|_{C_{\leq \mathfrak{t}} C_x}
[a_k]_{C_{\leq \mathfrak{t}} C^1_x}
+
[v_\ell]_{C_{\leq \mathfrak{t}} C^1_x}
\|a_k\|_{C_{\leq \mathfrak{t}} C_x}
\right)
\\
&\leq
C\lambda^{-1}\sum_{|k|=\lambda_0}
\left(
[a_k]_{C_{\leq \mathfrak{t}} C^1_x}
+
D_n \|a_k\|_{C_t C_x}
\right)
\\
&\leq C \lambda^{-1}
\mu \delta_n^{1/2} ( D_n + \varsigma_{n+1}^{\alpha-1} ).
\end{align*}
We also point out that the previous bounds do not contain any factor $L$, since we have purposefully restricted ourselves to times $t \leq \mathfrak{t} = \mathfrak{t}_1$.
Therefore, recalling \eqref{eq:v_{n+1}_squared} and the bounds just obtained, we have for $\delta$ sufficiently small
\begin{align*}
\left|
\int_{\mathbb{T}^3} \left( |v_{n+1}(x,t)|^2 - |v_\ell(x,t)|^2 - |w_o(x,t)|^2  \right) dx \right| 
&\leq
C
\| w_c \|_{C_x} \left( \| w_c \|_{C_x} + \|w_o\|_{C_x} \right)
\\
&\quad+
C \delta_{n+2}^{6/5}
+
C \lambda^{-1} \mu \delta_n^{1/2} ( D_n + \varsigma_{n+1}^{\alpha-1} )
\\
&\leq
C \lambda^\delta \delta_n^{1/2} \delta_{n+2}^{6/5}.
\end{align*}
Moreover, recalling \autoref{cor:U},
\begin{align*}
\int_{\mathbb{T}^3} |w_o|^2 dx
&=
\int_{\mathbb{T}^3} Tr(R_\ell) dx
+
\sum_{1 \leq |k| \leq 2\lambda_0}\int_{\mathbb{T}^3} Tr(U_k) e^{ik \lambda \phi_{n+1}} dx,
\end{align*}
and
\begin{align*}
\int_{\mathbb{T}^3} Tr(R_\ell) dx
&=
3\int_{\mathbb{T}^3} \rho_\ell dx
=
e(t)(1-\delta_{n+1}) - \int_{\mathbb{T}^3} |v_\ell|^2 dx
+
3\int_{\mathbb{T}^3} \tilde\rho_\ell dx ,
\end{align*}
we can rearrange terms to get
\begin{align*}
\left| e(t)(1-\delta_{n+1}) - \int_{\mathbb{T}^3} |v_{n+1}|^2 dx\right|
&\leq
\left|
\int_{\mathbb{T}^3} \left( |v_{n+1}|^2 - |v_\ell|^2 - |w_o|^2  \right) dx \right|
+
3\left|\int_{\mathbb{T}^3} \tilde\rho_\ell dx\right|
\\
&\quad+ \sum_{1 \leq |k| \leq 2\lambda_0}
\left| \int_{\mathbb{T}^3} Tr(U_k) e^{ik \lambda \phi_{n+1}} dx \right|
\\
&\leq
C \lambda^\delta \delta_n^{1/2} \delta_{n+2}^{6/5}
+
9 r_0^{-1} \eta \delta_{n+1}
+
C \lambda^{-1} \mu \delta_n  \varsigma_{n+1}^{\alpha-1}.
\\
&\leq
10 r_0^{-1} \eta \delta_{n+1}.
\end{align*}

In the expression above we have taken $\delta$ sufficiently small, $\alpha_\star<\alpha$ sufficiently close to $1/2$, and $C_\varsigma$, $C_\mu$ sufficiently large, so that the factor $10 r_0^{-1} \eta$ appears on the right-hand-side.
The proof is complete.
\end{proof}

\section{Proof of \autoref{prop:it}} \label{sec:proof}
We are finally ready to prove our main proposition.
Progressive measurability of the approximate solution at level $n+1$ descends directly from the definition of $v_{n+1},q_{n+1},\mathring{R}_{n+1}$.

We verify first the iterative assumption on the energy \eqref{eq:est_energy} and the $C_{\leq \mathfrak{t}_L} C_x$ norms, namely \eqref{eq:est_Reynolds} and \eqref{eq:est_pres}.
As for the former, we have proved in \autoref{prop:it_energy} that almost surely for every $t \leq \mathfrak{t}$
\begin{align*}
\left| e(t)(1-\delta_{n+1}) - \int_{\mathbb{T}^3} |v_{n+1}(x,t)|^2 dx \right| 
\leq 
\tfrac{10\,\eta }{r_0} \delta_{n+1}
\leq
\tfrac14  \delta_{n+1}\underline{e}
\leq
\tfrac14 \delta_{n+1}e(t) , 
\end{align*}
where the second inequality holds true if we take $\eta$ small enough. Recall also the upper bound on $\eta$ given by \autoref{lem:bounds_tilde_e}. 
The value of $\eta$ will be fixed hereafter according to these constraints.

Let us move to \eqref{eq:est_Reynolds}. By \eqref{eq:R&q} and \eqref{eq:decomposition_R} it holds
\begin{align*}
\mathring{R}_{n+1}
=
\mathring{R}^{trans}
+
\mathring{R}^{moll}
+
\mathring{R}^{osc}
+
\mathring{R}^{flow}
+
\mathring{R}^{comp}.
\end{align*}
Putting together \autoref{prop:R_trans}, \autoref{prop:R_moll}, \autoref{prop:R_osc}, \autoref{prop:R_flow} and \autoref{prop:R_comp} from \autoref{sec:iter} we can bound, almost surely for every $t \leq \mathfrak{t}_L$, $L \in \N$, $L \geq 1$:
\begin{align*}
\|\mathring{R}_{n+1}\|_{C_{\leq \mathfrak{t}_L} C_x} 
&\leq
CL_n^{3r+5}\lambda^{\delta} \mu^{-1} \delta_n^{1/2}
\\
&\quad+
CL_n^3 D_n \ell^\alpha
\\
&\quad+
CL_n^{3r+6} \lambda^{\delta-1} \mu  \delta_n \varsigma_{n+1}^{\alpha-1}
\\
&\quad+
CL_n^{3r+5} \ell^{-\delta} (n+1)\varsigma_n^{\alpha'}
\\
&\quad+
C L_n^{3r+6} \lambda^\delta \delta_n^{1/2} \delta_{n+2}^{6/5}
,
\end{align*}
where $r \geq r_\star+2$, $\delta>0$ is small, $\alpha' \in (0,\alpha)$ is close to $1/2$, and $C$ is an unimportant constant possibly depending on $\overline{e},K_0, \eta, M_v$ and $r_\star$.
Now we can choose $\mu,\lambda$ large such that the first and third line on the right-hand-side of the inequality above are small; then, $D_n \ell^\alpha$ and $(n+1)\varsigma_{n+1}^{\alpha'}$ are small by definition of $\ell,\varsigma_n$.
More precisely, up to choosing $\delta$ small enough, $\alpha_\star < \alpha' < \alpha$ sufficiently close to $1/2$, and $C_\varsigma,C_\mu$ large enough, the expression above can be rewritten as
\begin{align*}
\|\mathring{R}_{n+1}\|_{C_{\leq \mathfrak{t}_L} C_x} 
&\leq
\eta L_n^{3r+6} \delta_{n+2}.
\end{align*}

Thus the desired estimate holds true as soon as we choose 
\[
m \geq 3r + 6 \geq 3r_\star + 12.
\]

Iterative assumption \eqref{eq:est_pres} follows immediately from \autoref{prop:it_pres}, whereas the bound \eqref{eq:est_div} on $\dvgphin v_{n+1}$ comes from \autoref{prop:diverg}, up to choosing $C_\ell,C_\varsigma$ large enough.

Let us check
\begin{align*}
\| v_{n+1}-v_n \|_{C_{\leq \mathfrak{t}_L}C_x}
\leq
M_v L_n^4 \delta_n^{1/2}.
\end{align*}
It holds $v_{n+1}-v_n = v_\ell - v_n + w_o + w_c$, and by mollification and \eqref{eq:w_o_bound} we have 
\begin{align*}
\| v_\ell - v_n \|_{C_{\leq \mathfrak{t}_L}C_x}
&\leq
C L_n \ell D_n
\leq
C L_n \delta_{n},
\\
\| w_o \|_{C_{\leq \mathfrak{t}_L}C_x}
&\leq
C L_n^{1/2} \delta_n^{1/2},
\end{align*}
for some constant $C$ depending only on $\overline{e}$.
Moreover, by \autoref{lem:w_c} it holds
\begin{align*}
\| w_c \|_{C_{\leq \mathfrak{t}_L}C_x}
&\leq
CL_n^{3+4\delta} \delta_{n+2}^{6/5},
\end{align*}
for some universal constant $C$, and thus the desired estimate holds true for some $M_v$ depending only on $\overline{e}$.

Finally, we need to prove
\begin{align*}
\max \left\{ 
\|v_{n+1}\|_{C^1_{\leq \mathfrak{t}_L,x}},
\|q_{n+1}\|_{C^1_{\leq \mathfrak{t}_L,x}},
\|\mathring{R}_{n+1}\|_{C_{\leq \mathfrak{t}_L}C^1_x}
\right\}
&\leq 
A L_{n+1} \delta_n^{1/2} \left( \frac{D_n}{\delta_{n+4}} \right)^{1+\varepsilon}.
\end{align*}

Let us start with $\|\mathring{R}_{n+1}\|_{C_{\leq \mathfrak{t}_L}C^1_x}$. Collecting the results of \autoref{sec:iter}, we have
\begin{align*}
\|\mathring{R}_{n+1}\|_{C_{\leq \mathfrak{t}_L}C^1_x}
\leq
C L_n^{3r+10} \delta_n^{1/2} \lambda^{1+2\delta} .
\end{align*}
Therefore, we can choose for instance
\begin{align*}
r_\star = 7,
\qquad
m \geq 3r +10 &\geq 3r_\star + 16 = 37,
\qquad
\varepsilon \geq 15,
\end{align*}
and $A$ sufficiently large so that $\lambda^{1+2\delta} \leq A \left( \frac{D_n}{\delta_{n+4}} \right)^{1+\varepsilon} \leq D_{n+1}$.
Indeed, up to multiplicative constants depending on $C_\ell, C_\varsigma$ and $C_\mu$ it holds for $\alpha_\star$ sufficiently close to $1/2$ and $\delta$ small enough
\begin{align*}
\lambda^{1+2\delta}
\lesssim
\frac{D_n^{15}}{\delta_{n+3}^{20} \delta_{n+4}^5},
\end{align*}  
and we can use $A$ to absorb any multiplicative constant in the inequality above.

Next, by \autoref{prop:it_pres} and the iterative assumption \eqref{eq:C^1,1}
\begin{align*}
\| q_{n+1} \|_{C^1_{\leq \mathfrak{t}_L,x}}
\leq
CL_n D_n +
CL_n^{3r+10} \lambda^{1+2\delta} \delta_n
\leq
A L_{n+1} \delta_n^{1/2} \left( \frac{D_n}{\delta_{n+4}} \right)^{1+\varepsilon}.
\end{align*} 
As for the spatial norm of $v_{n+1}$, we recall by \autoref{lem:w_c} and \autoref{lem:w_o}
\begin{align*}
\|v_{n+1}\|_{C_{\leq \mathfrak{t}_L}C^1_x}
&\leq
C L_n^7 \delta_n^{1/2} \lambda^{1+2\delta},
\end{align*}
whereas for the time derivative of $v_{n+1}$ we take advantage of
\begin{align*}
\partial_t v_{n+1}
=
\dvgphin \left( \mathring{R}_{n+1} - v_{n+1} \otimes v_{n+1} - q_{n+1} Id \right)
\end{align*}
and the bounds on the $C_{\leq \mathfrak{t}_L}C^1_x$ norms of $v_{n+1}$, $q_{n+1}$ and $\mathring{R}_{n+1}$ just proved, up to choosing $m \geq 3r+11$ to compensate for an additional factor $L_n$.   

To conclude, take for instance $r_\star=7$, $m=38$ and $\varepsilon=15$.
The resulting H\"older exponent $\vartheta$ is given by \eqref{eq:vartheta} and \eqref{eq:vartheta'}: we can choose any
\begin{align*}
\vartheta
<
\frac{1}{2cb+2}
=
\frac{m-1}{2(1+\varepsilon)(m+\varepsilon)^5+m-2-\varepsilon}
=
\frac{38-1}{2(1+15)(38+15)^5+38-2-15}.
\end{align*}
The upper threshold is approximatively given by $2.76 \times 10^{-9}$, which is extremely low and far from the Onsager's critical exponent $1/3$ of the deterministic case. Since we have used a suboptimal choice of parameters during our construction for the only sake of simplicity, this exponent could be slightly improved taking more care of that.
However, we do not believe that a significant improvement on the exponent is within the reach of the techniques used in this paper.
Finally, if we restrict ourselves to local solutions, namely on the time interval $t \leq \mathfrak{t}_1$, we do not need to impose $m \geq 3r+11 \geq 3r_\star+17$, and we can just take $m=4$. This gives upper threshold for local solutions approximatively equal to $3.78 \times 10^{-8}$.

\appendix 

\section{Equivalence between \eqref{eq:euler} and \eqref{eq:random_euler}}

\begin{prop} \label{prop:equivalence}
A process $(u,p)$ is a solution to \eqref{eq:euler} in the sense of \autoref{def:sol_euler} if and only if $(v,q)$ given by \eqref{eq:change_variables} is a solution to \eqref{eq:random_euler} in the sense of \autoref{def:sol_random_euler}.
\end{prop}
\begin{proof}
It is clear that progressive measurability and regularity of trajectories are preserved by the changes of coordinates \eqref{eq:change_variables} and \eqref{eq:change_variables_inv}, as well as the divergence-free conditions.
Therefore, we only need to check the conditions on the time evolution of the coupling between solutions and test functions.
Let $(v,q)$ be a solution to \eqref{eq:random_euler} in the sense of \autoref{def:sol_random_euler}, and let $h:\Omega \to C_{loc}([0,\infty), C^\infty(\mathbb{T}^3,\R^3)$ be a square integrable semimartingale satisfying 
\begin{align*}
dh=H_0\,dt + \sum_{k \in I} H_k \bullet dB^k,
\end{align*}
for some progressively measurable processes $H_0, \{H_k\}_{k \in I} : \Omega \to C_{loc}([0,\infty),C^\infty(\mathbb{T}^3,\R^3))$.
Since $\phi$ is measure preserving, and recalling $v=u\circ \phi$ we have $\int_{\mathbb{T}^3} u \cdot h =\int_{\mathbb{T}^3} v \cdot (h\circ\phi)$.
Moreover, the process $h \circ \phi$ is a semimartingale satisfying
\begin{align*}
d(h\circ\phi)
=
(H_0 \circ \phi) \,dt
+
\sum_{k \in I}
((H_k + (\sigma_k\cdot\nabla) h) \circ\phi) \bullet dB^k(t),
\end{align*}
and therefore 
\begin{align*}
d\left( \int_{\mathbb{T}^3} u \cdot h \right)
&=
\int_{\mathbb{T}^3}
v \cdot (H_0 \circ \phi + (v \cdot \nabla^\phi)(h \circ \phi))  dt
+
\int_{\mathbb{T}^3} q \, \dvgphi (h\circ \phi) \,dt
\\
&\quad+ 
\sum_{k \in I}\int_{\mathbb{T}^3}
v \cdot \left((H_k+(\sigma_k\cdot\nabla) h)\circ\phi \right) \bullet dB^k
\\
&=
\int_{\mathbb{T}^3}
u \cdot (H_0 + (u \cdot \nabla) h)  dt
+
\int_{\mathbb{T}^3} p \, \dvg h \,dt
\\
&\quad+
\sum_{k \in I}\int_{\mathbb{T}^3}
u \cdot (H_k + (\sigma_k \cdot \nabla) h) \bullet dB^k.
\end{align*}
The converse implication can be proved in a similar fashion, noticing that
\begin{align*}
d(h \circ \phi^{-1})
&=
(H_0 \circ \phi^{-1})\,dt
+
\sum_{k \in I}
(H_k \circ \phi^{-1} - (\sigma_k \cdot \nabla)(h \circ \phi^{-1})) \bullet dB^k.
\end{align*}
\end{proof}

\section{Proof of \autoref{prop:aux_est}} \label{app:proof_prop}

In this section all the multiplicative constants $C$ may depend on $\overline{e},\eta,M_v,r_\star$ without explicit mentioning, but will be always independent of $s,\tau$.
In particular, when no ambiguity may occur we drop the symbol $C_{\leq \mathfrak{t}_L}$ in the uniform-in-time spatial H\"older norms $C_{\leq \mathfrak{t}_L}C_x^r$ below.

We first recall the following estimate for the H\"older norm of products. Let $f,g$ be smooth functions on the torus and take $r \in [0,1]$; then it is immediate to see
\begin{align*}
\| fg \|_{C^r_x} \leq \|f\|_{C^r_x} \|g\|_{C_x}+\|f\|_{C_x} \|g\|_{C^r_x}.
\end{align*}
For $r \in \N$, $r > 1$, using the Leibniz rule for the derivative of the product, we have for every integer $k \leq r$ a constant $C=C(k)$ such that
\begin{align*}
[fg]_{C^{k}_x}
&\leq C
\sum_{r_1+r_2=k}
[ f ]_{C^{r_1}_x}
[ g ]_{C^{r_2}_x}
\leq C
\sum_{r_1+r_2=k}
\| f \|_{C^{r_1}_x}
\| g \|_{C^{r_2}_x},
\qquad
r_1,r_2 \in \N,
\end{align*}
and taking the sum over $k \leq r$ we get a constant $C=C(r)$, increasing in $r$, such that 
\begin{align*} 
\|fg \|_{C^{r}_x}
&\leq C
\sum_{r_1+r_2=r}
\| f \|_{C^{r_1}_x}
\| g \|_{C^{r_2}_x},
\qquad
r_1,r_2 \in \N.
\end{align*} 

Therefore, recalling \eqref{eq:a_k},
\begin{align*}
a_k(y,s,\tau)
&=
\mathbf{1}_{\{k \in \Lambda_j\}} 
\sqrt{\rho_\ell(y,s)} 
\,
\underbrace{\gamma_k^{(j)}\left(\frac{R_\ell(y,s)}{\rho_\ell(y,s)}\right)}_{\eqqcolon
\Gamma(y,s)}
\,
\underbrace{\psi_k^{(j)} \left(\tilde{v}(y,s),\tau\right)}_{\eqqcolon
\Psi(y,s,\tau)},
\end{align*}
for any $s \leq \mathfrak{t}_L$ and $\tau \leq \lambda \mathfrak{t}_L$ we have
\begin{align}
\|a_k(\cdot,s,\tau)\|_{C^r_x}
&\leq \label{eq:|a_k|_C^r_small}
\|\sqrt{\rho_\ell}\|_{C^r_x}
\|\Gamma\|_{C_x}
\|\Psi\|_{C_x}
\\
&\quad+ \nonumber
\|\sqrt{\rho_\ell}\|_{C_x}
\|\Gamma\|_{C^r_x}
\|\Psi\|_{C_x}
\\
&\quad+ \nonumber
\|\sqrt{\rho_\ell}\|_{C_x}
\|\Gamma\|_{C_x}
\|\Psi\|_{C^r_x},
&r \in [0,1],\quad
\\
\|a_k(\cdot,s,\tau)\|_{C^r_x}
&\leq \label{eq:|a_k|_C^r}
C \sum_{r_1+r_2+r_3=r}
\|\sqrt{\rho_\ell}\|_{C^{r_1}_x}
\|\Gamma\|_{C^{r_2}_x}
\|\Psi\|_{C^{r_3}_x},
&r \in \N, r > 1.
\end{align}

Also, notice that since $a_k$ depends on $\tau$ only via $\psi_k^{(j)}$,
\begin{align*} 
\partial_\tau a_k
&=
\mathbf{1}_{\{k \in \Lambda_j\}} 
\sqrt{\rho_\ell} \, \Gamma \, \partial_\tau \Psi,
\end{align*}
and thus
\begin{align} 
\|\partial_\tau a_k(\cdot,s,\tau)\|_{C^r_x}
&\leq \label{eq:Dtau(a_k)_small}
\|\sqrt{\rho_\ell}\|_{C^r_x}
\|\Gamma\|_{C_x}
\|\partial_\tau \Psi\|_{C_x}
\\
&\quad+ \nonumber
\|\sqrt{\rho_\ell}\|_{C_x}
\|\Gamma\|_{C^r_x}
\|\partial_\tau \Psi\|_{C_x}
\\
&\quad+ \nonumber
\|\sqrt{\rho_\ell}\|_{C_x}
\|\Gamma\|_{C_x}
\|\partial_\tau \Psi\|_{C^r_x},
&r \in [0,1],\quad
\\
\label{eq:Dtau(a_k)}
\|\partial_\tau a_k(\cdot,s,\tau)\|_{C^r_x}
&\leq
C \sum_{r_1+r_2+r_3=r}
\|\sqrt{\rho_\ell}\|_{C^{r_1}_x}
\|\Gamma\|_{C^{r_2}_x}
\|\partial_\tau \Psi\|_{C^{r_3}_x},
&r \in \N, r > 1.
\end{align}
Similarly, for $r \in [0,1]$
\begin{align} 
\|(\partial_\tau a_k + i(k \cdot \tilde{v})a_k)(\cdot,s,\tau)\|_{C^r_x}
&\leq \label{eq:Dmat(a_k)_small}
\|\sqrt{\rho_\ell}\|_{C^r_x}
\|\Gamma\|_{C_x}
\|\partial_\tau \Psi + i(k \cdot \tilde{v})\Psi\|_{C_x}
\\
&\quad+ \nonumber
\|\sqrt{\rho_\ell}\|_{C_x}
\|\Gamma\|_{C^r_x}
\|\partial_\tau \Psi + i(k \cdot \tilde{v})\Psi\|_{C_x}
\\
&\quad+ \nonumber
\|\sqrt{\rho_\ell}\|_{C_x}
\|\Gamma\|_{C_x}
\|\partial_\tau \Psi + i(k \cdot \tilde{v})\Psi\|_{C^r_x},
\end{align}
whereas for $r \in \N, r > 1$
\begin{align}
\label{eq:Dmat(a_k)}
\|(\partial_\tau a_k + i(k \cdot \tilde{v})a_k)(\cdot,s,\tau)\|_{C^r_x}
&\leq
C \sum_{r_1+r_2+r_3=r}
\|\sqrt{\rho_\ell}\|_{C^{r_1}_x}
\|\Gamma\|_{C^{r_2}_x}
\|\partial_\tau \Psi + i(k \cdot \tilde{v})\Psi\|_{C^{r_3}_x}.
\end{align}

Also, notice that $\| \rho_\ell \|_{C_x} \leq C L_n \delta_n$ and $\| \Gamma \|_{C_x},\| \Psi \|_{C_x} \leq C$ by construction.

\begin{proof}[Proof of \eqref{est:a_k,r_small}]
In view of \eqref{eq:|a_k|_C^r_small} we need to estimate the quantities $\sqrt{\rho_\ell}, \Gamma, \Psi$ in the H\"older space $C^r_x$, $r \in [0,1]$. For $r=0$ the bounds obtained are sub-optimal, but still sufficient for our future purposes and we do not treat this case separately for the sake of simplicity. 

\emph{Bound on $\|\sqrt{\rho_\ell}\|_{C^r_x}$}. Since $c\delta_n \leq \rho_\ell \leq CL_n \delta_n$ and $\tilde{\rho}_\ell \geq \eta\delta_{n+1}$ by construction, for every $s \leq \mathfrak{t}_L$ it holds 
\begin{align*}
\|\sqrt{\rho_\ell(\cdot,s)}\|_{C^r_x}
&\leq
C L_n^{1/2} \delta_n^{1/2}
+
C \delta_n^{-1/2} \|\rho_\ell(\cdot,s)\|_{C^r_x}
\\
&\leq
C L_n^{1/2} \delta_n^{1/2}
+
C \delta_n^{-1/2} L_n \|\mathring{R}_\ell(\cdot,s)\|_{C^r_x}
\\
&\leq
C L^2_n \delta_n^{-1/2} \delta_{n+1}^{1-r} D_n^r.
\end{align*}

\emph{Bound on $\|\Gamma\|_{C^r_x}$}. By definition of $\rho_\ell$, the range of the function $R_\ell / \rho_\ell$ is contained in $B_{\frac{r_0}{2}(Id)}$; since $\gamma_k^{(j)}$ is smooth on $B_{r_0(Id)}$ we can bound
\begin{align*}
\|\Gamma(\cdot,s)\|_{C^r_x}
\leq
C \left\| \frac{R_\ell(\cdot,s)}{\rho_\ell(\cdot,s)} \right\|_{C^r_x}
\leq
C
+
C \left\| \frac{\mathring{R}_\ell(\cdot,s)}{\rho_\ell(\cdot,s)} \right\|_{C^r_x},
\end{align*}
with the constant $C$ depending only on $\gamma_{k}^{(j)}$ and $r_0$.
Therefore for every $s \leq \mathfrak{t}_L$ we have
\begin{align*}
\|\Gamma(\cdot,s)\|_{C^r_x}
&\leq C+C 
\|\mathring{R}_\ell(\cdot,s)\|_{C_x} 
\|\rho_\ell^{-1}(\cdot,s)\|_{C^r_x}
+
C 
\|\mathring{R}_\ell(\cdot,s)\|_{C^r_x} 
\|\rho_\ell^{-1}(\cdot,s)\|_{C_x}.
\\
&\leq C 
L_n^3 \delta_n^{-1}\delta_{n+1}^{1-r} D_n^r.
\end{align*}

\emph{Bound on $\|\Psi\|_{C^r_x}$}.
Recall that $r\in [0,1]$, thus \eqref{e:phiestimate0} implies
\begin{align*}
\| \Psi(\cdot,s,\tau)\|_{C^r_x}
&\leq
\|\psi_k^{(j)}(\cdot,\tau)\|_{C^r_v}
\|\tilde{v}(\cdot,s)\|_{C^1_x}^r
\leq
C \mu^r \|\tilde{v}(\cdot,s)\|_{C^1_x}^r.
\end{align*}
By \autoref{lem:C0_Cm}, for every $s \leq \mathfrak{t}_L$ and $\tau \leq \lambda \mathfrak{t}_L$ it holds (recall that the very definition of $D_n,\varsigma_n$ implies $D_n \leq \varsigma_{n+1}^{\alpha-1}$)
\begin{align} \label{eq:tilde(v)_C1}
\|\tilde{v}(\cdot,s)\|_{C^1_x}
&\leq
\|v_\ell(\cdot,s)\|_{C^1_x}
+
\|\dot{\phi}_{n+1}(\cdot,s)\|_{C^1_x}
\\
&\leq \nonumber
C L_n D_n
+
CL \varsigma_{n+1}^{\alpha-1}
\\
&\leq \nonumber
CL_n \varsigma_{n+1}^{\alpha-1}.
\end{align} 
Putting all together we recover the desired bound for $\|a_k(\cdot,s,\tau)\|_{C^r_x}$, $r \in [0,1]$. 
\end{proof}
\begin{proof}[Proof of \eqref{est:Dmat(a_k),r_small}]
Thanks to \eqref{eq:Dmat(a_k)_small} and the estimates on $\|\sqrt{\rho_\ell(\cdot,s)}\|_{C^r_x}$, $\|\Gamma(\cdot,s)\|_{C^r_x}$ shown above, we only need to control the quantity 
$\|\partial_\tau \Psi + i(k \cdot \tilde{v})\Psi\|_{C^r_x}$, $r \in [0,1]$.
Since $\partial_\tau \Psi + i(k \cdot \tilde{v})\Psi$ depends on $y$ only through its argument $\tilde{v}$, it holds 
\begin{align*}
\|\partial_\tau \Psi + i(k \cdot \tilde{v})\Psi\|_{C^r_x}
&\leq
\|\partial_\tau \Psi + i(k \cdot \tilde{v})\Psi\|_{C^r_v}
\| \tilde{v} \|_{C^1_x}^r
\leq C \mu^{r-1}
\| \tilde{v} \|_{C^1_x}^r,
\end{align*}
where we have used \eqref{e:phiestimate2} with $r=0,1$ and interpolation. Using \eqref{eq:tilde(v)_C1} and collecting all the necessary bounds, we get the desired inequality.
\end{proof}
\begin{proof}[Proof of \eqref{est:a_k,r_large}]
Let us now move to the estimate for $\|a_k(\cdot,s,\tau)\|_{C^r_x}$, $r \in \N$, $r > 1$.
By \eqref{eq:|a_k|_C^r} we only need to bound $\sqrt{\rho_\ell}, \Gamma, \Psi$ in the H\"older spaces $C^{r_i}_x$, $r_i \in \N$, $r_i \geq 1$, $i=1,2,3$. 

\emph{Bound on $\|\sqrt{\rho_\ell}\|_{C^{r_1}_x}$}.
By \cite[Proposition 4.1]{DLS14} we have for every $s \leq \mathfrak{t}_L$
\begin{align*}
\| \sqrt{\rho_\ell} \|_{C^{r_1}_x}
&\leq 
C L_n^{1/2} \delta_n^{1/2}
+
C \sum_{j=1}^{r_1} \delta_n^{1/2-j} \|\rho_\ell \|_{C_x}^{j-1} \|\rho_\ell \|_{C^{r_1}_x}
\\
&\leq CL_n^{1/2} \delta_n^{1/2}
+ C
\delta_n^{-1/2}  L_n^{{r_1}-1} \|\rho_\ell \|_{C^{r_1}_x},
\\
\|\rho_\ell \|_{C^{r_1}_x}
&\leq CL_n \delta_n
+ C
\sum_{j=1}^{r_1} \delta_{n+1}^{1-j} L_n^j \|\mathring{R}_\ell\|_{C_x}^{j-1} \|\mathring{R}_\ell\|_{C^{r_1}_x}
\\
&\leq C L_n^{2{r_1}} D_n \ell^{1-{r_1}},
\end{align*}
from which we deduce
\begin{align*}
\| \sqrt{\rho_\ell} \|_{C^{r_1}_x}
&\leq C L_n^{3{r_1}-1} \delta_n^{-1/2} D_n \ell^{1-{r_1}}.
\end{align*}

\emph{Bound on $\|\Gamma\|_{C^{r_2}_x}$}.
Again by Leibniz rule for the derivative of a product,
\begin{align*}
\|\Gamma\|_{C^{r_2}_x}
&\leq
C \sum_{h_1+h_2={r_2}} \|\mathring{R}_\ell\|_{C^{h_1}_x} \|\rho_\ell^{-1}\|_{C^{h_2}_x}
,
\end{align*}
and \cite[Proposition 4.1]{DLS14} implies for $h_2 \geq 1$ and for every $s \leq \mathfrak{t}_L$ 
\begin{align*}
\|\rho_\ell^{-1}\|_{C^{h_2}_x}
&\leq
C \sum_{j=1}^{h_2} \delta_n^{-1-j} \|\rho_\ell\|_{C_x}^{j-1} \|\rho_\ell\|_{C^{h_2}_x}
\leq C L_n^{3{h_2}-1} \delta_n^{-2} D_n \ell^{1-{h_2}}.
\end{align*}
Therefore, since $\|\mathring{R}_\ell\|_{C_x} \leq L_n \delta_{n+1}$ and $\|\mathring{R}_\ell\|_{C^{h_1}_x} \leq L_n D_n \ell^{1-h_1}$ when $h_1 \geq 1$, for every $s \leq \mathfrak{t}_L$ it holds
\begin{align*}
\|\Gamma\|_{C^{r_2}_x}
&\leq
C L_n^{3{r_2}} \delta_n^{-1} D_n \ell^{1-{r_2}} .
\end{align*}

\emph{Bound on $\|\Psi\|_{C^{r_3}_x}$}.
As usual, we estimate for given $s \leq \mathfrak{t}_L$ and $\tau \leq \lambda \mathfrak{t}_L$
\begin{align*}
\|\Psi\|_{C^{r_3}_x}
&\leq
C+C \sum_{j=1}^{r_3} \mu^j \| \tilde{v} \|_{C_x}^{j-1} \|\tilde{v}\|_{C^{r_3}_x}
\\
&\leq 
C+C \sum_{j=1}^{r_3} \mu^j ( L_n^{1/2} + L \varsigma_{n+1}^{\alpha-1} )^{j-1} ( L_n D_n \ell^{1-{r_3}} + L\varsigma_{n+1}^{\alpha-1})
\\
&\leq
C L_n^{({r_3}+1)/2} \mu^{r_3} \varsigma_{n+1}^{({r_3}-1)(\alpha-1)}(D_n \ell^{1-{r_3}} + \varsigma_{n+1}^{\alpha-1}).
\end{align*}
We have used \autoref{lem:flow} in the second inequality.
Taking into account the bounds above we arrive to
\begin{align*}
\|a_k(\cdot,s,\tau)\|_{C^r_x}
&\leq
C L_n^{3r}  \delta_n^{1/2} \mu^r \varsigma_{n+1}^{(r-1)(\alpha-1)} (D_n \ell^{1-r} + \varsigma_{n+1}^{\alpha-1}).
\end{align*}
\end{proof}
\begin{proof}[Proof of \eqref{est:Dmat(a_k),r_large}]
As for the estimate \eqref{est:Dmat(a_k),r_large}, we only need to control $\|\partial_\tau \Psi + i(k \cdot \tilde{v})\Psi\|_{C^{r_3}_x}$, ${r_3} \in \N$, ${r_3} \geq 1$.
We proceed as follows. Using \cite[Proposition 4.1]{DLS14} and \eqref{e:phiestimate2}, we have for given $s \leq \mathfrak{t}_L$ and $\tau \leq \lambda \mathfrak{t}_L$
\begin{align*}
\|\partial_\tau \Psi + i(k \cdot \tilde{v})\Psi\|_{C^{r_3}_x}
&\leq
C
\sum_{j=1}^{r_3} \mu^{j-1} \|\tilde{v}\|_{C_x}^{j-1} \|\tilde{v}\|_{C^{r_3}_x}
\\
&\leq
C \mu^{{r_3}-1} L_n^{({r_3}+1)/2} \varsigma_{n+1}^{({r_3}-1)(\alpha-1)}(D_n \ell^{1-{r_3}} + \varsigma_{n+1}^{\alpha-1}).
\end{align*} 
Using \eqref{eq:Dmat(a_k)}, the conclusion is immediate.
\end{proof}
\begin{proof}[Proof of \eqref{est:Dtau(a_k),r_small} and \eqref{est:Dtau(a_k),r_large}]
By \eqref{e:phiestimate3} the bounds for $\partial_\tau \Psi$ are exactly the same as for $\Psi$, up to a multiplicative factor $\|\tilde{v}\|_{C_x} \leq CL_n^{1/2}\varsigma_{n+1}^{\alpha-1}$ in front. Since the largest power of $L$ comes from $\Gamma$, by \eqref{eq:Dtau(a_k)_small} and \eqref{eq:Dtau(a_k)} the desired estimates are obtained by \eqref{est:a_k,r_small} and \eqref{est:a_k,r_large} by multiplying with an additional factor $\varsigma_{n+1}^{\alpha-1}$.
\end{proof}
%
%
%
%
%
%
\begin{proof}[Proof of \eqref{est:Ds(a_k)0} and \eqref{est:Ds(a_k)}]
Let us now prove the estimates for $\partial_s a_k$. We have
\begin{align*}
\partial_s a_k
&=
\mathbf{1}_{\{k \in \Lambda_j\}} 
\partial_s\sqrt{\rho_\ell} \, \Gamma \, \Psi
+
\mathbf{1}_{\{k \in \Lambda_j\}} 
\sqrt{\rho_\ell} \, \partial_s\Gamma \, \Psi
+
\mathbf{1}_{\{k \in \Lambda_j\}} 
\sqrt{\rho_\ell} \, \Gamma \, \partial_s\Psi,
\end{align*}
so that 
\begin{align} \label{eq:Ds(a_k), C^r_small}
\| \partial_s a_k \|_{C_x} 
&\leq 
\|\partial_s\sqrt{\rho_\ell} \|_{C_x}
\| \Gamma \|_{C_x}
\| \Psi \|_{C_x}
\\
&\quad+ \nonumber
\| \sqrt{\rho_\ell} \|_{C_x}
\|\partial_s \Gamma \|_{C_x}
\| \Psi \|_{C_x}
\\
&\quad+ \nonumber
\|\sqrt{\rho_\ell} \|_{C_x}
\| \Gamma \|_{C_x}
\|\partial_s \Psi \|_{C_x},
\end{align} 
and for $r \in \N$, $r \geq 1$
\begin{align} \label{eq:Ds(a_k), C^r}
\| \partial_s a_k \|_{C^r_x}
&\leq
C\sum_{r_1+r_2+r_3=r}
\|\partial_s\sqrt{\rho_\ell} \|_{C^{r_1}_x}
\| \Gamma \|_{C^{r_2}_x}
\| \Psi \|_{C^{r_3}_x}
\\
&\quad+ \nonumber
C\sum_{r_1+r_2+r_3=r}
\| \sqrt{\rho_\ell} \|_{C^{r_1}_x}
\|\partial_s \Gamma \|_{C^{r_2}_x}
\| \Psi \|_{C^{r_3}_x}
\\
&\quad+ \nonumber
C\sum_{r_1+r_2+r_3=r}
\|\sqrt{\rho_\ell} \|_{C^{r_1}_x}
\| \Gamma \|_{C^{r_2}_x}
\|\partial_s \Psi \|_{C^{r_3}_x}.
\end{align}
The quantities $\| \sqrt{\rho_\ell} \|_{C^{r_1}_x}$, $\| \Gamma \|_{C^{r_2}_x}$ and $\| \Psi \|_{C^{r_3}_x}$ have been bounded during previous steps of the proof; we are left to provide analogous bounds for $\partial_s\sqrt{\rho_\ell}$, $\partial_s \Gamma $ and $\partial_s\Psi$.

\emph{Bounds on $\partial_s \sqrt{\rho_\ell}$}.
Rewrite, for fixed $s \leq \mathfrak{t}_L$:
\begin{align} \label{eq:bound_ds_sqrt_0}
\| \partial_s \sqrt{\rho_\ell} \|_{C_x} 
\leq 
C \delta_n^{-1/2} \| \partial_s \rho_\ell \|_{C_x},
\end{align}
and for every ${r_1} \geq 1$
\begin{align} \label{eq:bound_ds_sqrt_r}
\| \partial_s \sqrt{\rho_\ell} \|_{C^{r_1}_x}
&\leq 
C \sum_{h_1+h_2=r_1}
\|\partial_s \rho_\ell\|_{C^{h_1}_x} \|{\rho_\ell}^{-1/2}\|_{C^{h_2}_x}
.
\end{align}

It holds
\begin{align*}
\partial_s \rho_\ell
&=
\frac{\partial_s |\mathring{R}_\ell|^2}{r_0 \sqrt{\eta^2 \delta_{n+1}^2 + |\mathring{R}_\ell|^2}}
+
\partial_s \gamma_n
=
\frac{\sum_{i,j}\mathring{R}_\ell^{i,j} \partial_s \mathring{R}_\ell^{i,j}}{r_0 \sqrt{\eta^2 \delta_{n+1}^2 + |\mathring{R}_\ell|^2}}
+
\partial_s \gamma_n.
\end{align*}

Now, notice that by \eqref{eq:conditions_gamma3} we have $|\partial_s \gamma_n | \leq C\| \tilde{e} \|_{C^1_{\leq \mathfrak{t}}} \leq C\overline{e}+C\overline{e}^{1/2}D_n$, where the second inequality is justified by our iterative assumptions \eqref{eq:est_energy} and \eqref{eq:C^1,1}, and (recall $\chi_\ell \geq 0$ to apply Cauchy-Schwartz inequality to $v_n \chi_\ell^{1/2}$ and $\chi_\ell^{1/2}$)
\begin{align*}
\int_{\mathbb{T}^3} |v_\ell(x,t)|^2 dx
&=
\int_{\mathbb{T}^3} \left| \int_{\mathbb{T}^3 \times \R} v_n(x-y,t-s) \chi_\ell(y,s) dy ds\right|^2 dx
\\
&\leq
\int_{\mathbb{T}^3} \int_{\mathbb{T}^3 \times \R} |v_n(x-y,t-s)|^2 \chi_\ell(y,s) dy ds  dx
\\
&\leq
\int_{\mathbb{T}^3 \times \R} e(t-s) \chi_\ell(y,s) dy ds 
\leq 
\overline{e}.
\end{align*} 
Moreover, $\|\partial_s \mathring{R}_\ell \|_{C^h_x} \leq CL_n \delta_{n+1} \ell^{-1-h}$ for every $h \in \N$, by standard mollification estimates.
Hence, using \cite[Proposition 4.2]{DLS14} to bound the H\"older norm of the term under square root and recalling \eqref{eq:definition_ell}, we get
\begin{align*}
\|\partial_s \rho_\ell\|_{C_x} 
&\leq
C L_n^2 \delta_{n+1} \ell^{-1},
\\
\|\partial_s \rho_\ell\|_{C^{h_1}_x} 
&\leq
C L_n^{2h_1+2} \delta_{n+1} \ell^{-1-h_1}.
\end{align*} 

Plugging back into \eqref{eq:bound_ds_sqrt_0} and \eqref{eq:bound_ds_sqrt_r} we obtain
\begin{align*}
\| \partial_s \sqrt{\rho_\ell} \|_{C_x} 
&\leq 
C L_n^2 \delta_n^{-1/2} \delta_{n+1} \ell^{-1},
\\
\| \partial_s \sqrt{\rho_\ell} \|_{C^{r_1}_x}
&\leq
CL_n^{3{r_1}+1}\delta_n^{-1/2}\delta_{n+1} \ell^{-1-r_1}.
\end{align*}

\emph{Bounds on $\partial_s \Gamma$}.
Since $\gamma_k^{(j)}$ is smooth on $B_{r_0(Id)}$ and $R_\ell/\rho_\ell$ takes values in $B_{\frac{r_0}2(Id)}$, we can bound
\begin{align*}
\|\partial_s \Gamma\|_{C_x} 
&\leq C
\left\|\,\partial_s \frac{\mathring{R}_\ell}{\rho_\ell}  \right\|_{C_x}
,
\quad
\|\partial_s \Gamma\|_{C^{r_2}_x} 
\leq C
\left\|\,\partial_s \frac{\mathring{R}_\ell}{\rho_\ell}\right\|_{C^{r_2}_x}.
\end{align*} 

On the other hand,
\begin{align*}
\partial_s \frac{\mathring{R}_\ell}{\rho_\ell} 
&=
\frac{\partial_s \mathring{R}_\ell}{\rho_\ell} 
-
\frac{\mathring{R}_\ell \partial_s \rho_\ell}{\rho_\ell^2},
\end{align*}
therefore, taking also into account the previous estimates for $\rho_\ell$
\begin{align*}
\|\partial_s \Gamma\|_{C_x} 
&\leq
C L_n^3 \delta_n^{-1} \delta_{n+1} \ell^{-1},
\\
\|\partial_s \Gamma\|_{C^{r_2}_x} 
&\leq
C L_n^{3{r_2}+2} \delta_n^{-2} \delta_{n+1}^2 \ell^{-1-{r_2}}. 
\end{align*}

\emph{Bounds on $\partial_s \Psi$}.
Finally, since $\Psi$ depends on $s$ only through $\tilde{v}(y,s)$:
\begin{align*}
\partial_s \Psi
&=
D_v \psi_k^{(j)} (\tilde{v},\tau)
\cdot
\partial_s \tilde{v},
\end{align*}
and recalling \eqref{e:phiestimate0} and \cite[Proposition 4.1]{DLS14} for the H\"older seminorm of compositions, we get
\begin{align*}
\| \partial_s \Psi \|_{C_x}
&\leq
C \mu \| \partial_s \tilde{v} \|_{C_x},
\\
\|\partial_s \Psi \|_{C^{r_3}_x}
&\leq
C \mu \| \partial_s \tilde{v} \|_{C_x}
+
C|D_v \psi_k^{(j)}| \|\partial_s \tilde{v}\|_{C^{r_3}_x}
\\
&\quad+
C\sum_{i=1}^{r_3} |D_v^{1+i} \psi_k^{(j)}| \|\tilde{v}\|_{C_x}^{i-1} \|\tilde{v}\|_{C^{r_3}_x} \|\partial_s \tilde{v}\|_{C_x}
\\
&\leq
C \mu \| \partial_s \tilde{v} \|_{C^{r_3}_x}
+
C \mu^{{r_3}+1} L_n^{{r_3}} 
\varsigma_{n+1}^{({r_3}-1)(\alpha-1)}
(D_n \ell^{1-{r_3}} + \varsigma_{n+1}^{\alpha-1})
\| \partial_s \tilde{v} \|_{C_x}.
\end{align*} 
The partial derivative of $\tilde{v}$ with respect to $s$ is given by $\partial_s {v}_\ell + \ddot{\phi}_{n+1}$, therefore
\begin{align*}
\| \partial_s \tilde{v} \|_{C_x}
&\leq
C L_n D_n 
+
C L \varsigma_{n+1}^{\alpha-2}
\leq
C L_n \varsigma_{n+1}^{\alpha-2};
\\
\| \partial_s \tilde{v} \|_{C^{r_3}_x}
&\leq
C L_n D_n \ell^{-{r_3}}
+
C L \varsigma_{n+1}^{\alpha-2}
\leq
C L_n (D_n \ell^{-r_3} + \varsigma_{n+1}^{\alpha-2} ).
\end{align*}
Therefore
\begin{align*}
\| \partial_s \Psi \|_{C_x}
&\leq
C L_n \mu \varsigma_{n+1}^{\alpha-2},
\\
\|\partial_s \Psi \|_{C^{r_3}_x}
&\leq
C L_n^{r_3+1} \mu^{r_3+1} \varsigma_{n+1}^{r_3(\alpha-1)-1} (D_n \ell^{1-{r_3}} +\varsigma_{n+1}^{\alpha-1}).
\end{align*}

Collecting everything and plugging into \eqref{eq:Ds(a_k), C^r_small} and \eqref{eq:Ds(a_k), C^r} , we finally arrive to:
\begin{align*}
\| \partial_s a_k \|_{C_x}
&\leq
CL_n^{7/2} \delta_n^{1/2} \mu \varsigma_{n+1}^{\alpha-2},
\end{align*}
\begin{align*}
\| \partial_s a_k \|_{C^r_x}
&\leq
C L_n^{3r+5/2} \delta_n^{1/2}  \mu^{r+1} \varsigma_{n+1}^{r(\alpha-1)-1} (D_n \ell^{1-r}+\varsigma_{n+1}^{\alpha-1}).
\end{align*}

\end{proof}

\section{H\"older and Besov estimates}
\label{sec:schauder}
\begin{lem} \label{lem:C0_Cm}
Let $\phi=\phi_n$, $n \in \N$. 
For any $\delta \in [0,1)$ and $r \in \N$, $r \leq \kappa$ there exists a constant $C=C(\delta,r)$ such that the following holds almost surely for every $L \in \N$, $L \geq 1$. 
For every $f$ on $\mathbb{T}^3 \times (\infty,\mathfrak{t}_L]$ of class $C_{\leq \mathfrak{t}_L}C^{r+\delta}_x$ and every fixed $t \leq \mathfrak{t}_L$ 
\begin{align*}
\| f \circ \phi \|_{C^{r+\delta}_x} 
&\leq 
C L^{r+\delta}\|f\|_{C^{r+\delta}_x},
\\
\| f \circ \phi^{-1}\|_{C^{r+\delta}_x} &\leq 
C L^{r+\delta}\|f\|_{C^{r+\delta}_x},
\end{align*}
where $f \circ \phi$ denotes the map $\mathbb{T}^3 \times (\infty,\mathfrak{t}_L] \ni (x,t) \mapsto f(\phi(x,t),t)$, and similarly for $f \circ \phi^{-1}$.
\begin{proof}
Fix $t \leq \mathfrak{t}_L$.
For $r=0$, we have
\begin{align*}
[ f \circ \phi ]_{C^{\delta}_x}
&=
\sup_{x \neq y \in \mathbb{T}^3}
\frac{|f(\phi(x,t),t)-f(\phi(y,t),t)|}{|x-y|^\delta}
\\
&=
\sup_{x \neq y \in \mathbb{T}^3}
\frac{|f(\phi(x,t),t)-f(\phi(y,t),t)|}{|\phi (x,t)-\phi (y,t)|^\delta}
\frac{|\phi (x,t)-\phi (y,t)|^\delta}{|x-y|^\delta}
\leq C^\delta L^\delta \|f\|_{C^{\delta}_x}. 
\end{align*}
In the previous line we have used $|\phi (x,t)-\phi (y,t)| \leq \|\phi(\cdot,t)\|_{C^1_x} |x-y| \leq CL |x-y|$ since $t \leq \mathfrak{t}_L$, by the very definition of the stopping time $\mathfrak{t}_L$.

For $r>0$, we argue as follows. 
Let $\mathbf{k}$ be any multi-index of order $k\leq r$; we use Faà di Bruno formula and the previous decomposition to deduce 
\begin{align*}
[ D^\mathbf{k} (f \circ \phi) ]_{C^{\delta}_x}
&\leq 
\sum_{\substack{r_1,\dots,r_k\\|\mathbf{r}|=r_1+\dots+r_k}}
C_{r_i,k} 
\| (D^\mathbf{r}f) \circ \phi \|_{C^{\delta}_x}  
\prod_{j=1}^{k}  \|\phi\|_{C^j_x}^{r_j}
\\
&\leq
\sum_{\substack{r_1,\dots,r_k\\|\mathbf{r}|=r_1+\dots+r_k}}
C_{r_i,k} 
\| f\|_{C^{|\mathbf{r}|+\delta}_x} C^\delta L^\delta  
\prod_{j=1}^{k}  \|\phi\|_{C^j_x}^{r_j}
\\
&\leq
\sum_{\substack{r_1,\dots,r_k\\|\mathbf{r}|=r_1+\dots+r_k}}
C_{r_i,k}
\| f\|_{C^{|\mathbf{r}|+\delta}_x} C^{|\mathbf{r}|+\delta}L^{|\mathbf{r}|+\delta} ,
\end{align*}
where the sum is taken over all $r_i$ such that $r_1+2r_2+\dots+kr_k = k$ and multi-indices $\mathbf{r}$ with $|\mathbf{r}|=r_1+\dots+r_m$, and $C_{r_i,k}$ are constants depending only on $k$ and $r_1,\dots,r_k$. 
In the last inequality we have used $\|\phi\|_{C^j_x} \leq CL$ since $j \leq \kappa$ and $t \leq \mathfrak{t}_L$ is fixed.
We conclude using $|\mathbf{r}| \leq k \leq r$ and taking the sum over all multi-indices of order $\leq r$.
The second inequality is proved in a similar fashion.
\end{proof}
\end{lem}

\begin{lem} \label{lem:B-alpha}
Let $\phi=\phi_n$, $n \in \N$.
For every $\alpha \in (0,2)$, $\alpha \neq 1$ there exists $C=C(\alpha)$ such that for every $L \in \N$, $L\geq 1$ and every continuous function $f$ on $\T^3 \times (-\infty,\mathfrak{t}_L]$ it holds
\begin{align*}
\| f \circ \phi \|_{B^{-\alpha}_{\infty,\infty}}
&\leq
CL^{4-\alpha}\| f  \|_{B^{-\alpha}_{\infty,\infty}},
\\
\| f \circ \phi^{-1} \|_{B^{-\alpha}_{\infty,\infty}}
&\leq
CL^{4-\alpha}\| f  \|_{B^{-\alpha}_{\infty,\infty}}.
\end{align*}
\end{lem}
\begin{proof}
We only prove the second claim, the other being similar.
Let us preliminarily prove that for every smooth $g$ and $L\geq 1$ it holds
\begin{align} \label{eq:gB_1,1}
\| g \circ \phi \|_{B^{\alpha}_{1,1}}
&\leq
CL^{4-\alpha}\| g  \|_{B^{\alpha}_{1,1}}.
\end{align}

By continuity of the operator $(Id-\Delta)^{-1}:B^{\alpha-2}_{1,1}\to B^{\alpha}_{1,1}$ \cite[Proposition A.6]{MoWe17} it holds $\| g \circ \phi \|_{B^{\alpha}_{1,1}} \leq C \| (Id-\Delta) (g \circ \phi) \|_{B^{\alpha-2}_{1,1}}$. 
Since $g$ is smooth we have
\begin{align*}
(Id-\Delta) (g \circ \phi)
=
g \circ \phi
-
\sum_{j,k = 1}^d
\partial_{x_j} g \circ \phi \,\partial_{x_k}^2 \phi^j
-
\sum_{j,k,\ell = 1}^d
\partial_{x_j} \partial_{x_\ell} g \circ \phi\, \partial_{x_k}\phi^j \partial_{x_k} \phi^\ell
\end{align*}

By paraproduct estimates for Besov functions (see \cite[Lemma 2.2]{HoZhZh21c+} or \cite[Proposition A.7]{MoWe17}) and \autoref{lem:flow}
\begin{align*}
\| g \circ \phi \|_{B^{\alpha}_{1,1}} 
&\leq 
C \| (Id-\Delta) (g \circ \phi) \|_{B^{\alpha-2}_{1,1}}
\\
&\leq
C \| g \circ \phi \|_{B^{\alpha-2}_{1,1}}
+
CL \sum_{j=1}^d \| \partial_{x_j}g \circ \phi \|_{B^{\alpha-2}_{1,1}}
+
CL^2 \sum_{j,\ell=1}^d \| \partial_{x_j}\partial_{x_\ell}g \circ \phi \|_{B^{\alpha-2}_{1,1}}
\\
&\leq
CL^{2-\alpha} \| g \|_{B^{\alpha-2}_{1,1}}
+
CL^{3-\alpha} \sum_{j=1}^d \| \partial_{x_j}g \|_{B^{\alpha-2}_{1,1}}
+
CL^{4-\alpha} \sum_{j,\ell=1}^d \| \partial_{x_j}\partial_{x_\ell}g\|_{B^{\alpha-2}_{1,1}},
\end{align*}
where in the last line we have used the duality $B^{\alpha-2}_{1,1} = (B^{2-\alpha}_{\infty,\infty})^* = (C^\alpha_x)^*$, with equivalence of norms, the measure preserving property of $\phi^{-1}$, and previous \autoref{lem:C0_Cm}. 
By continuity of derivatives (Proposition A.5 in \cite{MoWe17}) we conclude the proof of \eqref{eq:gB_1,1}.

Let us now come back to the estimate for $f \circ \phi^{-1}$.
By duality, and since any function $g \in B^{\alpha}_{1,1}$ can be approximated by smooth functions in the $B^{\alpha}_{1,1}$ topology, there exists a constant $C$ such that 
\begin{align*}
\| f \circ \phi^{-1} \|_{B^{-\alpha}_{\infty,\infty}}
&=
\sup_{\substack{g \in B^{\alpha}_{1,1},\\ \|g\|_{B^{\alpha}_{1,1}}=1}}
\int_{\T^d} (f \circ \phi^{-1}) g dx
\leq
\sup_{\substack{g \in C^\infty,\\ \|g\|_{B^{\alpha}_{1,1}}\leq 2}}
\int_{\T^d} (f \circ \phi^{-1}) g dx.
\end{align*}

Thus, since $f$ is continuous, $\phi$ is measure preserving and by \eqref{eq:gB_1,1}
\begin{align*}
\sup_{\substack{g \in C^\infty,\\ \|g\|_{B^{\alpha}_{1,1}}\leq 2}}
\int_{\T^d} (f \circ \phi^{-1}) g dx
&=
\sup_{\substack{g \in C^\infty,\\ \|g\|_{B^{\alpha}_{1,1}}\leq 2}}
\int_{\T^d} f (g\circ \phi) dx
\\
&\leq
\sup_{\substack{h \in C^\infty,\\ \|h\|_{B^{\alpha}_{1,1}}\leq CL^{4-\alpha}}}
\int_{\T^d} f h dx
\leq
CL^{4-\alpha} \|f\|_{B^{-\alpha}_{\infty,\infty}}.
\end{align*}

\end{proof}

\begin{lem} \label{lem:Cbeta_C0}
Let $\phi=\phi_n$, $n \in \N$. 
For any $\beta \in (0,\alpha]$ there exists a constant $C=C(\beta,\alpha)$ such that almost surely for every $L \in \N$, $L \geq 1$, for every smooth function $f$ on $\mathbb{T}^3 \times (\infty,\mathfrak{t}_L]$ and given $t \leq \mathfrak{t}_L$ it holds 
\begin{align*}
\| f \circ \phi \|_{C^\beta_t C_x}
&\leq
CL^{\beta/\alpha} \|f\|_{C_t C^{\beta/\alpha}_x}
+
\|f\|_{C^\beta_t C_x},
\\
\| f \circ \phi^{-1} \|_{C^\beta_t C_x}
&\leq
CL^{\beta/\alpha} \|f\|_{C_t C^{\beta/\alpha}_x} 
+
\|f\|_{C^\beta_t C_x}.
\end{align*}

\end{lem}
\begin{proof}
Let us only prove the claim for $f \circ \phi$.
The $C_t C_x$ norm of $f \circ \phi$ is already controlled by the right-hand-side thanks to previous \autoref{lem:C0_Cm}, thus we only need a bound for the H\"older seminorm $[ f \circ \phi ]_{C^\beta_t C_x}$. 
We have
\begin{align*}
[ f \circ \phi ]_{C^\beta_t C_x}
&=
\sup_{t \neq s \leq \mathfrak{t}_L}
\sup_{x \in \mathbb{T}^3}
\frac{|f(\phi(x,t),t)-f(\phi(x,s),s)|}{|t-s|^\beta}
\\
&\leq
\sup_{t \neq s \leq \mathfrak{t}_L}
\sup_{x \in \mathbb{T}^3}
\frac{|f(\phi(x,t),t)-f(\phi(x,s),t)|}{|\phi(x,t)-\phi(x,s)|^{\beta/\alpha}} \frac{{|\phi(x,t)-\phi(x,s)|^{\beta/\alpha}}}{|t-s|^\beta}
\\
&\quad+
\sup_{t \neq s \leq \mathfrak{t}_L}
\sup_{x \in \mathbb{T}^3}
\frac{|f(\phi(x,s),t)-f(\phi(x,s),s)|}{|t-s|^{\alpha}}
\\
&\leq
C^{\beta/\alpha} L^{\beta/\alpha} \|f\|_{C_t C^{\beta/\alpha}_x}
+
\|f\|_{C^\beta_t C_x}.
\end{align*}
\end{proof}

\begin{lem}[Schauder estimates]\label{lem:schauder}
Let $\phi=\phi_n$, $n \in \N$. 
For any $\delta\in (0,1)$ and any $r\in \N$, $r + 2 \leq \kappa$, there exists a constant $C=C(\delta,r)$ with the following properties holding almost surely for every $L \in \N$, $L \geq 1$.
If $\psi, \Psi: \T^3 \times (\infty,\mathfrak{t}_L] \to \R$ are the unique zero-average solutions of
\begin{align*}
\Delta^\phi \psi = f,
\quad
\Delta^\phi \Psi = \dvgphi F,
\qquad
t\in (\infty,\mathfrak{t}_L],
\end{align*}
then for every fixed $t \leq \mathfrak{t}_L$ we have
\begin{align*}
\|\psi\|_{C^{r+2+\delta}_x} 
&\leq 
C L^{2r+2+2\delta} \|f\|_{C^{r+\delta}_x},
\\
\|\Psi\|_{C^{r+1+\delta}_x} 
&\leq 
C L^{2r+1+2\delta}\|F\|_{C^{r+\delta}_x}.
\end{align*}

Moreover we have the almost sure estimates for every $v:\T^3 \times (\infty,\mathfrak{t}_L] \to \R$, $A:\T^3 \times (\infty,\mathfrak{t}_L] \to \R^{3\times 3}$ and $t \leq \mathfrak{t}_L$:
\begin{align}
\|\mathcal{Q}^\phi v\|_{C^{r+\delta}_x} 
&\leq 
C L^{2r+2\delta}\|v\|_{C^{r+\delta}_x},\label{e:Schauder_Q}
\\
\|\mathcal{P}^\phi v\|_{C^{r+\delta}_x} 
&\leq 
C L^{2r+2\delta}\|v\|_{C^{r+\delta}_x},\label{e:Schauder_P}
\\
\|\mathcal{R}^\phi v\|_{C^{r+1+\delta}_x} 
&\leq 
C L^{2r+1+2\delta}\|v\|_{C^{r+\delta}_x},\label{e:Schauder_R}
\\
\|\mathcal{R}^\phi (\dvgphi A)\|_{C^{r+\delta}_x}
&\leq 
C L^{2r+2\delta} \|A\|_{C^{r+\delta}_x},\label{e:Schauder_Rdiv}
\\
\|\mathcal{R}^\phi \mathcal{Q}^\phi (\dvgphi A)\|_{C^{r+\delta}_x}
&\leq 
C L^{2r+2\delta} \|A\|_{C^{r+\delta}_x}.\label{e:Schauder_RQdiv}
\end{align}
Finally, for $\phi,v,L$ as above it holds
\begin{align} \label{e:Schauder_R_Besov}
\|\mathcal{R}^\phi v\|_{C^\delta_x} 
&\leq 
C L^{5+4\delta} \|v\|_{B^{\delta-1}_{\infty,\infty}}.
\end{align}
\end{lem}

\begin{proof}
By the very definition of the operators $\Delta^\phi$ and $\dvgphi$ we have $\Delta (\psi \circ \phi^{-1}) = f \circ \phi^{-1}$ and $\Delta (\Psi \circ \phi^{-1}) = \dvg (F \circ \phi^{-1})$ for every $t \leq \mathfrak{t}_L$; since $\phi$ and $\phi^{-1}$ are smooth in space, by \autoref{lem:C0_Cm} and the usual Schauder estimates for the Laplace problem we get
\begin{align*}
\|\psi\|_{C^{r+2+\delta}_x}
\lesssim
L^{r+2+\delta}\|\psi \circ \phi^{-1}\|_{C^{r+2+\delta}_x}
&\lesssim 
L^{r+2+\delta}\|f \circ \phi^{-1}\|_{C^{r+\delta}_x}
\lesssim 
L^{2r+2+2\delta}\|f\|_{C^{r+\delta}_x},
\\
\|\Psi \|_{C^{r+1+\delta}_x}
\lesssim
L^{r+1+\delta}\|\Psi \circ \phi^{-1}\|_{C^{r+1+\delta}_x} 
&\lesssim
L^{r+1+\delta}\|F \circ \phi^{-1}\|_{C^{r+\delta}_x}
\lesssim
L^{2r+1+2\delta}\|F\|_{C^{r+\delta}_x},
\end{align*}
where the implicit constant in the inequalities above depends on $r$ and $\delta$. 
Estimates \eqref{e:Schauder_Q} to \eqref{e:Schauder_RQdiv} can be deduced from the analogous esimates for $\mathcal{Q},\mathcal{P}$ and $\mathcal{R}$ contained in \cite[Proposition 4.3]{DLS14} with a similar argument.
Finally, \eqref{e:Schauder_R_Besov} can be derived by the explicit expression of $\mathcal{P}^\phi,\mathcal{R}^\phi$,  Schauder estimates in Besov spaces with negative exponent and \autoref{lem:B-alpha}; for instance (recall \autoref{note:Q})
\begin{align*}
\| \nabla^{\phi} \mathcal{P}^\phi u \|_{C^\delta_x}
\lesssim
L^{1+2\delta} \|u\|_{C^{1+\delta}_x}
\lesssim
L^{2+3\delta} \|u \circ \phi^{-1}\|_{C^{1+\delta}_x}
&\lesssim
L^{2+3\delta} \left\|\left( v-\int_{\T^3}v\right) \circ \phi^{-1}\right\|_{C^{\delta-1}_x}
\\
&\lesssim
L^{5+4\delta} \left\| v-\int_{\T^3}v\right\|_{C^{\delta-1}_x}
\lesssim
L^{5+4\delta} \left\| v\right\|_{C^{\delta-1}_x}.
\end{align*}
\end{proof}

The next proposition is analogous to \cite[Proposition 4.4]{DLS14}.
\begin{prop}[Stationary phase Lemma]\label{prop:stat_phase_lem}
Let $\phi=\phi_n$, $n \in \N$. 
Take a smooth function $a\in C^{\infty}(\T^3)$ and let $k\in\Z^3\setminus\{0\}$ and $\lambda\geq 1$ be fixed. 
Define $f(x) \coloneqq a(x)e^{i\lambda k \cdot x}$, $x \in \T^3$.

(i) For any $r\in\N$, we have almost surely for every $t\in \R$
\begin{equation}\label{e:average}
\left|\int_{\T^3}f(\phi(x,t))\,dx\right|\leq \frac{[a]_{C^r_x}}{\lambda^r}.
\end{equation}

(ii) For any $\delta\in(0,1)$ and $r\in\N$, $r + 1 \leq \kappa$, we have almost surely for every $L \in \N$, $L \geq 1$ and $t \in (\infty,\mathfrak{t}_L]$
\begin{align}
\label{eq:schauder1}
\|\mathcal{R}^\phi (f \circ \phi)\|_{C^\delta_x}
\leq CL^\delta
\left(
\lambda^{\delta-1}\|a\|_{C_x}
+
\lambda^{\delta-r}[a]_{C^r_x}
+
\lambda^{-r}[a]_{C^{r+\delta}_x}
\right),
\\
\label{eq:schauder2}
\|\mathcal{R}^\phi \mathcal{Q}^\phi (f \circ \phi)\|_{C^\delta_x}
\leq CL^\delta
\left(
\lambda^{\delta-1}\|a\|_{C_x}
+
\lambda^{\delta-r}[a]_{C^r_x}
+
\lambda^{-r}[a]_{C^{r+\delta}_x}
\right),
\end{align}
where $C=C(\delta,r)$.
\end{prop}
\begin{proof}
The proof follows immediately after reduction to \cite[Proposition 4.4]{DLS14}:
for the first part of the proposition, we use $\int_{\T^3}f(\phi(x,t))dx= \int_{\T^3}f(x)dx$ almost surely for every $t$ since $\phi$ is measure preserving;
for the second part, we use \autoref{lem:C0_Cm} and the estimates for $\|\mathcal{R} f\|_{C^\delta_x}$, $\|\mathcal{R}\mathcal{Q} f\|_{C^\delta_x}$ proved in \cite{DLS14}.
\end{proof}

\begin{lem} \label{lem:conv_Besov}
Let $\alpha \in \R$, $p_1,p_2 \in (1,\infty)$ such that $1/p_1+1/p_2 = 1$, then for every $\delta>0$ there exists $C$ such that for every $f \in B^{\alpha+\delta}_{p_1,\infty}(\T^3)$ with zero mean and $g \in L^{p_2}(\T^3)$ the convolution on the torus $f \ast_{\T^3} g$ belongs to $B^{\alpha}_{\infty,\infty}(\T^3)$ and
\begin{align*}
\| f \ast_{\T^3} g \|_{B^{\alpha}_{\infty,\infty}}
\leq
C
\| f \|_{B^{\alpha+\delta}_{p_1,\infty}} \|g\|_{L^{p_2}}.
\end{align*}
\end{lem}
\begin{proof}
First of all, notice that the convolution $f \ast_{\T^3} g$ is well defined also for $\alpha<0$ by duality, and has zero mean because $f$ does.
We recall (see Lemma 2.1 in \cite{HoZhZh21c+}) that for every $1 \leq p_1 \leq p_2 \leq \infty$ and $1 \leq q_1 \leq q_2 \leq \infty$ it holds $B^{\alpha}_{p_1,q_1} \subset B^{\alpha-3(1/p_1-1/p_2)}_{p_2,q_2}$ and for every $p \in (1,\infty)$ it holds $B^\alpha_{p,1} \subset W^{\alpha,p} \subset B^\alpha_{p,\infty} \subset B^{\alpha-\delta}_{p,1}$, with continuity of the inclusions.

Let $p \in (1,\infty)$. Then, $\|f \ast_{\T^3} g \|_{B^{\alpha}_{\infty,\infty}} \lesssim \|f \ast_{\T^3} g \|_{B^{\alpha+3/p}_{p,\infty}}$, and since $f \ast g$ is zero-mean
\begin{align*}
\|f \ast_{\T^3} g \|_{B^{\alpha+3/p}_{p,\infty}}
&\lesssim
\|(-\Delta)^{\alpha/2+3/(2p)}f \ast_{\T^3} g \|_{B^0_{p,\infty}}
\lesssim
\|(-\Delta)^{\alpha/2+3/(2p)}f \ast_{\T^3} g \|_{L^p}
\\
&\lesssim
\|(-\Delta)^{\alpha/2+3/(2p)}f \ast_{\T^3} g \|_{L^\infty}
\lesssim
\|(-\Delta)^{\alpha/2+3/(2p)}f\|_{L^{p_1}} \| g \|_{L^{p_2}},
\end{align*}
where in the last line we have used H\"older inequality on $\T^3$. Thus since $p_1 \in (1,\infty)$ and $f$ is zero mean
\begin{align*}
\|(-\Delta)^{\alpha/2+3/(2p)}f\|_{L^{p_1}}
\lesssim
\|f\|_{W^{\alpha+3/p,p_1}}
\lesssim
\|f\|_{B^{\alpha+3/p}_{p_1,1}}
\lesssim
\|f\|_{B^{\alpha+4/p}_{p_1,\infty}}.
\end{align*}
We conclude taking $p > 4/\delta$.
\end{proof}

\begin{lem} \label{lem:scaling}
Let $\ell=2^{-N}$ for some positive $N \in \N$.
Let $f$ be a smooth function on $\R^d$ with $\mbox{supp}(f) \subset (0,2\pi)^d$, and denote 
\begin{align*}
f^0 &\coloneqq f - (2\pi)^{-d}\int_{[0,2\pi]^d} f,
\\
f_\ell &\coloneqq \ell^{-d}f(\cdot/\ell),
\\
f^0_\ell &\coloneqq f_\ell - (2\pi)^{-d}\int_{[0,2\pi]^d} f_\ell
=
f_\ell - (2\pi)^{-d}\int_{[0,2\pi]^d} f.
\end{align*}
Extend $f$, $f^0$, $f_\ell$ and $f^0_\ell$ periodically on $\T^d$. 
Then for every $\alpha \in \R$, $p \in [1,\infty]$ it holds
\begin{align*}
\| f^0_\ell \|_{B^{\alpha}_{p,\infty}} 
=
\ell^{d/p-d-\alpha} \| f^0 \|_{B^\alpha_{p,\infty}}
\leq
\ell^{d/p-d-\alpha} \| f \|_{B^\alpha_{p,\infty}}.
\end{align*}
\end{lem}
\begin{proof}
For every $k \in \Z^d$ one can compute the Fourier coefficients (renormalizing constants omitted)
\begin{align} \label{eq:hat_f}
\hat{f^0_\ell}(k)
\coloneqq
\int_{[0,2\pi]^d}{f^0_\ell}(x)e^{-ik\cdot x} dx
=
\mathbf{1}_{\{k \neq 0\}} \hat{f_\ell}(k)
=
\mathbf{1}_{\{k \neq 0\}}
\mathbf{1}_{\{k \in 2^N\Z^d\}}
2^{Nd} \hat{f}(2^{-N}k),
\end{align}
where the last identity comes from
\begin{align*}
\int_{[0,2\pi]^d} f_\ell(x) e^{-ik\cdot x} dx
&=
2^{Nd} \int_{[0,2\pi]^d} f(2^N x) e^{-ik\cdot x} dx
&=
2^{Nd} \int_{[0,2\pi]^d} \sum_{h \in \Z^d} \hat{f}(h) e^{i (2^N h-k) \cdot x} dx.
\end{align*}

Let $\rho_j$, $j\geq -1$ a Littlewood-Paley partition of unity as that of Proposition 2.10 in \cite{BaChDa11}. We can choose the radial function $\varphi$ in that proposition with support in $[7/8,31/16]$.
Denote $K_j \coloneqq \mathscr{F}^{-1} \rho_j$ the inverse Fourier transform of $\rho_j$.
For $j \geq 0$ it holds $\int_{\R^d}{K_j}(x) dx = \rho_j(0)=0$ and 
\begin{align*}
\mathscr{F}^{-1} \rho_j(x)
=
\int_{\R^d} \rho_j(\xi) e^{i\xi \cdot x} d\xi
=
\int_{\R^d} \rho(2^{-j}\xi) e^{i\xi \cdot x} d\xi
=
2^{jd} \int_{\R^d} \rho_j(\zeta) e^{i 2^j \zeta \cdot x} d\zeta
=
2^{jd} \mathscr{F}^{-1} \rho(2^j x),
\end{align*}
and therefore for every $j \geq N$ we have
\begin{align*}
\Delta_j f^0_\ell
&=
K_j \ast f^0_\ell
=
\int_{\R^d} \mathscr{F}^{-1} \rho_j(y)f^0_\ell(\cdot-y) dy
\\
&=
\int_{\R^d} \mathscr{F}^{-1} \rho_j(y)f_\ell(\cdot-y) dy
\\
&=
\int_{\R^d}
2^{jd} \mathscr{F}^{-1} \rho(2^j y) 2^{Nd} f(2^N \cdot - 2^N y) dy
\\
&=
2^{Nd} \int_{\R^d}
2^{(j-N)d} \mathscr{F}^{-1} \rho(2^{j-N} z) f(2^N \cdot - z) dz
=2^{Nd} \Delta_{j-N} f(2^N \cdot).
\end{align*}
We have used that $K_j$ has zero average in the third line.
On the other hand, for $-1 \leq j \leq N-1$ we have $\Delta_j f^0_\ell=0$ since $\Delta_j f^0_\ell = \mathscr{F}^{-1} (\rho(2^{-j} \cdot ) \hat{f^0_\ell})) = \mathscr{F}^{-1} (\rho_j \hat{f^0_\ell}))$, and $\rho_j, \hat{f^0_\ell}$ have disjoint supports ($\rho_j$ has support in the annulus $2^{N-1}7/8 \leq |\xi| \leq 2^{N-1}31/16 < 2^N$ and $\hat{f^0_\ell}$ is non-zero only on $2^N \Z^d \setminus \{0\}$ by \eqref{eq:hat_f}).
Therefore
\begin{align*}
\| f^0_\ell \|_{B^\alpha_{p,\infty}}
&=
\sup_{j \geq -1} 2^{\alpha j} \|\Delta_j f^0_\ell\|_{L^p(\T^d)}
\\
&=
\sup_{j \geq N} 2^{\alpha j} \|\Delta_j f^0_\ell\|_{L^p(\T^d)}
\\
&= 2^{Nd}
\sup_{j \geq N} 2^{\alpha j} \|\Delta_{j-N} f(2^N \cdot)\|_{L^p(\T^d)}
\\
&=
2^{Nd(1-1/p)}
\sup_{j \geq N} 2^{\alpha j} \|\Delta_{j-N} f\|_{L^p(\T^d)}
\\
&=
2^{N\alpha+Nd(1-1/p)}
\sup_{j \geq N} 2^{\alpha (j-N)} \|\Delta_{j-N} f\|_{L^p(\T^d)}
\\
&=
2^{N\alpha+Nd(1-1/p)} \| f^0 \|_{B^\alpha_{p,\infty}}
\leq
2^{N\alpha+Nd(1-1/p)} \| f \|_{B^\alpha_{p,\infty}}.
\end{align*}
\end{proof}

%
%

\bibliographystyle{plain}

\begin{thebibliography}{10}

\bibitem{Ag22+}
Antonio Agresti.
\newblock Delayed blow-up and enhanced diffusion by transport noise for systems
  of reaction-diffusion equations.
\newblock arXiv:2207.08293, 2022.

\bibitem{BaChDa11}
H.~Bahouri, J-Y. Chemin, and R.~Danchin.
\newblock {\em Fourier Analysis and Nonlinear Partial Differential Equations}.
\newblock Grundlehren der mathematischen Wissenschaften, 343. Springer, 2011.

\bibitem{BrFlMa16}
Z.~Brzeźniak, F.~Flandoli, and M.~Maurelli.
\newblock Existence and uniqueness for stochastic 2{D} {E}uler flows with
  bounded vorticity.
\newblock {\em Arch. Rational Mech. Anal.}, 221:107--142, 2016.

\bibitem{BrMa19}
Zdzisław Brzeźniak and Mario Maurelli.
\newblock Existence for stochastic 2{D} {E}uler equations with positive
  ${H}^{-1}$ vorticity.
\newblock arXiv:1906.11523, 2019.

\bibitem{Bu15}
T.~Buckmaster.
\newblock {Onsager’s conjecture almost everywhere in time}.
\newblock {\em Commun. Math. Phys.}, 333(3):1175–1198, 2015.

\bibitem{BuDLSzVi19}
T.~Buckmaster, C.~De~Lellis, L.~Sz\'{e}kelyhidi, Jr., and V.~Vicol.
\newblock Onsager's conjecture for admissible weak solutions.
\newblock {\em Comm. Pure Appl. Math.}, 72(2):229--274, 2019.

\bibitem{BuDLIsSz15}
T.~Buckmaster, C.~De Lellis, P.~Isett, and L.~Székelyhidi Jr.
\newblock Anomalous dissipation for 1/5-{H}\"older {E}uler flows.
\newblock {\em Annals of Mathematics}, 182(1):127–172, 2015.

\bibitem{BuDLIsSz16}
Tristan Buckmaster, Camillo De~Lellis, and László Székelyhidi~Jr.
\newblock Dissipative {E}uler flows with {O}nsager-critical spatial regularity.
\newblock {\em Communications on Pure and Applied Mathematics},
  69(9):1613--1670, 2016.

\bibitem{ChDoZh22}
Weiquan Chen, Zhao Dong, and Xiangchan Zhu.
\newblock Sharp non-uniqueness of solutions to stochastic {N}avier-{S}tokes
  equations.
\newblock arXiv:2208.08321, 2022.

\bibitem{CoWeTi94}
P.~Constantin, E.~Weinan, and E.S. Titi.
\newblock {Onsager's conjecture on the energy conservation for solutions of
  Euler's equation}.
\newblock {\em Commun. Math. Phys.}, 165:207–209, 1994.

\bibitem{CoGoHo17}
C.~J. Cotter, G.~A. Gottwald, and D.~D. Holm.
\newblock Stochastic partial differential fluid equations as a diffusive limit
  of deterministic {L}agrangian multi-time dynamics.
\newblock {\em Proc. R. Soc. A.}, 473(2205):20170388, 2017.

\bibitem{CrDiFrOb13}
D.~Crisan, J.~Diehl, P.~K. Friz, and H.~Oberhauser.
\newblock {Robust filtering: Correlated noise and multidimensional
  observation}.
\newblock {\em The Annals of Applied Probability}, 23(5):2139 -- 2160, 2013.

\bibitem{CrFlHo19}
Dan Crisan, Franco Flandoli, and Darryl~D. Holm.
\newblock Solution properties of a 3{D} stochastic {E}uler fluid equation.
\newblock {\em J. Nonlinear Sci.}, 29(3):813--870, 2019.

\bibitem{DLS09}
C.~De~Lellis and L.~Sz{\'e}kelyhidi, Jr.
\newblock The {E}uler equations as a differential inclusion.
\newblock {\em Ann. of Math. (2)}, 170(3):1417--1436, 2009.

\bibitem{DLS13}
Camillo De~Lellis and László Székelyhidi.
\newblock Dissipative continuous euler flows.
\newblock {\em Invent. Math.}, 193:377--407, 2013.

\bibitem{DLS14}
Camillo De~Lellis and László Székelyhidi.
\newblock Dissipative {E}uler flows and {O}nsager's conjecture.
\newblock {\em J. Eur. Math. Soc. (JEMS)}, 016(7):1467--1505, 2014.

\bibitem{DeHoVo16}
Arnaud Debussche, Martina Hofmanová, and Julien Vovelle.
\newblock Degenerate parabolic stochastic partial differential equations:
  Quasilinear case.
\newblock {\em The Annals of Probability}, 44(3):1916--1955, 2016.

\bibitem{DePa22+}
Arnaud Debussche and Umberto Pappalettera.
\newblock Second order perturbation theory of two-scale systems in fluid
  dynamics.
\newblock arXiv:2206.07775, 2022.

\bibitem{DeFlVi14}
Fran\c{c}ois Delarue, Franco Flandoli, and Dario Vincenzi.
\newblock Noise prevents collapse of {V}lasov-{P}oisson point charges.
\newblock {\em Communications on Pure and Applied Mathematics},
  67(10):1700--1736, 2014.

\bibitem{DrHo20}
Theodore~D. Drivas and Darryl~D. Holm.
\newblock Circulation and energy theorem preserving stochastic fluids.
\newblock {\em Proceedings of the Royal Society of Edinburgh: Section A
  Mathematics}, 150(6):2776–2814, 2020.

\bibitem{FlGaLu21c}
F.~Flandoli, L.~Galeati, and D.~Luo.
\newblock Delayed blow-up by transport noise.
\newblock {\em Comm. Partial Differential Equations}, 46, 2021.

\bibitem{FlGuPr10}
F.~Flandoli, M.~Gubinelli, and E.~Priola.
\newblock Well-posedness of the transport equation by stochastic perturbation.
\newblock {\em Invent. math.}, 180:1--53, 2010.

\bibitem{FlGuPr11}
F.~Flandoli, M.~Gubinelli, and E.~Priola.
\newblock Full well-posedness of point vortex dynamics corresponding to
  stochastic 2{D} {E}uler equations.
\newblock {\em Stochastic Processes and their Applications}, 121(7):1445--1463,
  2011.

\bibitem{FlLu21}
Franco Flandoli and Dejun Luo.
\newblock High mode transport noise improves vorticity blow-up control in 3{D}
  {N}avier–{S}tokes equations.
\newblock {\em Probab. Theory Relat. Fields}, 180:309--363, 2021.

\bibitem{FlPa21}
Franco Flandoli and Umberto Pappalettera.
\newblock 2{D} {E}uler equations with {S}tratonovich transport noise as a
  large-scale stochastic model reduction.
\newblock {\em J. Nonlinear Sci.}, 31:24, 2021.

\bibitem{FlPa22}
Franco Flandoli and Umberto Pappalettera.
\newblock From additive to transport noise in 2{D} fluid dynamics.
\newblock {\em Stoch. PDE: Anal. Comp.}, 10:964--1004, 2022.

\bibitem{FrRi14}
Peter Friz and Sebastian Riedel.
\newblock {Convergence rates for the full Gaussian rough paths}.
\newblock {\em Annales de l'Institut Henri Poincaré, Probabilités et
  Statistiques}, 50(1):154 -- 194, 2014.

\bibitem{FrVi10}
Peter~K. Friz and Nicolas~B. Victoir.
\newblock {\em Multidimensional Stochastic Processes as Rough Paths: Theory and
  Applications}.
\newblock Cambridge Studies in Advanced Mathematics. Cambridge University
  Press, 2010.

\bibitem{FrHa20}
P.K. Friz and M.~Hairer.
\newblock {\em A Course on Rough Paths: With an Introduction to Regularity
  Structures}.
\newblock Universitext. Springer International Publishing, 2020.

\bibitem{GHVi14}
Nathan~E. Glatt-Holtz and Vlad~C. Vicol.
\newblock Local and global existence of smooth solutions for the stochastic
  {E}uler equations with multiplicative noise.
\newblock {\em The Annals of Probability}, 42(1):80--145, 2014.

\bibitem{GrPa22}
Francesco Grotto and Umberto Pappalettera.
\newblock Burst of point vortices and non-uniqueness of 2{D} {E}uler equations.
\newblock {\em Arch. Rational Mech. Anal.}, 245:89--125, 2022.

\bibitem{HoZhZh21c+}
Martina Hofmanová, Rongchan Zhu, and Xiangchan Zhu.
\newblock Global existence and non-uniqueness for 3{D} {N}avier–{S}tokes
  equations with space-time white noise.
\newblock {\em Arch. Rational Mech. Anal.}, 247(46):1-70, 2023.

\bibitem{HoZhZh21b+}
Martina Hofmanová, Rongchan Zhu, and Xiangchan Zhu.
\newblock Global-in-time probabilistically strong and {M}arkov solutions to
  stochastic 3{D} {N}avier–{S}tokes equations: existence and non-uniqueness.
\newblock {\em The Annals of Probability}, 51(2): 524-579, 2023.

\bibitem{HoZhZh21+}
Martina Hofmanová, Rongchan Zhu, and Xiangchan Zhu.
\newblock Non-uniqueness in law of stochastic 3{D} {N}avier-{S}tokes equations.
\newblock arXiv:1912.11841, 2021.
To appear in {\em J. Eur. Math. Soc. (JEMS)}.

\bibitem{HoZhZh22b+}
Martina Hofmanová, Rongchan Zhu, and Xiangchan Zhu.
\newblock Non-unique ergodicity for deterministic and stochastic {3D
  Navier--Stokes and Euler} equations.
\newblock arXiv:2208.08290, 2022.

\bibitem{HoZhZh22}
Martina Hofmanová, Rongchan Zhu, and Xiangchan Zhu.
\newblock On ill- and well-posedness of dissipative martingale solutions to
  stochastic 3{D E}uler equations.
\newblock {\em Communications on Pure and Applied Mathematics},
  75(11):2446--2510, 2022.

\bibitem{Ho15}
D.~D. Holm.
\newblock Variational principles for stochastic fluid dynamics.
\newblock {\em Proc. R. Soc. A.}, 471:20140963, 2015.

\bibitem{Is18}
P.~Isett.
\newblock A proof of {O}nsager's conjecture.
\newblock {\em Ann. of Math. (2)}, 188(3):871--963, 2018.

\bibitem{Ki09}
Jong~Uhn Kim.
\newblock Existence of a local smooth solution in probability to the stochastic
  {E}uler equations in {${\bf R}^3$}.
\newblock {\em J. Funct. Anal.}, 256(11):3660--3687, 2009.

\bibitem{Ku97}
H.~Kunita.
\newblock {\em Stochastic Flows and Stochastic Differential Equations}.
\newblock Cambridge Studies in Advanced Mathematics. Cambridge University
  Press, 1997.

\bibitem{LuZh22}
Huaxiang Lü and Xiangchan Zhu.
\newblock Global-in-time probabilistically strong solutions to stochastic
  power-law equations: existence and non-uniqueness.
\newblock arXiv:2209.02531, 2022.
To appear in {\em Stochastic Processes and their Applications}.

\bibitem{La22+}
Theresa Lange.
\newblock Regularization by noise of an averaged version of the
  {N}avier-{S}tokes equations.
\newblock {\em J. Dyn. Diff. Equat.}, https://doi.org/10.1007/s10884-023-10255-5, 2023.

\bibitem{Lu21++}
Dejun Luo.
\newblock Regularization by transport noise for {3D MHD} equations.
\newblock {\em Science China Mathematics}, 66:1375--1394, 2023.

\bibitem{BeMa02}
Andrew~J. Majda and Andrea~L. Bertozzi.
\newblock {\em Vorticity and incompressible flow}, volume~27 of {\em Cambridge
  Texts in Applied Mathematics}.
\newblock Cambridge University Press, Cambridge, 2002.

\bibitem{MoWe17}
J.C. Mourrat and H.~Weber.
\newblock The dynamic $\phi^4_3$ model comes down from infinity.
\newblock {\em Commun. Math. Phys.}, 356:673–753, 2017.

\bibitem{ReSc21+}
Marco Rehmeier and Andre Schenke.
\newblock Nonuniqueness in law for stochastic hypodissipative {N}avier-{S}tokes
  equations.
\newblock arXiv:2104.10798, 2021.

\bibitem{Sc93}
V.~Scheffer.
\newblock An inviscid flow with compact support in space-time.
\newblock {\em J. Geom. Anal.}, {\bf 3}(4):343--401, 1993.

\bibitem{Sh97}
A.~Shnirelman.
\newblock On the nonuniqueness of weak solution of the {E}uler equation.
\newblock {\em Comm. Pure Appl. Math.}, 50(12):1261--1286, 1997.

\bibitem{Wi11}
E.~Wiedemann.
\newblock Existence of weak solutions for the incompressible {E}uler equations.
\newblock {\em Ann. Inst. H. Poincar\'e Anal. Non Lin\'eaire}, {\bf
  28}(5):727--730, 2011.

\bibitem{Ya20}
Kazuo Yamazaki.
\newblock Non-uniqueness in law for two-dimensional {N}avier-{S}tokes equations
  with diffusion weaker than a full {L}aplacian.
\newblock arXiv:2008.04760, 2020.

\bibitem{Ya21}
Kazuo Yamazaki.
\newblock Non-uniqueness in law of three-dimensional {N}avier-{S}tokes
  equations diffused via a fractional {L}aplacian with power less than one
  half.
\newblock {\em Stoch PDE: Anal. Comp.}. https://doi.org/10.1007/s40072-023-00293-x, 2023.

\bibitem{Yu63}
V.~I. Yudovich.
\newblock Non-stationary flow of an ideal incompressible liquid.
\newblock {\em USSR Computational Mathematics and Mathematical Physics},
  3:1407--1456, 1963.

\end{thebibliography}

\end{document}